\documentclass[noinfoline]{imsart}

\RequirePackage[OT1]{fontenc}
\RequirePackage{amsthm,amsmath}
\RequirePackage[colorlinks,citecolor=blue,urlcolor=blue]{hyperref}
\usepackage[normalem]{ulem}

\usepackage[english]{babel}

\usepackage{graphicx}
\usepackage{amsfonts}
\usepackage{amssymb}    
\usepackage{vmargin}
\usepackage{color}
\usepackage{multirow}

\def\var{\mbox{Var}}

\def\1{{\mathbf 1}}
\def\N{{\mathbb{N}}}

\def\E{{\mathbb{E}}}

\def\Pin{{\boldsymbol{\Pi}}}

\newtheorem{thrm}{Theorem}[section]
\newtheorem{prop}[thrm]{Proposition}
\newtheorem{lemma}[thrm]{Lemma}
\newtheorem{cor}[thrm]{Corollary}
\newtheorem{defi}[thrm]{Definition}
\newtheorem{remark}{Remark}[section]

\def\K{\mathcal{K}}
\def\F{{\mathcal{F}}}
\def\H{\mathcal{H}}
\setmarginsrb{3.7cm}{2cm}{3.7cm}{3cm}{1cm}{1cm}{1cm}{1cm}
\begin{document}

\begin{frontmatter}
\title{Minimax adaptive tests for the Functional Linear model}
\runtitle{Functional tests}

\begin{aug}
  \author{\fnms{Nadine} \snm{Hilgert}\thanksref{t1} \ead[label=h]{nadine.hilgert@supagro.inra.fr}}, 
    \author{\fnms{Andr\'{e}} \snm{Mas}\thanksref{t2} \ead[label=m]{mas@math.univ-montp2.fr}}
 \and          
\author{Nicolas Verzelen\thanksref{t1} \ead[label=v]{nicolas.verzelen@supagro.inra.fr}}

\begin{keyword}[class=AMS] 
\kwd[Primary ]{62J05}  
\kwd[; secondary ]{62G10}
\end{keyword} 
 
\begin{keyword} 
\kwd{Functional linear regression} 
\kwd{eigenfunction}
\kwd{principal component analysis}
\kwd{adaptive testing}
\kwd{minimax hypothesis testing}
\kwd{minimax separation rate}
\kwd{multiple testing}
\kwd{ellipsoid}
\kwd{goodness-of-fit}
\end{keyword}

\affiliation{INRA and Universit\'e Montpellier 2}

\thankstext{t1}{
 INRA, UMR  729 MISTEA,
F-34060 Montpellier, FRANCE.  \newline \printead{h,v}}
 
\thankstext{t2}{
Universit\'{e} Montpellier 2, F-34000 Montpellier, FRANCE. \printead{m}}

  \runauthor{Hilgert et al.}

\end{aug}
%

\begin{abstract}. We introduce two novel procedures to test the nullity of the slope function in the functional linear model with real output. The test statistics combine multiple testing ideas and random projections of the input data through functional Principal Component Analysis. Interestingly, the procedures are completely data-driven and do not require any prior knowledge on the smoothness of the slope nor on the smoothness of the covariate functions. The levels and powers against local alternatives  are assessed in a nonasymptotic setting. This allows us to prove that these procedures are minimax adaptive (up to an unavoidable $\log\log n$ multiplicative term) to the unknown regularity of the slope. As a side result, the minimax separation distances of the slope are derived for a large range of regularity classes. A numerical study illustrates these 
theoretical results.
 \end{abstract}

\end{frontmatter}

\section{Introduction}

Consider the following functional linear regression model where the scalar response $Y$ is related to a square integrable random function $X(.)$ through
\begin{equation}\label{modelinitial}
 Y=\omega+ \int_{\mathcal{T}} X(t)\theta(t)dt+\epsilon\ .
\end{equation}
Here, $\omega$ is a constant, denoting the intercept of the model, $\mathcal{T}$ is the domain of $X(.)$, $\theta(.)$ is an unknown function representing the slope function, and $\epsilon$ is a centered random noise variable. 
In functional linear regression, much interest focuses on the nonparametric estimation of $\theta(.)$ in (\ref{modelinitial}), given an {\it i.i.d.} sample $\left({X}_{i},  {Y}_{i},\right)  _{1\leq i\leq n}$ of $(X,Y)$.
Testing whether $\theta$ belongs to a given finite dimensional linear subspace $\cal{V}$ {is a 
question} that arises in different problems such as dimension reduction, goodness-of-fit analysis, or lack-of-effect tests of a functional variable. If the properties of estimators of $\theta$ are widely discussed in the literature, there is still a great need to have generic test procedures supported by strong theoretical properties. This is the problem addressed in the present paper.

Let us reformulate the functional  model (\ref{modelinitial}) as a generic linear regression model in an infinite dimensional space.
 The random 
function  $X$ is assumed to belong to some separable Hilbert space
henceforth denoted $\mathcal{H}$ endowed with
the inner product $\left\langle .,.\right\rangle $. Examples of $\H$ include  $\mathcal{L}^{2}(  [  0,1]  )  $ or
Sobolev space $\mathcal{W}_{2}^{m}([0,1])$. 
For the sake of clarity, we consider that $\omega=0$ and that $X$ and $Y$ are centered. Thus, assuming that  $\theta$ also belongs to $\mathcal{H}$,  the statistical model (\ref{modelinitial}) is rephrased as  
\begin{eqnarray}
 \label{modelLFRM}
Y = \langle X,\theta\rangle  + \epsilon\ ,
\end{eqnarray}
where $\epsilon$ is a centered random variable independent from $X$ with unknown variance $\sigma^2$.
In the sequel, we note $\bf X$ and $\bf Y$ the size $n$ vectors of i.i.d. observations $X_i$ and $Y_i$ ($1\leq i\leq n $), while $\boldsymbol{\epsilon}$ stands for the size $n$ vector of the noise. 

In essence, testing a linear hypothesis of the form ``$\theta\in \cal{V}$'' is as difficult as testing ``$\theta=0$'' when a parametric estimator of $\theta$ in $\cal{V}$ is computed. Therefore we consider 
the problem of testing:
\[H_0:\ \text{``}\theta=0\text{''} \quad\quad \text{against}\quad \quad H_1:\ ``\theta\neq 0\text{''} \]
given an i.i.d. sample $(\bf X, \bf Y)$  from model (\ref{modelLFRM}).
The extension to general subspaces $\cal{V}$ is developed in the discussion section. \\

Most testing procedures are based on ideas that have been originally developed for the estimation of $\theta$. 
We briefly review the main approaches and the corresponding results in estimation. 

A first class of procedures is based on the minimization of a least-square type criterion penalized by a roughness term that assesses the ``plausibility'' of $\theta$. Such approaches include smoothing spline estimators~\cite{CFS03,CKS}, thresholding projection estimators~\cite{CaJ10}, or reproducing kernel Hilbert space methods~\cite{yuan2010rkhs}. A second class of procedures is based  on the functional principal components analysis (PCA) of ${\bf X}$~\cite{CLT,HH07}. It consists in estimating $\theta$ in a finite dimensional space spanned by the $k$ first eigenfunctions of the empirical covariance operator of ${\bf X}$. The main difference with the previous class of estimators lies in the fact that the finite dimensional space is estimated from the observations of the process $X$. See the survey~\cite{CS10} and references therein for an overview of these two approaches.

The theoretical properties of these classes of estimators have been investigated from different viewpoints: prediction~\cite{CFS03,CLT,CKS,yuan2010rkhs} (estimation of $\langle X_{n+1},\theta\rangle$ where $X_{n+1}$ follows the same distribution as $X$), pointwise prediction~\cite{cai-hall2006} (estimation of $\langle x,\theta\rangle$ for a fixed $x\in\H$) or the inverse problem~\cite{CoJ10,HH07} (estimation of $\theta$). For these three objectives, optimal rates of convergence have been derived and some of the aforementioned procedures have been shown to asymptotically achieve this rate~\cite{cai-hall2006,CKS,yuan2010rkhs,HH07}. Recently, some non-asymptotic results have emerged~\cite{CoJ10,CoJ11} for estimation procedures that rely on a prescribed basis of functions (e.g. splines). Most of these estimation  procedures rely on tuning parameters whose optimal value depend on quantities such as the noise variance, or the smoothness of $\theta$. In fact, there is a longstanding gap in the literature between 
theory, where the variance $\sigma^2$, the smoothness of $\theta$ and the smoothness of the covariance operator of $X$ are generally assumed to be known, and practice where they are unknown.

The literature on tests in the functional linear model is scarce.  In~\cite{cardot_test}, Cardot et al. introduced a test statistic based on the $k$ first components of the functional PCA of ${\bf X}$. Its limiting distribution is derived under $H_0$ and  the power of the corresponding test is proved to converge to one under $H_1$.  The main drawback of the procedure is that the number $k$ of components involved in the statistic has to be set. 
As for estimation, setting $k$ is arguably a difficult problem. 
To bypass this calibration issue, one may apply a permutation approach~\cite{CGS} or use bootstrap methodologies~\cite{CuFra,GMMC}.
While the levels of the corresponding tests are asymptotically controlled, there is again no theoretical guarantee on the power. \\

In this paper, our objective is to introduce automatic testing procedures whose powers are optimal from a nonasymptotic viewpoint. 

As a first step, we  introduce in Section \ref{section_param_test} 
Fisher-type non-adaptive tests, $T_{\alpha,k}$, corresponding to projections of ${\bf Y}$ on the $k$ first principal components of ${\bf X}$. 
We study their levels and powers in Sections \ref{section_param_test} and \ref{secion_power_minim.param_test}.
Under moment assumptions on $\epsilon$ and mild assumptions on the covariance of $X$, the level is smaller than $\alpha$ up to a $\log^{-1}(n)$ additional term, and a sharp control of the power is provided. Such results are comparable to state of the art results in nonparametric regression~\cite{baraud03,spokoiny96}. In our setting, the main difficulty in the proof is to control the randomness of the principal components of ${\bf X}$. The arguments rely on the perturbation theory of operators. While other estimation or testing procedures based on the Karhunen-Lo\`eve expansion have only been analyzed in an asymptotic setting~\cite{cai-hall2006,cardot_test,HH07}, our nonasymptotic results rely on less restrictive assumptions on  $X$ than those commonly used in the literature.
 In Section \ref{secion_power_minim.param_test}, we assess the optimality of the parametric test $T_{\alpha,k}$ in the minimax sense.
The notion of minimaxity of a level-$\alpha$ test $T_{\alpha}$ is related to the {\it separation distance}  of $T_{\alpha}$ over some  class of functions $\Theta$ (e.g. a Sobolev ball). Intuitively, the power of a reasonable test $T_{\alpha}$ should be large when the norm of $\theta$ is large while the power of $T_{\alpha}$ is close to $\alpha$ when $\theta$ is close to $0$. For the problem of testing $H_0$: ``${\theta}=0$'' against $H_{1,\Theta}$: ``$\theta\in\Theta\setminus\{0\}$'', the separation distance corresponds to the smallest distance $\rho$ such that $T_{\alpha}$ rejects $H_0$ with probability larger than $1-\beta$ for all $\theta\in\Theta$ whose norm is larger than $\rho$. The smaller the separation distance, the more powerful the test $T_{\alpha}$ is. The minimax separation distance over $\Theta$ is the smallest separation distance that is achieved by a level-$\alpha$ test. A test achieving this minimax separation distance is said to be minimax over $\Theta$. 
and minimax separation distances are formalized  in Section \ref{section_minimax_KL}. 
In the nonparametric regression setting, minimax separation distances have been derived in an asymptotic~\cite{Ingster93a,Ingster93b,Ingster93c} and a nonasymptotic~\cite{baraudminimax} setting. 
In this paper,  the separation distances of our testing procedures are nonasymptotically controlled. We derive minimax separation distance in the functional model (\ref{modelLFRM}) for a wide class of ellipsoids. 
We show that the parametric test $T_{\alpha,k}$ achieves the optimal rate of detection when the dimension $k$ is suitably chosen.\\ 

In practice, the regularity of $\theta$ is unknown.  However, the choice of $k$ in $T_{\alpha,k}$ depends on unknown quantities such as the regularity of $X$ or the regularity of $\theta$. Thus, assuming a priori that the function $\theta$ belongs to a particular smoothness class $\Theta$ and building an optimal test over $\Theta$ may lead to poor performances, for instance  if $\theta \notin \Theta$. For this reason, a more ambitious issue is to build a minimax adaptive testing procedure, that is a procedure which is simultaneously minimax for a wide range of regularity classes $\Theta$. Minimax adaptive testing procedures have already been studied in the nonparametric regression setting, from an asymptotic~\cite{spokoiny96} and a nonasymptotic~\cite{baraud03} viewpoint. As a second step, we combine the parametric tests $T_{\alpha,k}$ with multiple testing techniques in the spirit of \cite{baraud03}. 
 Two such multiple testing procedures are introduced in Section \ref{mult.test.proc}. They are completely data-driven: no tuning parameters are required, whose optimal values depend on $\theta$, the distribution of $X$ or on $\sigma$. Their levels and  powers are analyzed from a nonasymptotic viewpoint in Sections \ref{mult.test.proc} and \ref{section_rates_mult.test.proc}. We prove that our mulitiple testing procedures are  simultaneously minimax over the class of ellipsoids aforementioned(up to an unavoidable $\log\log n$ factor). As in the estimation setting~\cite{HH07}, the minimax separation distances involve the common regularity of $\theta$ and $X$.

\medskip

The two multiple testing procedures are illustrated and compared by simulations in Section \ref{section_simulation}. Extensions of the approach are discussed in Section \ref{section_discussion}. Section \ref{section_proofs} contains the main proofs while the lemmas involving perturbation theory are given in Section \ref{section_perturb}. All the technical and side results are postponed to appendices.

\section{Preliminaries }\label{section_preliminaries}

\subsection{Notations}
We remind that $\langle . ,.\rangle $ and $\|.\| $ respectively refer to the inner product and the corresponding norm in the Hilbert $\H$. In contrast, $\langle . ,.\rangle_{n}$ and $\|.\|_{n}$ stand for the inner product and the Euclidean norm in $\mathbb{R}^n$. Furthermore, $\otimes$ refers to the tensor product. We assume henceforth that $X$ is centered and has a second moment that is $\mathbb{E}[\|{\bf X}\| ^2]<\infty$. The covariance operator of $X$ is defined as the linear operator $\Gamma$ defined on $\H$ as follows:
\[\Gamma h= \mathbb{E}[X\otimes X h]= \mathbb{E}[\langle h,X\rangle X]\ , \quad h\in \H\ .\]  It is well  known that $\Gamma$ is a symmetric, positive trace-class hence Hilbert-Schmidt operator, which implies that $\Gamma$ is diagonalizable in an orthonormal basis.  We denote $(\lambda_j)_{j\geq 1}$ the non-increasing sequence of eigenvalues of $\Gamma$, while the sequence $(V_j)_{j\geq 1}$ stands for a corresponding sequence of eigenfunctions. It follows that $\Gamma$ decomposes as  
$\Gamma=\sum_{j=1}^{\infty}\lambda_j V_j\otimes V_j$. For any integer $k\geq 1$, we note $\Gamma_k= \sum_{j=1}^k \lambda_j V_j\otimes V_j$ the  operator such that $\Gamma_kh=\Gamma h$ for $h\in\mathrm{Vect}(V_1,\ldots,V_k)$ and $\Gamma_k h=0$ if $h\in \ (V_1,\ldots,V_k)^\perp$.

In the sequel, $C$, $C_1$,$\ldots$ denote positive universal constants that may vary
from line to line. The notation $C(.)$ specifies
the dependency on some quantities.

\subsection{Karhunen-Lo\`{e}ve expansion and functional PCA}

We recall here a  classical tool of functional data analysis :
the {\bf Karhunen-Lo\`{e}ve expansion}, denoted KL expansion in the sequel. 
\begin{defi}
There exists an expansion of  $X$ in the basis $(V_j)_{j\geq 1}$: 
$X=\sum\left\langle X,V_{j}%
\right\rangle  V_{j}$.
The real random variables $\left\langle X,V_{j}\right\rangle $ are centered (when $X$ is centered), uncorrelated, and with variance $\lambda_{j}$. As a consequence, there exists a collection $(\eta^{(j)})_{j\geq 1}$ of random variables that are centered, uncorrelated, and with unit variance such that
\begin{equation}
X=\sum_{j=1}^{+\infty}\sqrt{\lambda_{j}}\eta^{(j)}V_{j}\ .\label{KL-exp}%
\end{equation}
The decomposition is called the KL-expansion of $X$.
\end{defi}
The eigenfunction $V_{j}$ is the $j$-th principal direction whose amount of variance coincides with $\lambda_{j}$.
When $X$ is a Gaussian process, the $\left(  \eta^{(j)}\right)  _{j\in\mathbb{N}}$ form an i.i.d
sequence with $\eta^{(1)}\sim\mathcal{N}\left(  0,1\right) $. 
If the eigenfunctions $(V_{j})$ and the eigenvalues $(\lambda_j)$ are unknown in practice, they can be estimated from the data using functional principal component analysis. In the sequel, we note $\widehat{\Gamma}_n$  the empirical covariance operator defined by 
\[\widehat{\Gamma}_nh=\frac{1}{n}\sum_{i=1}^n  X_i\otimes  X_i h= \frac{1}{n}\sum_{i=1}^n \langle  X_i,h\rangle   X_i \ ,\quad h\in\H\ . \]
Functional PCA allows to estimate $\left(\lambda_{j},V_{j}\right)$, $j\geq 1$ by diagonalizing the empirical covariance operator $\widehat{\Gamma}_n$. These empirical counterparts of $(\lambda_{j},V_{j})$ are denoted $(\widehat{\lambda}_{j},\widehat{V_{j}})$ in the sequel.

Functional PCA is usually applied as a dimension reduction technique. One of its appealing features relies on its ability to capture most of the variance of $X$ by a $k$-dimensional projection on the space $\mathrm{Vect}(\widehat{V}_1,\ldots, \widehat{V}_k)$. For this reason, PCA is at the core of many procedures for functional data. After the seminal paper by Dauxois et al.~\cite{DPR}, the convergence of the random eigenelements $(\widehat{\lambda}_{j},\widehat{V_{j}})$ has been assessed from an asymptotic point of view~\cite{HHN06,  HHN09, HVial06,MaMe03}. One issue with such a dimension reduction method is the choice of the tuning parameter $k$, whose optimal value usually depends on unknown quantities. Besides plugging the $(\widehat{\lambda}_{j},\widehat{V_{j}})$ into linear estimates creates non-linearity and usually introduces stochastic dependence.


\section{Parametric test}
\label{section_param_test}

\subsection{Definition}
\label{param.stat}

In the sequel, $k$ denotes a positive integer smaller than $n/2$.
As a first step, we consider the parametric testing problem of
the hypotheses: 
\begin{equation}\label{eq:definition_hypothesis_parametric}
H_0:\
``\theta=0"
\quad\text{against}\quad H_{1,k}:\
``\theta\in\text{Vect}[(V_{j})_{j=1,\ldots,k}]\setminus\{0\}"\
.
\end{equation}
Given a dimension $k$ of the Karhunen-Lo\`eve expansion, we note $\hat{k}^{KL}$ as $k\wedge \mathrm{Rank}(\widehat{\Gamma}_n)$. In order to introduce the parametric statistic, let us restate the functional linear model into a finite dimensional linear model. We consider the response vector ${\bf Y}$ of size $n$, the $n\times \hat{k}^{KL}$ design matrix ${\bf W}$ defined by ${\bf W}_{i,j}= \langle X_i,\widehat{V}_j\rangle$ for  $i=1,\ldots n$, $j=1,\ldots \hat{k}^{KL}$, the parameter  vector $\vartheta$ defined by $\vartheta_j=\langle \theta,\widehat{V}_j\rangle$, $j=1,\ldots \hat{k}^{KL}$, and the size $n$ noise vector  $\tilde{\boldsymbol{\epsilon}}$ defined by $\tilde{\epsilon}_i=
\epsilon_i+ \langle X_i,\theta\rangle - [{\bf W}\vartheta]_i$. The functional linear model is equivalently written as
\begin{eqnarray*}
 {\bf Y} = {\bf W} \vartheta + \tilde{\boldsymbol{\epsilon}}\ . 
\end{eqnarray*}
Intuitively,  testing ``$\vartheta=0$" is a reasonable proxy  for testing $H_0$ against $H_{1,k}$. For this reason, we propose a Fisher-type statistic.

\begin{defi} 
In the sequel, $\widehat{\Pin}_k$ stands for the orthogonal projection in $\mathbb{R}^n$ onto the
space generated by the $\hat{k}^{KL}$ columns of ${\bf W}$.
For any $k\leq n/2$, we consider the statistic $\phi_{k}({\bf Y},{\bf X})$ defined by
\begin{eqnarray}\label{definition_phi_KL}
 \phi_{k}({\bf Y},{\bf X}):=
\frac{\|\widehat{\Pin}_k {\bf
Y}\|^2_n}{\|{\bf Y}-\widehat{\Pin}_k {\bf Y}\|_n^2/(n-\hat{k}^{KL})}\  .
\end{eqnarray}
\end{defi}

The main difference with a classical Fisher statistic comes from the fact that the projection $\widehat{\Pin}_k$ is {\it random}. This projector is built using the $\hat{k}^{KL}$ first  directions $(\widehat{V}_1,\widehat{V_2},\ldots, \widehat{V}_{\hat{k}^{KL}})$ of the
empirical Karhunen-Lo\`eve expansion of ${\bf X}$. Let us call $\widetilde{\bf \Pi}_k$ the orthogonal projector in $\mathbb{R}^n$ onto the space spanned by $(\langle X_i,V_j\rangle )_{i=1,\ldots n}$, $j=1,\ldots, k$. If we knew the basis $(V_j)$, $j\geq 1$ in advance, we would  use this orthogonal projector instead of $\widehat{\bf \Pi}_k$. We shall prove that, under $H_0$, $\phi_{k}({\bf Y},{\bf X})/\hat{k}^{KL}$ behaves like a Fisher distribution with $(\hat{k}^{KL},n-\hat{k}^{KL})$ degrees of freedom.
\begin{defi}[Parametric tests] Fix $\alpha\in (0,1)$
We reject  $H_0$ against $H_{1,k}$ when the statistic
\begin{eqnarray}\label{test_parametrique_T.alpha.k}
 T_{\alpha,k}:=\phi_{k}({\bf Y},{\bf X})- \hat{k}^{KL}\bar{\mathcal{F}}^{-1}_{\hat{k}^{KL},n-\hat{k}^{KL}}(\alpha)\ .
\end{eqnarray}
is positive.
\end{defi}

\begin{remark}[Other interpretations of $\phi_{k}({\bf Y},{\bf X})$]
Consider $\widehat{\theta}_k$ the least-squares estimator of $\theta$ in the space generated by $\widehat{V}_j$, $j=1,\ldots, \hat{k}^{KL}$.
It is proved in Section \ref{section_proof_lien_cardot} that $\|\widehat{\Pin}_k{\bf Y}\|_n^2=\|\widehat{\Gamma}^{1/2}_n\widehat{\theta}_k\|^2$. Thus, the numerator of (\ref{definition_phi_KL}) corresponds to some norm of $\widehat{\theta}_k$. Intuitively, the larger $\widehat{\theta}_k$, the larger the statistic $\phi_{k}({\bf Y},{\bf X})$ is. Furthermore, $\|\widehat{\Pin}_k{\bf Y}\|_n^2$ also expresses as the numerator of the statistic $D_n$ considered in Cardot et al.~\cite{cardot_test} (see Section \ref{section_proof_lien_cardot} for details).

\end{remark}
\begin{remark}From the considerations above, we see that the transformed parameter $\Gamma^{1/2} \theta$ naturally occurs in the definition of $\phi_{k}({\bf Y},{\bf X})$. In fact, hypotheses $H_{0}$ and $H_{1,k}$ remain unchanged, if we replace $\theta$ by $\Gamma^{1/2} \theta$ in \eqref{eq:definition_hypothesis_parametric} as soon as $\Gamma$ is injective. The crucial role of this synthetic parameter is underlined in \cite{Meister11} where the functional linear regression model is proved to be asymptotically equivalent to a white noise model with signal $\Gamma^{1/2} \theta$.
\end{remark}

\subsection{Size}\label{section_level} 

We study the type I error of the parametric tests $T_{\alpha,k}$.  On one hand, we  control exactly the size of the tests when the noise $\epsilon$ is normally distributed.  On the other hand, we  bound the size of the tests when the noise is only constrained to admit a fourth moment. 

\subsubsection{Gaussian noise}

\begin{equation*}
{\bf A.1}\hspace{3cm} \epsilon\text{ follows a Gaussian distribution }
\mathcal{N}(0,\sigma^2)\ .  \hspace{2.5cm}
\end{equation*}

\begin{prop}[Size of $T_{\alpha,k} $ under Gaussian errors]\label{prop_niveau_gaussien_T.alpha.k}
Under Assumption ${\bf A.1}$ and if $k\leq n/2$, we have for any $n\geq 2$, $\mathbb{P}_{0}(T_{\alpha,k} > 0) =  \alpha$.
\end{prop}

Observe that this control does not require any assumption on the process $X$.

\subsubsection{Non-Gaussian noise}

In this part, the noise $\epsilon$ is only assumed to admit a fourth order moment, but we perform additional assumptions on $X$ and $k$.
\begin{eqnarray*}
{\bf B.1}\hspace{3.5cm} \sup_{j\geq 1}\E[(\eta^{(j)})^4]\leq C_1\text{  and  }\frac{\mathbb{E}\left[\epsilon^4\right]}{\sigma^4}\leq C_2\ , \hspace{2.5cm} 
\end{eqnarray*}
where $C_1$ and $C_2$ are two positive constants.
\begin{eqnarray*}
 {\bf B.2}\hspace{0.4cm}\text{For some $\gamma>0$, }\left(j\lambda_j((\log^{1+\gamma}j)\vee 1)\right)_{j\geq 1}\text{ is decreasing and } \mathrm{Ker}\Gamma=\left\{0\right\}.\ 
\end{eqnarray*}
\begin{eqnarray*}
 {\bf B.3}\hspace{4.4cm} k\leq n^{1/4}/\log^{4}(n)\ .\hspace{4.4cm} 
\end{eqnarray*}
Assumption ${\bf B.1}$ is classical, since we need to control second order moments for the empirical
covariance operator $\widehat{\Gamma}_{n}$. This comes down to inspecting the behavior of
the fourth order moments of the $\eta^{(j)}$'s. The second part of ${\bf B.2}$ ensures  that the framework is truly functional. 
The first part of ${\bf B.2}$ is mild and holds for an $X$ that may have very irregular paths (it holds for the Brownian motion for which $\lambda_j \propto j^{-2}$) and for classical examples of eigenvalue sequences: with polynomial decay, exponential decay, or even Laurent sequences such as $\lambda_j=j^{-\delta}\cdot \log^{-\nu}\left(j\right)$ for $\delta>1$ and $\nu\geq0$. In fact, ${\bf B.2}$ is less restrictive than assumptions commonly used in the literature~\cite{cai-hall2006,cardot_test,HH07} since it does not require any spacing control between the eigenvalues.

The  restriction ${\bf B.3}$ on the dimension of the projection is classical for the analysis of statistical procedures based on the Karhunen-Lo\`eve expansion. 
If we knew the eigenfunctions $V_k$ of $\Gamma$ in advance, we could consider larger dimensions $k$. The estimation of the eigenfunctions $V_k$ becomes more difficult when $k$ increases. By considering dimensions $k$ that satisfy Assumption ${\bf B.3}$, we prove in the next theorem that the {\it random} projector $\widehat{{\bf \Pi}}_k$ concentrates well around its mean. It may be noticed that this assumption links $k$ and $n$ independently from the eigenvalues hence from any prior knowledge on the data.

\begin{thrm}[Size of $T_{\alpha,k}$]\label{thrm_niveau_T.alpha.k}
Under Assumptions ${\bf B.1-3}$, 
there exist positive constants $C(\alpha,\gamma)$ and $C_2$  such that the following holds.
For any $n\geq C_2$, we have
\begin{eqnarray*}
\mathbb{P}_{0}\left[T_{\alpha,k}> 0\right]\leq  \alpha+ \frac{C(\alpha,\gamma)}{\log(n)}\ .
\end{eqnarray*}
\end{thrm}

\begin{remark}
In the proof of Theorem \ref{thrm_niveau_T.alpha.k}, we show that, under $H_0$, the distribution of  $\phi_{k}({\bf Y},{\bf X})$  is close to a $\chi^2$  distribution with $k$ degrees of freedom. The arguments rely on perturbation theory for random operators (see Section \ref{section_perturb}). 
\end{remark}

\section{Power and minimaxity of $T_{\alpha,k} $}
\label{secion_power_minim.param_test}

Intuitively, the larger the signal-to-noise ratio $\mathbb{E}\left[\langle X,\theta\rangle ^2\right]/\sigma^2=\|\Gamma^{1/2}\theta\| ^2/\sigma^2$ is, the easier we can reject $H_0$. For this reason, we study  how large $\|\Gamma^{1/2}\theta\| ^2$ has to be, so that the test $T_{\alpha,k}$ rejects $H_0$ with probability larger than $1-\beta$ for a prescribed positive number $\beta$. We provide such type II errors under moment assumption of $\epsilon$. Additional controls of the power when $\epsilon$ follows a Normal distribution are stated in Appendix \ref{sec:power:gaussian}.

\subsection{Power of $T_{\alpha,k}$}

\begin{equation*}
 {\bf B.4}\hspace{4cm} \sup_{j\geq 1}\E[(\eta^{(j)})^8]\leq C\ .\hspace{4cm} 
\end{equation*}

\begin{thrm}[Power under non-Gaussian errors]\label{thrm_power_KL_fixed}
Let $\alpha$ and $\beta$ be fixed. Under ${\bf B.1-4}$, 
there exist positive constants $C(\gamma)$, $C_1$, $C_2$, and $C_3$ such that the following holds. Assume that $\alpha\geq e^{-\sqrt{n}}$,   $\beta\geq C(\gamma)/\log(n)$,  and that $n\geq C_3$. Then, $\mathbb{P}_{\theta}(T_{\alpha,k}>0)\geq 1-\beta$ for any $\theta$  satisfying
\begin{equation}\label{definition_puissance_complete}
\|\Gamma^{1/2}\theta\| ^2 \geq 
C_1\|(\Gamma^{1/2}-\Gamma_k^{1/2})\theta\|^2 + C_2\frac{\sigma^2}{n}\left(\sqrt{k\log\left(\frac{1}{\alpha \beta}\right)}+\log\left(\frac{1}{\alpha \beta}\right)\right)\ .
\end{equation}
 \end{thrm}

\begin{remark}
If we knew that $\theta$ belongs to the space spanned by the $k$ first eigenvectors $(V_1,\ldots, V_k)$ and if we knew these $k$ eigenvectors in advance, then we could  consider the statistic defined by 
$$\widetilde{\phi}_k({\bf X},{\bf Y}):= \frac{\|\widetilde{\bf \Pi}_k{\bf Y}\|_n^2}{\|{\bf Y}-\widetilde{\bf \Pi}_k{\bf Y}\|_n^2}-\bar{\mathcal{F}}^{-1}_{k,n-k}(\alpha)\ , $$
where $\widetilde{\bf \Pi}_k$ is the projection in $\mathbb{R}^n$ onto the space spanned by $(\langle X_i,V_j\rangle )_{i=1,\ldots n},$ $j=1,\ldots, k$.
The corresponding test is optimal in the minimax sense and rejects $H_0$ with probability larger than $1-\beta$ when 
\begin{equation}\label{puissance_parametrique}
\|\Gamma^{1/2}\theta\| ^2\geq C(\alpha,\beta)\sqrt{k}\sigma^2/n\ . 
\end{equation}
See~\cite{VV10} for a proof when $X$ is a Gaussian process and $\epsilon$ follows a Gaussian distribution, the extension to non Gaussian processes being straightforward.
 In (\ref{definition_puissance_complete}), we  recover an additional term $\|(\Gamma^{1/2}-\Gamma_k^{1/2})\theta\|^2$ because we do not assume that $\theta$ belongs to the space spanned by $(V_1,\ldots, V_k)$. The statistic $\phi_k({\bf Y},{\bf X})$ only captures the projection of $\theta$ onto $\mathrm{span}(V_1,\ldots,V_k)$. 
In fact, the test $T_{\alpha,k}$ rejects with large probability when 
\[\|\Gamma^{1/2}_k\theta\|^2= \|\Gamma^{1/2}\theta\|^2-\|(\Gamma^{1/2}-\Gamma_k^{1/2})\theta\|^2\] 
is large.
\end{remark}

\begin{remark}[Joint regularity of $\Gamma$ and $\theta$]  Looking more precisely at the bias term, we obtain
\[\|(\Gamma^{1/2}-\Gamma_k^{1/2})\theta\|^2=\sum_{j=k+1}^{\infty}\lambda_j\langle \theta,V_j\rangle^2\ .\]
Consequently, the bias term does not only depend on the rate of convergence of the eigenvalues of $\Gamma$, it also depends on the behavior of the sequence $\lambda_j\langle\theta,V_j\rangle^2$. In other words, the joint regularity of the covariance operator $\Gamma$ and of $\theta$ (in the expansion of $(V_j)$, $j\geq 1$) plays a role in the bias term. For a fixed $\theta$, the power of $T_{\alpha,k}$ is large for a tuning parameter $k$ that achieves a trade-off between the bias term $\|(\Gamma^{1/2}-\Gamma_k^{1/2})\theta\|^2$ and a variance term $\sqrt{k}\sigma^2/n$.
\end{remark}

\subsection{Minimax separation distance over an ellipsoid}\label{section_minimax_KL}

In this section, we assess the optimality of the procedure $T_{\alpha,k}$.  To this end, we study the optimal power of a level-$\alpha$ test, when $\theta$ is assumed to have a known regularity.

\begin{defi}[Ellipsoids]\label{definition_ellipsoid}
Given a non increasing sequence $(a_i)_{i\geq 1}$ and a positive number $R>0$, we define the ellipsoid $\mathcal{E}_a(R)$ by
\begin{eqnarray*}
 \mathcal{E}_a(R):= \left\{\theta\in\mathcal{H}: \quad \sum_{k=1}^{+\infty} \frac{\langle \theta,V_k\rangle ^2}{a_k^2}\leq R^2\sigma^2\right\} \ . 
\end{eqnarray*} 
\end{defi}

The ellipsoid $\mathcal{E}_a(R)$ contains all the elements $\theta\in\mathcal{H}$ that have a given regularity in the basis $(V_k)$, $k\geq 1$. In other words, it prescribes the rate of convergence of $\langle\theta,V_k\rangle $ towards $0$. The faster $a_k$ goes to zero, the more regular $\theta$ is assumed to be.

We take some positive numbers $\alpha$ and $\beta$ such that $\alpha+\beta<1$.
Let us consider a test $T$ taking its values in $\{0,1\}$. For any subset $\mathcal{C}\subset \mathcal{H}\times\mathbb{R}^+$,  
$ \boldsymbol{\beta}\left[T; \mathcal{C} \right]$ 
denotes the supremum of type II errors of the test $T$ for all  parameters $(\theta,\sigma)\in \mathcal{C}$:
\[\boldsymbol{\beta}\left[T; \mathcal{C} \right]:=\sup_{(\theta,\sigma)\in\mathcal{C}}\mathbb{P}_{\theta}[T= 	0]\ .\]

The {\bf $(\alpha,\beta)$-separation distance} of an $\alpha$-level test $T$ over 
the ellipsoid $\mathcal{E}_a(R)$, noted  $\rho[T;\mathcal{E}_a(R)]$ is the minimal number $\rho>0$ such that $T$ rejects $H_0$ with probability larger than $1-\beta$ for all $\theta\in\mathcal{E}_a(R)$ and $\sigma>0$ such that $\|\Gamma^{1/2}\theta\| ^2/\sigma^2\geq \rho^2$. Hence, $\rho[T;\mathcal{E}_a(R)]$ corresponds to the minimal distance such that the hypotheses $\{\theta=0,\ \sigma>0\}$ and $\{\theta\in\mathcal{E}_a(R),\ \sigma>0,\|\Gamma^{1/2}\theta\| ^2/\sigma^2\geq \rho^2\}$ are well separated by $T$.
\[\rho[T;\mathcal{E}_a(R)]:= \inf\left\{\rho>0,\quad \boldsymbol{\beta}\left[T; \left\{\theta\in\mathcal{E}_a(R), \sigma>0, \frac{\|\Gamma^{1/2}\theta\| ^2}{\sigma^2}\geq \rho^2\right\}\right]\leq \beta\right\}\ .\] 
By definition, $T$ has a power larger than $1-\beta$ for all $\theta\in\mathcal{E}_a(R)$ and $\sigma>0$ such that $\|\Gamma^{1/2}\theta\| ^2/\sigma^2\geq \rho^2[T,\mathcal{E}_a(R)]$.

\begin{defi}[Minimax Separation distance]\label{definition_minimax_separation_distance}
We consider 
\begin{equation}
 \rho^*[\alpha;\mathcal{E}_a(R)]:=\inf_{T_{\alpha}}\rho[T_{\alpha};\mathcal{E}_a(R)]\ ,
\end{equation}
where the infimum run over all level-$\alpha$ tests. This quantity is called the {\bf $(\alpha, \beta)$-minimax separation distance} over the ellipsoid $\mathcal{E}_a(R)$.

\end{defi}

\begin{remark}
The notion of $(\alpha, \beta)$-minimax separation distance is a non asymptotic counterpart of the detection boundaries studied in the Gaussian sequence model~\cite{donoho04}. Furthermore, as the variance $\sigma^2$ is unknown, 
this definition of the minimax separation distance considers the power of the testing  procedures for all possible values of $\sigma^2$.
\end{remark}

\begin{prop}[Minimax lower bound over an ellipsoid]\label{prop_minoration}
There exists a constant $C(\alpha,\beta)$ such that the following holds.
Let us assume that $X$ is a Gaussian process and that $\epsilon$ follows a Gaussian distribution. For any ellipsoid $\mathcal{E}_a(R)$, 
we have 
\begin{equation}\label{eq:minoration_minimax}
  \rho^*[\alpha;\mathcal{E}_a(R)]\geq \rho^2_{a,R,n}:= \sup_{k\geq 1}\left[C(\alpha,\beta)\left(\frac{\sqrt{k}}{n}\right)\wedge \left(R^2a_k^2\lambda_k\right)\right]\ .
\end{equation}

In other words, for any test $T_{\alpha}$ of  level $\alpha$, we have
\begin{eqnarray*}
 \boldsymbol{\beta}\left[T_{\alpha}; \left\{\theta\in\mathcal{E}_a(R), \sigma>0, \frac{\|\Gamma^{1/2}\theta\| ^2}{\sigma^2} \geq \rho^2_{a,R,n}\right\}\right]\geq \beta\ .
\end{eqnarray*}
\end{prop}
Consequently, the  $(\alpha,\beta)$ minimax-separation distance over $\mathcal{E}_a(R)$ is lower bounded by  $\rho^2_{a,R,n}$. The next proposition states the corresponding  upper bound.
\begin{cor}[Minimax upper bound]\label{prop_majoration_ellipsoide}
 Under ${\bf B.1,2,4}$, there exists positive constants $C(\gamma)$, $C_2$, $C_3(\alpha,\gamma)$, and $C_4(\alpha,\beta)$ such that the following holds. 
Given an ellipsoid $\mathcal{E}_a(R)$, we define 
\begin{equation}\label{definition_kn}
k^*_n:= \inf\left\{ k\geq 1 ,\  a^2_k\lambda_k R^2\leq \frac{\sqrt{k}}{n}\right\}\ . 
\end{equation}
Assume that $\alpha\geq e^{-\sqrt{n}}$, $\beta\geq C(\gamma)/\log(n)$,  $n\geq C_2$, and $k^*_n\leq n^{1/4}/\log^{4}(n)$. Then, the test $T_{\alpha,k_n^*}$ has a size smaller than $\alpha+ C_3(\alpha,\gamma)/\log(n)$ and is minimax over $\mathcal{E}_a(R)$:

\begin{equation}\label{majoration_Tk}
 \boldsymbol{\beta}\left[T_{\alpha,k^*_n}; \left\{\theta\in\mathcal{E}_a(R), \sigma>0, \frac{\|\Gamma^{1/2}\theta\| ^2}{\sigma^2}\geq C_4(\alpha,\beta)\rho^2_{a,R,n}\right\}\right]\leq \beta\ .
\end{equation}
\end{cor}

This corollary is a straightforward consequence of Theorem \ref{thrm_power_KL_fixed}. Hence, the test $T_{\alpha,k^*_n}$ is minimax over $\mathcal{E}_a(R)$, that is, its $(\alpha,\beta)$-separation distance equals (up to a multiplicative constant) the $(\alpha,\beta)$ minimax separation distance. Interestingly, the upper bound (\ref{majoration_Tk}) does not require the error $\epsilon$ to be normally distributed.

\begin{remark}
As a consequence, the $(\alpha,\beta)$-minimax separation distance over $\mathcal{E}_a(R)$ is of order
\begin{equation*}
\rho^2_{a,R,n}:= \sup_{k\geq 1}\left[C(\alpha,\beta)\left(\frac{\sqrt{k}}{n}\right)\wedge \left(R^2a_k^2\lambda_k\right)\right] \ .
\end{equation*}
It depends on the behavior of the non-increasing sequence $(\lambda_ka_k^2)$, where the sequence of eigenvalues $(\lambda_k)$ prescribes the ``regularity" of the process $X$ and  the sequence $(a_k)$ prescribes the regularity of $\theta$. In order to grasp the quantity $\rho^2_{a,R,n}$, let us specify some examples of sequences $\lambda_ka_k^2$:
\end{remark}

\begin{cor}\label{cor:minimax}
{\bf Polynomial decay.} If $\lambda_ka_k^2=k^{-s}$ with $s>7/2$, then the $(\alpha,\beta)$-minimax separation is of order $R^{2/(1+2s)}n^{-2s/(1+2s)}$. This rate is achieved by the test $T_{\alpha,k}$ with $k\asymp (R^2n)^{2/(1+2s)}$.\\
{\bf Exponential decay}. If $\lambda_ka_k^2=e^{-sk}$ with $s>0$, then the $(\alpha,\beta)$-separation distance of  $T_{\alpha}^{(1)}$ over $\mathcal{E}_a(R)$ is of order $\frac{\sqrt{\log(n)}}{\sqrt{s}n}$.
 This rate is achieved by the test $T_{\alpha,k}$ with $k\asymp \log(n)/s$.
\end{cor}
\begin{remark}
The condition $s>7/2$ in the polynomial regime arises because of Assumption ${\bf B.3}$ ($k\leq n^{1/4}/\log^4(n)$). This restriction is related to the difficulty to reliably estimate the eigenvalues $\lambda_k$ and eigenfunctions $V_k$ when $k$ is large (see Lemma \ref{lemlem1} and its proof). If the process $X$ and the function $\theta$ are less regular ($s<7/2$), our theory only allows us to take $k=n^{1/4}/\log^4(n)$ in $T_{\alpha,k}$ which leads to a  rate of testing of order (up to $\log$ terms) $n^{-s/4}R^2$  while the minimax lower bound is of order $R^{2/(1+2s)}n^{-2s/(1+2s)}$.  Note that similar restrictions also occur in state-of-the-art results for estimation. For instance, Condition (3.3) in  Hall and Horowitz \cite{HH07} amounts to $s>3$.
\end{remark}

In conclusion, $T_{\alpha,k}$ achieves the optimal rate of detection when $k$ is suitably chosen. However, the choice of $k$ depends on unknown quantities such as the regularity of $X$ or the regularity of $\theta$. Taking $k$ too small does not allow to detect non-zero $\theta$ such that the bias $\|(\Gamma^{1/2}-\Gamma_k^{1/2})\theta\|^2$ in \eqref{definition_puissance_complete} is too large. In contrast, taking $k$ too large leads to a large variance term $\sqrt{k}/n$ in \eqref{definition_puissance_complete}. The best $k$ corresponds to the trade-off between the bias term and the variance term in \eqref{definition_puissance_complete}. In the following, we introduce a procedure that nearly achieves this trade-off without requiring any prior knowledge of the regularity of $\Gamma$ or $\theta$.

\section{A multiple testing procedure}
\label{mult.test.proc}

\subsection{Definition}

In the sequel, $\mathcal{K}_n$  stands for a ``dyadic" collection of dimensions
defined by 
\begin{eqnarray}\label{collection_dimension}
\mathcal{K}_n=\{2^0,2^1,2^2,2^3\ldots,\bar{k}_n\}\ , 
\end{eqnarray}
where $\bar{k}_n$ is a power of $2$ that will be fixed later. As $k$ cannot be a priori chosen, we evaluate the statistic $\phi_{k}({\bf Y},{\bf X})$ for all $k$ belonging to a collection $\mathcal{K}_n$. This choice of the collection $\mathcal{K}_n$ is discussed in the next section.

\begin{defi}[KL-Test]\label{definition_KL_test}

We reject $H_0$: ``$\theta=0$" when the statistic
\begin{eqnarray}\label{definition_procedure_KL}
 T_{\alpha}:= \sup_{ k\in\mathcal{K}_n,\ k\leq \mathrm{Rank}(\widehat{\Gamma}_n)
}\left[\phi_{k}({\bf Y},{\bf X})-
\hat{k}^{KL}\bar{\F}_{\hat{k}^{KL},n-\hat{k}^{KL}}^{-1}\{\alpha_{\K_n}({\bf X})\}\right]\ 
\end{eqnarray}
is positive, where the  weight $\alpha_{\K_n}({\bf
X})$ is chosen according to one of the procedures $P_1$
and $P_2$ explained below.
\noindent
~\\ 
${\bf P_1}$: (Bonferroni) 
$\alpha_{\K_n}({\bf X})$ is equal to $ 
\alpha/|\mathcal{K}_n|$.
~\\
${\bf P_2}$: Let  ${\bf Z}$ be a standard Gaussian vector of size $n$. 
We take $\alpha_{\K_n} ({\bf X})= q_{{\bf
X},
\alpha}$,
the $\alpha$-quantile of the distribution of the random variable 
\begin{eqnarray}  
\inf_{k \in \mathcal{K}_n} \bar{\F}_{\hat{k}^{KL},n-\hat{k}^{KL}}\left[
\phi_k({\bf Z},{\bf X})/\hat{k}^{KL}\right]\label{methode_quantile} 
\end{eqnarray} 
conditionally to ${\bf X}$. 
\end{defi}
In the sequel, $T^{(1)}_{\alpha}$ (resp. $T^{(2)}_{\alpha}$) refers to the statistic $T_{\alpha}$, defined with Procedure $P_1$ (resp. $P_2$). $T^{(1)}_{\alpha}$ corresponds to a Bonferroni multiple testing procedure. In contrast $T^{(2)}_{\alpha}$ handles better the dependence between the statistics $\phi_k$, by using an ad-hoc quantile $q_{{\bf X},\alpha}$. We compare  these two tests in Section \ref{sub.sec_compT1_T2}.
This multiple testing approach has already been considered in the non-parametric fixed design regression setting~\cite{baraud03}.

\begin{remark}{\bf [Computation of $q_{{\bf X},\alpha}$]}
Let $Z$ be 
a standard Gaussian random vector of size $n$ 
independent of ${\bf X}$. As ${\boldsymbol{\epsilon}}$ is independent of 
${\bf X}$, the distribution of (\ref{methode_quantile}) conditionally to ${\bf 
  X}$ is the same as the 
distribution of  
$$\inf_{k \in \mathcal{K}_n} \bar{\mathcal{F}}_{\hat{k}^{KL},n-\hat{k}^{KL}} \left(\frac{ \| \widehat{\Pin}_k Z 
  \|_n^2 /\hat{k}^{KL} }{ \|Z - \widehat{\Pin}_k  Z  \|_n^2/(n-\hat{k}^{KL})}\right)$$ 
conditionally to ${\bf X}$. As a consequence, one can simulate a random variable that follows the same distribution as (\ref{methode_quantile}) conditionally to ${\bf X}$.
Hence, the quantile $q_{{\bf X},\alpha}$ is  easily worked out applying a Monte-Carlo approach.
\end{remark}

\begin{remark}{\bf [Choice of $\bar{k}_n$]}
 In practice, we advise to take $\bar{k}_n=2^{\lfloor \log_2 n\rfloor -1}$ which lies between $n/4$ and $n/2$. This choice is supported by practical experiences and results obtained in sections \ref{subsection_level_mult.test.proc} and Appendix \ref{sec:power:gaussian}.
Nevertheless, some of the theoretical results will require to take a slightly smaller value for $\bar{k}_n$.
\end{remark}

\subsection{Size of the tests}
\label{subsection_level_mult.test.proc}

\begin{prop}[Size of $T^{(1)}_{\alpha}$ and $T^{(2)}_{\alpha}$ under Gaussian errors]\label{prop_niveau_gaussien}
Under Assumption ${\bf A.1}$ and if $\bar{k}_n\leq n/2$, we have for any $n\geq 2$,
\begin{equation*}
 \mathbb{P}_{0}(T^{(1)}_{\alpha} > 0) \leq  \alpha\ ,\quad\quad\quad
 \mathbb{P}_{0}(T^{(2)}_{\alpha} > 0) = \alpha\ . 
\end{equation*}
\end{prop}
If the noise $\epsilon$ follows a Gaussian distribution, the size of $T_{\alpha}^{(2)}$ is exactly $\alpha$, while the size of $T_{\alpha}^{(1)}$ is smaller than $\alpha$ because of the Bonferroni correction. Let us now control the size of $T_{\alpha}^{(1)}$ assuming that $\epsilon$ admit a finite fourth moment. 
\begin{eqnarray*}
 {\bf B'.3}\hspace{4.4cm}\bar{k}_n\leq n^{1/4}/\log^{4}(n)\ .\hspace{4.4cm} 
\end{eqnarray*}
The assumption ${\bf B'.3}$ is the counterpart of {\bf B.3} for a multiple testing procedure. Next, we state the counterpart of Theorem \ref{thrm_niveau_T.alpha.k} for $T_{\alpha}^{(1)}$.

\begin{thrm}[Size of $T^{(1)}_{\alpha}$]\label{thrm_niveau}
Under Assumptions ${\bf B.1}$, ${\bf B.2}$, and ${\bf B'.3}$, there exist positive constants $C(\alpha,\gamma)$ and $C_2$  such that the following holds.
For any $n\geq C_2$, we have
\begin{eqnarray*}
\mathbb{P}_{0}\left[T^{(1)}_{\alpha}> 0\right]\leq  \alpha+ \frac{C(\alpha,\gamma)}{\log(n)}\ .
\end{eqnarray*}
\end{thrm}

\subsection{Comparison of $T^{(1)}_{\alpha}$ and $T^{(2)}_{\alpha}$}
\label{sub.sec_compT1_T2}
The test $T^{(2)}_{\alpha}$ is always more powerful
 than  $T^{(1)}_{\alpha}$ as shown in the next proposition.

\begin{prop}\label{prop_puissance_comparaison}
For any parameter $\theta\neq 0$, the tests
$T^{(1)}_{\alpha}$ and
$T^{(2)}_{\alpha}$ satisfy
\begin{eqnarray}\label{comparaison_puissance} 
\mathbb{P}_{\theta}\left(\left. T^{(2)}_{\alpha} >0\right|{\bf
X}\right)\geq 
\mathbb{P}_{\theta}\left(\left. T^{(1)}_{\alpha}>0 \right|{\bf 
  X}\right)\,\, \, \, {\bf X}\ a.s.\ . 
\end{eqnarray}  
\end{prop}
On one hand, the choice of Procedure $P_1$ is valid even for a non-Gaussian noise and avoids the computation
of the quantile $q_{{\bf X},\alpha}$. On the other hand, the test $T_{\alpha}^{(2)}$ has a size exactly $\alpha$ when the error is Gaussian and  is more powerful than the corresponding test with Procedure $P_1$. This comparison is numerically illustrated  in Section \ref{section_simulation}.

\section{Power and adaptation  of $T_{\alpha}^{(1)}$ }
\label{section_rates_mult.test.proc}

Since  $T^{(2)}_{\alpha}$ is always more powerful
 than $T^{(1)}_{\alpha}$, we only consider the power and the minimax optimality of $T_{\alpha}^{(1)}$. 


\begin{thrm}[Power under non-Gaussian errors]\label{thrm_power_KL_fixed_multiple}
Let $\alpha$ and $\beta$ be fixed. Under ${\bf B.1-2}$, ${\bf B'.3}$, ${\bf B.4}$, 
there exist positive constants $C(\gamma)$, $C_1$, $C_2$, and $C_3$ such that the following holds. Assume that $\alpha\geq e^{-\sqrt{n}}$,   $\beta\geq C(\gamma)/\log(n)$,  and that $n\geq C_3$. Then, $\mathbb{P}_{\theta}(T_{\alpha}^{(1)}>0)\geq 1-\beta$ for any $\theta$  satisfying
\begin{equation}\label{definition_puissance_complete_multiple}
\|\Gamma^{1/2}\theta\| ^2 \geq \inf_{k\in\mathcal{K}_n} 
C_1\|(\Gamma^{1/2}-\Gamma_k^{1/2})\theta\| ^2 + C_2\frac{\sigma^2}{n}\left(\sqrt{k\log\left(\frac{\log n}{\alpha \beta	}\right)}+\log\left(\frac{\log n}{\alpha \beta}\right)\right)\ .
 \end{equation}
 \end{thrm}

\begin{remark} Comparing Theorems \ref{thrm_power_KL_fixed} and \ref{thrm_power_KL_fixed_multiple}, we observe that the rejection region $T_{\alpha}^{(1)}$ almost contains all the rejection regions of the tests $T_{\alpha,k}$ for all $k\in\mathcal{K}_n$.  The price to pay for this feature is an additional  $\sqrt{\log\log n}$ in the variance term of (\ref{definition_puissance_complete_multiple}):
\[\frac{\sigma^2}{n}\left(\sqrt{k\log\left(\frac{\log n}{\alpha \beta	}\right)}+\log\left(\frac{\log n}{\alpha \beta}\right)\right)\]
This $\log\log(n)$ term corresponds to the quantity $\log(|\mathcal{K}_n|)$. If we had used a collection of the form $\{1,\ldots ,\bar{k}_n\}$ instead of $\mathcal{K}_n$ the $\log\log(n)$ would have been replaced by a $\log(n)$. We prove below that this $\log\log n$ term  is in fact unavoidable for an adaptive procedure.
\end{remark}
As for $T_{\alpha,k}$, additional controls of the power when $\epsilon$ follows a Normal distribution are stated in Appendix \ref{sec:power:gaussian}. Let us now consider  the power of $T_{\alpha}^{(1)}$ over ellipsoids $\mathcal{E}_{a}(R)$. In the sequel, $\lfloor .\rfloor$ stands for the integer part, while $\log_2(.)$ corresponds to the binary logarithm.
\begin{cor}[Power of $T_{\alpha}^{(1)}$ over ellipsoids]\label{coro_adaptation}
Under ${\bf B.1}$, ${\bf B.2}$,  and ${\bf B.4}$, there exist positive constants $C(\gamma)$, $C_1$, $C_2$, $C_3(\alpha,\beta)$, and $C_4(\alpha,\beta)$ such that the following holds. Assume that $\alpha\geq e^{-\sqrt{n}}$, that $\beta\geq C(\gamma)/\log(n)$, and $n\geq C_2$. Consider the test $T_{\alpha}^{(1)}$ with $\bar{k}_n=2^{\lfloor \log_2 [n^{1/4}/\log^4(n)] \rfloor}$. Fix  any ellipsoid $\mathcal{E}_a(R)$.
\begin{enumerate}
 \item   
 We have $\mathbb{P}_{\theta}(T_{\alpha}^{(1)}>0)\geq 1- \beta$ for any $\theta\in\mathcal{E}_a(R)$ satisfying 
\begin{equation*}
\frac{\|\Gamma^{1/2}\theta\| ^2}{\sigma^2}\geq C_3(\alpha,\beta) \inf_{k=1,2,4,\ldots,\bar{k}_n} 
\left[\lambda_{k+1}a^2_{k+1}R^2+\frac{\sigma^2}{n}\left(\sqrt{k\log\log n}+\log\log n\right)\right]\ .
\end{equation*}
\item Consider $k^*_n$ as in (\ref{definition_kn}). 
If $\log\log(n)\leq k_n^*\leq  \bar{k}_n$, then $\mathbb{P}_{\theta}(T_{\alpha}^{(1)}>0)\geq 1- \beta$ for any $\theta\in\mathcal{E}_a(R)$ satisfying 
\begin{equation*}
 \frac{\|\Gamma^{1/2}\theta\| ^2}{\sigma^2}\geq C_4(\alpha,\beta)\sqrt{\log\log n}\rho^2_{a,R,n}\ ,
\end{equation*}
where $\rho_{a,R,n}$ is defined in \eqref{eq:minoration_minimax}
\end{enumerate}

\end{cor}
This is a direct consequence of Theorem \ref{thrm_power_KL_fixed}.

\begin{remark}
If we compare Corollary \ref{coro_adaptation} with the minimax lower bound of Proposition \ref{prop_minoration}, we observe that the separation distance only matches up to a factor of order $\sqrt{\log\log n}$. As a consequence, $T_{\alpha,k}$ is  almost minimax over all ellipsoids  $\mathcal{E}_a(R)$ satisfying $\log\log(n)\leq k_n^*\leq  \bar{k}_n$. Next, we prove that this $\sqrt{\log\log(n)}$ term loss is unavoidable when the ellipsoid  $\mathcal{E}_a(R)$ is unknown.  
\end{remark}

\begin{prop}[Minimax lower bounds over a collection of nested ellipsoids]\label{prop_minoration_adaptation}
There exists a positive constant $C(\alpha,\beta)$ such that the following holds.
Let us assume that $X$ is a Gaussian process, that the noise $\epsilon$ follows a Gaussian distribution, and that the rank of $\Gamma$ is infinite. For any ellipsoid $\mathcal{E}_a(R)$, 
we set
\begin{eqnarray*}
 \tilde{\rho}^2_{a,R,n}:= \sup_{k\geq 1}\left[C(\alpha,\beta)\left(\frac{\sqrt{\log\log (k\vee 3)}\sqrt{k}}{n}\right)\wedge \left(R^2a_k^2\lambda_k\right)\right]\ . 
\end{eqnarray*}
For any non increasing sequence $(a_k)_{k\geq 1}$ and any test $T$ of level $\alpha$, we have
\begin{eqnarray*}
 \boldsymbol{\beta}\left[T; \bigcup_{R>0}\left\{\theta\in\mathcal{E}_a(R), \sigma>0,\  \frac{\|\Gamma^{1/2}\theta\| ^2}{\sigma^2} \geq \tilde{\rho}^2_{a,R,n}\right\}\right]\geq \beta\ .
\end{eqnarray*}
\end{prop}
As a consequence, there is a  $\sqrt{\log \log n }$ price to pay if we  simultaneously  consider a nested collection of ellipsoids. Such impossibility for perfect adaptation has already been observed for the testing problem  in the classical nonparametric regression framework~\cite{spokoiny96}.

\begin{remark}
In order to compare the lower and upper bounds of Proposition \ref{prop_minoration_adaptation} and Corollary \ref{coro_adaptation}, let us  specify the sequence $\lambda_ka_k^2$:
\begin{itemize}
\item  Polynomial decay. If $\lambda_ka_k^2=k^{-s}$, then the $(\alpha,\beta)$-separation distance of  $T_{\alpha}^{(1)}$ over $\mathcal{E}_a(R)$ is of  order $$R^{2/(1+2s)}\left(\frac{\sqrt{\log\log(n)}}{n}\right)^{2s/(1+2s)}\ ,$$
for $s>7/2$. By Proposition \ref{prop_minoration_adaptation}, this rate is optimal for adaptation.
\item  Exponential decay. If $\lambda_ka_k^2=e^{-sk}$, then the $(\alpha,\beta)$-separation distance of  $T_{\alpha}^{(1)}$ over $\mathcal{E}_a(R)$ is of  order $$\frac{\sqrt{\log(n)\log\log (n)}}{\sqrt{s}n} ,$$
for any $s>0$. By Proposition \ref{prop_minoration_adaptation}, this rate is almost optimal for adaptation (up to a $\sqrt{\log\log(n)/\log\log\log n}$ term).
\end{itemize}

\end{remark}
In conclusion, the procedure $T_{\alpha}^{(1)}$ is adaptive to the unknown regularity of $\theta$, to the unknown regularity of the eigenvalues $(\lambda_k)_{k\geq 1}$ and to the unknown noise variance $\sigma^2$. Interestingly, the minimax rate of testing depends on the decay of the non-increasing sequence $(\lambda_ka^2_k)_{k\geq 1}$.

\section{Simulations}\label{section_simulation}

\subsection{Experiments}

{\bf Setting}. The performances of the procedures $T_{\alpha}^{(1)}$ and $T_{\alpha}^{(2)}$ are  illustrated for various choices of the function $\theta$. In all experiments, the noise $\epsilon$ follows a standard Gaussian distribution with unit variance, while the process $X$ is a Brownian motion defined on $[0,1]$. The eigenfunctions and eigenvalues of the covariance operator of the Brownian motion have been computed in Ash \& Gardner~\cite{Ash-Gardner}:
\begin{equation*}
\lambda_j=\frac{1}{(j-0.5)^2 \pi^2} \quad \text{ and } \quad V_j(t)=\sqrt{2} \sin\big( (j-0.5) \pi t \big) \; , t \in [0,1] \; , \quad j=1,2,\dots
\end{equation*}
In practice $X(t)$ has been simulated using a truncated version of the Karhunen Lo\`eve expansion $ \sum_{j=1}^{100}\sqrt{\lambda_j}\eta^{(j)} V_j(t)$, where 
 the $\left(\eta^{(j)}\right)_{j \in \N }$ form an i.i.d. sequence of standard normal variables. The function $X(t)$ is observed on $1000$ evenly spaced points in $[0,1]$.

{\bf Testing procedure}. For each experiment, we perform the tests $T_{\alpha}^{(1)}$ (procedure $P_1$) and $T_{\alpha}^{(2)}$ (procedure $P_2$) with $\bar{k}_n=2^{\lfloor \log_2 n -1\rfloor}$. The quantile $q_{{\bf X},\alpha}$ involved in $P_2$ is computed by Monte Carlo simulations. For each experiment, we use 1000 random simulations to estimate this quantile.

{\bf Choice of $\theta$.}
\begin{enumerate}
 \item In the first experiment, we fix $\theta=0$ as a way to evaluate the sizes of the testing procedures.

\item In the second experiment, we build directly the function $\theta$ in the KL basis of $X$. The set $\Theta_{KL}$ is made of all the functions $\theta_{B,\xi}$ with $B>0$, $\xi>0$, and 
\begin{equation}
\label{theta-KL-sim}
\theta_{B,\xi}(t):=\frac{ B}{\sqrt{\sum_{k= 1}^{+\infty}k^{-2\xi-1}}} \; \sum_{j=1}^{100}{j^{-\xi-0.5} V_j(t)}\ ,
\end{equation}
where $\xi$ is a smoothness parameter. 
Observe that $B$ stands for the $l_2$ norm of the function $\theta_{B,\xi}$. As shown on Figure \ref{fig-theta-KL}, the smoothness of $\theta_{B,\xi} \in \Theta_{KL}$ increases with $\xi$. For this experiment, we have an explicit expression of the joint regularity of $\theta$ and $\Gamma$: \[\|(\Gamma^{1/2}-\Gamma_k^{1/2})\theta\| ^2= \frac{B^2}{ \pi^2\left(\sum_{l= 1}^{+\infty}l^{-2\xi-1}\right)}\sum_{j=k+1}^{100} (j-0.5)^{-2}j^{-2\xi-1}\ .\]
In practice, we fix $\xi= 0.1,\ 0.5,\ 1$ and $B=0.1,\ 0.5,\ 1$.

\begin{figure}[!ht]
\centering
\includegraphics[height=7cm,width=8cm]{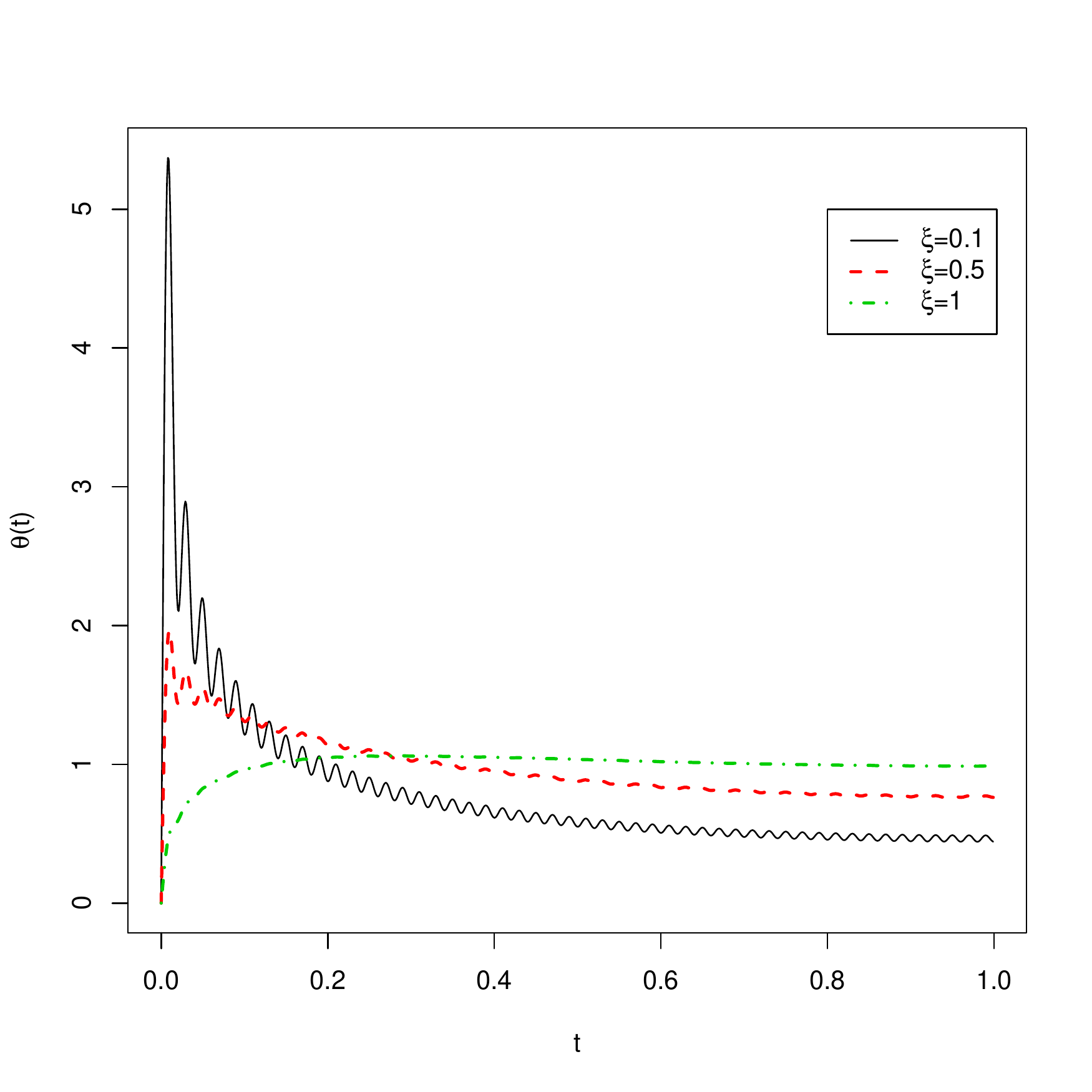}
\caption{Three functions $\theta$ in $\Theta_{KL}$ when $B=1$.}\label{fig-theta-KL}
\end{figure}

\item In the third experiment, we  consider the set $\Theta_G$ of functions
\begin{equation*}
 \theta_{B,\tau}(t) = B\exp\left[-\frac{(t-0.5)^2}{2\tau^2}\right]  \left[ \int_{0}^1 \exp\left[-\frac{(x-0.5)^2}{\tau^2}\right] dx\right]^{-1/2}\ ,
\end{equation*}
with $B>0$ and $\tau>0$. Here, $B$ stands for the $l_2$ norm of $\theta_{B,\tau}$ and $\tau$ is a smoothness parameter. In fact,  $\theta_{B,\tau}(t)$ corresponds (up to a constant) to the density of a normal variable with mean $0.5$ and variance $\tau^2$. As $\tau$ decreases to $0$, $\theta_{B,\tau}$ converges to a Dirac function centered on $0.5$. 
In practice, we fix $\tau=0.01,\ 0.02,\ 0.05$ and $B=0.5,\ 1,\ 2$.
 \end{enumerate}
{\bf Number of experiments.} We have set $n=100$ and $n=500$. For each set of parameters $(n,B,\xi)$ or $(n,B,\tau)$, 10 000 trials  were run to estimate the percentages of rejection of $H_0$ (ie. the percentages of positive values of $T_{\alpha}^{(1)}$ and $T_{\alpha}^{(2)}$ with $\alpha=5\%$), along with their 95\% confidence intervals.

\subsection{Results}
The two procedures $P_1$ and $P_2$ have been  implemented in R~\cite{Rcran} on a 3 GHz Intel Xeon processor, with a 4000KB cache size and 8GB total physical memory. 

\begin{table}[!ht]
\caption{First simulation study: Null hypothesis is true. Percentages of rejection of $H_0$ and 95\% confidence intervals} \label{tableR=0}
\begin{tabular}{c|cc|cc|}
\cline{2-5}
& \multicolumn{2}{c}{} & \multicolumn{2}{|c|}{}\\[-0.2cm]
& \multicolumn{2}{c}{$n=100$} & \multicolumn{2}{|c|}{$n=500$} \\[0.1cm]
\hline
\multicolumn{1}{|c|}{}&&&&\\[-0.3cm]
\multicolumn{1}{|c|}{$T_{\alpha}^{(1)}$} & 3.47 & ($\pm$ 0.36)&  2.61  & ($\pm$ 0.31)  \\[0.1cm]
\cline{1-5}
\multicolumn{1}{|c|}{}&&&&\\[-0.3cm]
\multicolumn{1}{|c|}{$T_{\alpha}^{(2)}$} &  4.97 &  ($\pm$ 0.43)&  5.26 & ($\pm$ 0.44) \\[0.1cm]
\cline{1-5}
\end{tabular}
\end{table}

{\bf First setting}. The percentages of rejection of $T_{\alpha}^{(1)}$ and $T_{\alpha}^{(2)}$ under $H_0$ with $n=100$ and $n=500$ are provided in Table \ref{tableR=0}. As expected, the size of $T_{\alpha}^{(1)}$ decreases when $n$ increases because we pay a price for the Bonferroni correction. The size of $T_{\alpha}^{(2)}$ remains close to the nominal level $\alpha=5\%$.

{\bf Second setting}. Tables \ref{tableKLn100} and \ref{tableKLn500} depict the results for $\theta\in\Theta_{KL}$ with $n=100$ and $n=500$ respectively. 
As expected, the power of the procedures is increasing with $B$ as $\|\theta\| $ becomes larger. Furthermore, the power also increases with $\xi$. This corroborates the rates stated in Section \ref{section_rates_mult.test.proc}, since the function $\theta_{B,\xi}$ becomes smoother when $\xi$ increases. In every setting the test  $T_{\alpha}^{(2)}$ with the second procedure performs better than $T_{\alpha}^{(1)}$.

{\bf Third setting}. The results of the last experiment are provided  in Tables \ref{tableGn100} for $n=100$ and \ref{tableGn500} for $n=500$. Again, the power is increasing with $B$, $n$ and $\tau$. Here,  $\tau$ does not directly correspond to the rate of convergence of the sequence $(\int_{0}^1\theta_{B,\tau}V_j(t)dt)$, $j\geq 1$ as $\xi$ does in the last example. Nevertheless, it is difficult to detect a function $\theta_{B,\tau}$ when $\tau$ decreases, that is when $\theta_{B,\tau}$ becomes 
close to a Dirac function.

In each setting, the test under $P_2$ is more powerful than the test under $P_1$. Nevertheless, the procedure $P_2$ is slightly slower to compute as it requires the evaluations of  the quantile $q_{{\bf X},\alpha}$ by a Monte-Carlo method. Under $P_1$, the mean computation time is $9$ seconds for $n=100$ and 12 seconds for $n=500$. In contrast, it  respectively equals  $11$ and $18$ seconds under $P_2$.


\begin{table}[!ht]
\caption{Second simulation study: $\theta \in \Theta_{KL}$, $n=100$. Percentages of rejection of $H_0$ and 95\% confidence intervals} \label{tableKLn100}
\begin{tabular}{cc|cc|cc|cc|}
\cline{3-8}
& & \multicolumn{2}{c}{} & \multicolumn{2}{|c|}{} & \multicolumn{2}{|c|}{}\\[-0.2cm]
& & \multicolumn{2}{c}{$B=0.1$} & \multicolumn{2}{|c|}{$B=0.5$} & \multicolumn{2}{|c|}{$B=1$}\\[0.1cm]
\hline
\multicolumn{1}{|c|}{\multirow{3}{*}{$\xi=0.1$}} & \multicolumn{1}{|c|}{} &&&&&&\\[-0.3cm]
\multicolumn{1}{|c|}{} &
\multicolumn{1}{|c|}{$T_{\alpha}^{(1)}$} & 3.88 & ($\pm$ 0.38) & 21.41& ($\pm$ 0.8)  &77.24 & ($\pm$ 0.82)\\[0.1cm]
\cline{2-8}
\multicolumn{1}{|c|}{} &&&&&&\\[-0.3cm]
\multicolumn{1}{|c|}{} &
\multicolumn{1}{|c|}{$T_{\alpha}^{(2)}$} & 5.8 & ($\pm$ 0.46) & 26.38& ($\pm$ 0.86)  & 81.78 & ($\pm$ 0.76) \\[0.1cm]
\cline{1-8}
\multicolumn{1}{|c|}{\multirow{3}{*}{$\xi=0.5$}} &  \multicolumn{1}{|c|}{} &&&&&&\\[-0.3cm]
\multicolumn{1}{|c|}{} &
\multicolumn{1}{|c|}{$T_{\alpha}^{(1)}$}  & 4.74 & ($\pm$ 0.42) & 46.47&($\pm$ 0.98)  &98.68 & ($\pm$ 0.22)\\[0.1cm]
\cline{2-8}
\multicolumn{1}{|c|}{} &&&&&&\\[-0.3cm]
\multicolumn{1}{|c|}{} &
\multicolumn{1}{|c|}{$T_{\alpha}^{(2)}$} & 6.65 & ($\pm$ 0.49) & 52.79& ($\pm$ 0.98)  & 99.06 & ($\pm$ 0.19)\\[0.1cm]
\cline{1-8}
\multicolumn{1}{|c|}{\multirow{3}{*}{$\xi=1$}} & \multicolumn{1}{|c|}{} &&&&&&\\[-0.3cm]
\multicolumn{1}{|c|}{} & 
\multicolumn{1}{|c|}{$T_{\alpha}^{(1)}$} &  4.8 & ($\pm$ 0.42) & 62.67& ($\pm$ 0.95) & 99.75& ($\pm$ 0.1)\\[0.1cm]
\cline{2-8}
\multicolumn{1}{|c|}{} &&&&&&\\[-0.3cm]
\multicolumn{1}{|c|}{} &
\multicolumn{1}{|c|}{$T_{\alpha}^{(2)}$} & 7.07 & ($\pm$ 0.5) & 68.3& ($\pm$ (0.91) & 99.84 & ($\pm$ 0.08)\\[0.1cm]
\cline{1-8}
\end{tabular}
\end{table}
\begin{table}[!ht]
\caption{Second simulation study: $\theta\in \Theta_{KL}$, $n=500$. Percentages of rejection of $H_0$ and 95\% confidence intervals} \label{tableKLn500}
\begin{tabular}{cc|cc|cc|cc|}
\cline{3-8}
& & \multicolumn{2}{c}{} & \multicolumn{2}{|c|}{} & \multicolumn{2}{|c|}{}\\[-0.2cm]
& & \multicolumn{2}{c}{$B=0.1$} & \multicolumn{2}{|c|}{$B=0.5$} & \multicolumn{2}{|c|}{$B=1$}\\[0.1cm]
\hline
\multicolumn{1}{|c|}{\multirow{4}{*}{$\xi=0.1$}} & \multicolumn{1}{|c|}{} &&&&&&\\[-0.3cm]
\multicolumn{1}{|c|}{} &
\multicolumn{1}{|c|}{$T_{\alpha}^{(1)}$} &  5.17 & ($\pm$ 0.43) & 86.98& ($\pm$ 0.66) &100 & ($\pm$ 0)\\[0.1cm]
\cline{2-8}
\multicolumn{1}{|c|}{} &&&&&&\\[-0.3cm]
\multicolumn{1}{|c|}{} &
\multicolumn{1}{|c|}{$T_{\alpha}^{(2)}$} &  8.48 & ($\pm$ 0.55) & 90.89& ($\pm$ 0.56)  & 100 &($\pm$ 0)\\[0.1cm]
\cline{1-8}
\multicolumn{1}{|c|}{\multirow{3}{*}{$\xi=0.5$}} &  \multicolumn{1}{|c|}{} &&&&&&\\[-0.3cm]
\multicolumn{1}{|c|}{} & 
\multicolumn{1}{|c|}{$T_{\alpha}^{(1)}$} &  8.81 &  ($\pm$ 0.56) & 99.85& ($\pm$ 0.08) &100 &($\pm$ 0) \\[0.1cm]
\cline{2-8}
\multicolumn{1}{|c|}{} &&&&&&\\[-0.3cm]
\multicolumn{1}{|c|}{} &
\multicolumn{1}{|c|}{$T_{\alpha}^{(2)}$} &  13.07 & ($\pm$ 0.66) & 99.88& ($\pm$ 0.07) & 100 & ($\pm$ 0)\\[0.1cm]
\cline{1-8}
\multicolumn{1}{|c|}{\multirow{3}{*}{$\xi=1$}} &  \multicolumn{1}{|c|}{} &&&&&&\\[-0.3cm]
\multicolumn{1}{|c|}{} & 
\multicolumn{1}{|c|}{$T_{\alpha}^{(1)}$} & 11.38 & ($\pm$ 0.62) & 99.99& ($\pm$ 0.02) &100 & ($\pm$ 0)\\[0.1cm]
\cline{2-8}
\multicolumn{1}{|c|}{} &&&&&&\\[-0.3cm]
\multicolumn{1}{|c|}{} &
\multicolumn{1}{|c|}{$T_{\alpha}^{(2)}$}  & 16.13 & ($\pm$ 0.72) & 100& ($\pm$ 0)  & 100 &($\pm$ 0)\\[0.1cm]
\cline{1-8}
\end{tabular}
\end{table}

\begin{table}[!ht]
  \caption{Third simulation study: $\theta \in \Theta_G$, $n=100$. Percentage of rejection of $H_0$ and 95\% confidence interval} \label{tableGn100}
\begin{tabular}{cc|cc|cc|cc|}
\cline{3-8}
& & \multicolumn{2}{c}{} & \multicolumn{2}{|c|}{} & \multicolumn{2}{|c|}{}\\[-0.2cm]
& & \multicolumn{2}{c}{$B=0.5$} & \multicolumn{2}{|c|}{$B=1$} & \multicolumn{2}{|c|}{$B=2$}\\[0.1cm]
\hline
\multicolumn{1}{|c|}{\multirow{3}{*}{$\tau=0.01$}} & \multicolumn{1}{|c|}{} &&&&&&\\[-0.3cm]
\multicolumn{1}{|c|}{} & 
\multicolumn{1}{|c|}{$T_{\alpha}^{(1)}$} &  4.94 & ($\pm$ 0.42) & 11.85&  ($\pm$ 0.63) &46.69 & ($\pm$ 0.98)\\[0.1cm]
\cline{2-8}
\multicolumn{1}{|c|}{} &&&&&&\\[-0.3cm]
\multicolumn{1}{|c|}{} &
\multicolumn{1}{|c|}{$T_{\alpha}^{(2)}$}  & 7.25 &  ($\pm$ 0.51) & 15.49&   ($\pm$ 0.71) & 53.56 & ($\pm$ 0.98)\\[0.1cm]
\cline{1-8}
\multicolumn{1}{|c|}{\multirow{3}{*}{$\tau=0.02$}} &  \multicolumn{1}{|c|}{} &&&&&&\\[-0.3cm]
\multicolumn{1}{|c|}{} & 
\multicolumn{1}{|c|}{$T_{\alpha}^{(1)}$} &  7.33 & ($\pm$ 0.51)  & 23.09&  ($\pm$ 0.83) &80.26 &  ($\pm$ 0.78)\\[0.1cm]
\cline{2-8}
\multicolumn{1}{|c|}{} &&&&&&\\[-0.3cm]
\multicolumn{1}{|c|}{} &
\multicolumn{1}{|c|}{$T_{\alpha}^{(2)}$} &  10 &  ($\pm$ 0.59) & 28.54&  ($\pm$ 0.89)  & 84.04 & ($\pm$ 0.72)\\[0.1cm]
\cline{1-8}
\multicolumn{1}{|c|}{\multirow{3}{*}{$\tau=0.05$}} &  \multicolumn{1}{|c|}{} &&&&&&\\[-0.3cm]
\multicolumn{1}{|c|}{} & 
\multicolumn{1}{|c|}{$T_{\alpha}^{(1)}$} &  13.85 &  ($\pm$ 0.68) & 56.51&   ($\pm$ 0.97)&99.48 &  ($\pm$ 0.14)\\[0.1cm]
\cline{2-8}
\multicolumn{1}{|c|}{} &&&&&&\\[-0.3cm]
\multicolumn{1}{|c|}{} &
\multicolumn{1}{|c|}{$T_{\alpha}^{(2)}$} &  18.13 & ($\pm$ 0.76) & 63.09& ($\pm$ 0.95)  & 99.65 &($\pm$ 0.12)\\[0.1cm]
\cline{1-8}
\end{tabular}
\end{table}
\begin{table}[!ht]
  \caption{Third simulation study: $\theta \in \Theta_G$, $n=500$. Percentage of rejection of $H_0$ and 95\% confidence interval} \label{tableGn500}
\begin{tabular}{cc|cc|cc|cc|}
\cline{3-8}
& & \multicolumn{2}{c}{} & \multicolumn{2}{|c|}{} & \multicolumn{2}{|c|}{}\\[-0.2cm]
& & \multicolumn{2}{c}{$B=0.5$} & \multicolumn{2}{|c|}{$B=1$} & \multicolumn{2}{|c|}{$B=2$}\\[0.1cm]
\hline
\multicolumn{1}{|c|}{\multirow{3}{*}{$\tau=0.01$}} & \multicolumn{1}{|c|}{} &&&&&&\\[-0.3cm]
\multicolumn{1}{|c|}{} & 
\multicolumn{1}{|c|}{$T_{\alpha}^{(1)}$} &  12.41 & ($\pm$ 0.65) & 54.6& ($\pm$ 0.98) &99.75 &($\pm$ 0.1) \\[0.1cm]
\cline{2-8}
\multicolumn{1}{|c|}{} &&&&&&\\[-0.3cm]
\multicolumn{1}{|c|}{} &
\multicolumn{1}{|c|}{$T_{\alpha}^{(2)}$} &  17.99 & ($\pm$ 0.75) & 63.16& ($\pm$ 0.95)  & 99.98 &($\pm$ 0.07)\\[0.1cm]
\cline{1-8}
\multicolumn{1}{|c|}{\multirow{3}{*}{$\tau=0.02$}} &  \multicolumn{1}{|c|}{} &&&&&&\\[-0.3cm]
\multicolumn{1}{|c|}{} & 
\multicolumn{1}{|c|}{$T_{\alpha}^{(1)}$} &  26.11 & ($\pm$ 0.86) &88.91& ($\pm$ 0.62) &100 & ($\pm$ 0)\\[0.1cm]
\cline{2-8}
\multicolumn{1}{|c|}{} &&&&&&\\[-0.3cm]
\multicolumn{1}{|c|}{} &
\multicolumn{1}{|c|}{$T_{\alpha}^{(2)}$} &  33.95 & ($\pm$ 0.93) & 92.62& ($\pm$ 0.51)  & 100 &($\pm$ 0)\\[0.1cm]
\cline{1-8}
\multicolumn{1}{|c|}{\multirow{3}{*}{$\tau=0.05$}} &  \multicolumn{1}{|c|}{} &&&&&&\\[-0.3cm]
\multicolumn{1}{|c|}{} & 
\multicolumn{1}{|c|}{$T_{\alpha}^{(1)}$} &  65.38 & ($\pm$ 0.93) & 99.95& ($\pm$ 0.04) &100 & ($\pm$ 0)\\[0.1cm]
\cline{2-8}
\multicolumn{1}{|c|}{} &&&&&&\\[-0.3cm]
\multicolumn{1}{|c|}{} &
\multicolumn{1}{|c|}{$T_{\alpha}^{(2)}$} &  72.74 & ($\pm$ 0.87) & 99.99&  ($\pm$ 0.02) & 100 &($\pm$ 0)\\[0.1cm]
\cline{1-8}
\end{tabular}
\end{table}

\vspace{0.5cm}

\section{Discussion}\label{section_discussion}

Two multiple testing procedures of the nullity of the slope function $\theta$ have been proposed in this paper. They are completely data-driven and benefit from optimal properties assessed in a nonasymptotic setting. We address here some extensions of our results.

\vspace{0.2cm}

Although we focused on the null-hypothesis ``$H_0$: $\theta=0$'', our approach easily extends to  linear hypotheses $H_{0,\cal{V}}$: ``$\theta \in \cal{V}$'', where $\cal{V}$ is a given finite dimensional subspace of $\H$ of dimension $p<n/2$. As previously, the procedure relies on parametric statistics for testing $H_{0,\cal{V}}$ against $H_{1,k,\cal{V}}: ``\theta\in (\text{Vect}(V_1,\ldots,V_k)+\cal{V})\setminus V$'', where $k$ is a positive integer. We consider the $n\times \hat{k}^{KL}$ design matrix ${\bf W}$ defined by ${\bf W}_{i,j}= \langle {\bf X}_i,\widehat{V}_j\rangle$ for  $i=1,\ldots n$, $j=1,\ldots \hat{k}^{KL}$. The space generated by the $\hat{k}^{KL}$ columns of the matrix $W$ is denoted $\underline{\cal{W}}_{\hat{k}^{KL}}$. Considering a basis $\left(\xi_1,\dots,\xi_p\right)$ of $\cal{V}$, we define $\underline{\cal{V}}_p$ as the space generated by the $p$ columns of the matrix whose $(ij)^{th}$ element is  $<{\bf X}_i,\xi_j>$.  In the sequel, $\widehat{\Pin}_{k,\cal{V}}$ stands for the 
orthogonal projection in $\mathbb{R}^n$ onto  $\underline{\cal{V}}_p^{\perp}\cap \underline{\cal{W}}_{\hat{k}^{KL}}$ of dimension less or equal to $\hat{k}^{KL}$, while $\widehat{\Pin}_{\cal{V}}$ stands for the orthogonal projection onto $\underline{\cal{V}}_p$. Then, we consider the following parametric statistic:
\begin{eqnarray}\label{definition_phi_KL.V}
 \phi_{k,\cal{V}}({\bf Y},{\bf X}):=
\frac{\|\widehat{\Pin}_{k,\cal{V}} {\bf
Y}\|^2_n}{\|{\bf Y}-\widehat{\Pin}_{k,\cal{V}} {\bf Y}-\widehat{\Pin}_{\cal{V}} {\bf Y}\|_n^2/[n-\dim(\underline{\cal{V}}_p +\underline{\cal{W}}_{\hat{k}^{KL}})]}\  .
\end{eqnarray}
Under $H_{0,\cal{V}}$, $\phi_{k,\cal{V}}({\bf Y},{\bf X})/\hat{k}^{KL}$ behaves like a Fisher distribution with $(\dim(\underline{\cal{V}}_p^{\perp} \cap\underline{\cal{W}}_{\hat{k}^{KL}}),n-\dim(\underline{\cal{V}}_p +\underline{\cal{W}}_{\hat{k}^{KL}}))$ degrees of freedom. The proof is the same as that for $\phi_k({\bf Y},{\bf X})$. In typical situations, we have $\dim(\underline{\cal{V}}_p^{\perp} \cap\underline{\cal{W}}_{\hat{k}^{KL}})=k$  and $\dim(\underline{\cal{V}}_p +\underline{\cal{W}}_{\hat{k}^{KL}})=k+p$.
We reject $H_{0,\cal{V}}$ when the statistic
\begin{eqnarray*}
 T_{\alpha,\cal{V}}:= \sup_{ k\in\mathcal{K}_n,\ k\leq \mathrm{Rank}(\widehat{\Gamma}_n)
}\left[\phi_{k,\cal{V}}({\bf Y},{\bf X})-
\hat{k}^{KL}\bar{\F}_{\dim(\underline{\cal{V}}_p^{\perp} \cap\underline{\cal{W}}_{\hat{k}^{KL}}),n-\dim(\underline{\cal{V}}_p +\underline{\cal{W}}_{\hat{k}^{KL}})}^{-1}\{\alpha_{\K_n}({\bf X})\}\right]\ 
\end{eqnarray*}
is positive, where the  weight $\alpha_{\K_n}({\bf X})$ is chosen according to procedure $P_1$ (Bonferroni) or a slight variation of $P_2$ (Monte-Carlo). All the results stated for $T^{(1)}_{\alpha}$ and $T^{(2)}_{\alpha}$ are still valid with $T_{\alpha,\cal{V}}$. The extension to affine subspaces $\cal{V}$ is also possible.

\medskip
The power of $T_{\alpha}^{(1)}$ has been analysed over the collection of ellipsoids $\mathcal{E}_a(R)$. The considered ellipsoids describing the nonparametric alternatives are determined by the principal directions $(V_j)_{j\geq 1}$, which are generally unknown. In fact, for some functions  $\theta$ that are well represented by a prescribed basis  (as wavelet, spline or Fourier basis) and whose expansion in the eigenfunction basis decreases slowly, projecting the data onto the Karhunen-Lo\`eve expansion is not necessarily best suited. Alternatively, one can adopt a similar approach in the context of a prescribed basis (as wavelet, spline or Fourier basis) instead of the eigenfunctions basis discussed above. The size and the power of the corresponding procedures are in fact easier to derive than for a Karhunen-Lo\`eve approach as we do not have to control the randomness of the basis.  We refer for instance to \cite{baraud03} for such results in a fixed design regression problem. As 
$\theta$ is unknown, the best choice of basis (prescribed or estimated by PCA) is also unknown. A solution is to combine testing procedures based on different basis.




\section{Main proofs}\label{section_proofs}
	
In this section, we emphasize the core of the proofs. Arguments based on perturbation theory are introduced in the next section. All the technical and side results are postponed to Appendix \ref{section_proof_minimax}--\ref{section_technique}.

\subsection{Additional notations}

Given any integer $k<\mathrm{Rank}(\Gamma)$, we recall that  $\Gamma_k=\sum_{j=1}^k\lambda_i V_j\otimes V_j$, where $\otimes$ stands for the tensor product. Similarly, $\widehat{\Gamma}_{n,k}:=\sum_{j=1}^k\widehat{\lambda}_i \widehat{V}_j\otimes \widehat{V}_j$ denotes its empirical counterpart. For any  $k<\mathrm{Rank}(\Gamma)$, we note $\Pi_k$ the orthogonal projection in $\H$ onto the space spanned by $V_j$, $j=1,\ldots, k$, while $\widehat{\Pi}_k$ stands for the orthogonal projection onto the space spanned by $\widehat{V}_j$, $j=i,\ldots, k\wedge \mathrm{Rank}(\widehat{\Gamma}_n)$.

In order to translate the definition of the testing procedure into functional data analysis framework, we shall use   
 $\Delta=\mathbb{E}(\langle X,.\rangle Y)$. We note $\Delta_n=\sum_{i=1}^n\langle X_i,.\rangle {\bf Y}_i/n$ its empirical counterpart. For any $k\leq \mathrm{Rank}(\Gamma)$, we note
$A_k= \sum_{j=1}^k \lambda_j^{-1/2}\langle V_j,.\rangle V_j$ and $\widehat{A}_{k}= \sum_{j=1}^{k\wedge \mathrm{Rank}(\widehat{\Gamma}_n)} \widehat{\lambda}_j^{-1/2}\langle\widehat{V}_j,.\rangle \widehat{V}_j$ its empirical counterpart.

Let $S$ be a bounded linear operator on the Hilbert space $\H$. The corresponding operator norm will be denoted $\left\|\cdot\right\|_{\infty}$ where $\left\|S\right\|_{\infty}=\sup_{x\in \bf{B}\left(0,1\right)}\left\|S\left(x\right)\right\| $ and $\bf{B}\left(0,1\right)$ stands for the unit ball of $\H$.
Let $T$ be a Hilbert-Schmidt operator.
$\left\Vert \cdot\right\Vert _{HS}$ denotes the Hilbert-Schmidt norm and
$\mathrm{tr}$ stands for the classical trace (defined for trace-class
operators). We recall that $\left\Vert T\right\Vert _{HS}^{2}=\mathrm{tr}\left(  T^{\ast
}T\right)  $.

In the sequel, we note $\bar{\chi}_{k}(u)$ the probability that a $\chi^2$ variable with $k$ degrees of freedom is larger than $u$, while $\bar{\chi}_k^{-1}(u)$ denotes the  $1-u$ quantile of a $\chi^2$ random variable.

\subsection{Connection between $\phi_k({\bf Y},{\bf X})$ and the procedure of Cardot et al.~\cite{cardot_test}}\label{section_proof_lien_cardot}
In fact, the numerator of the statistic $\phi_k$ is exactly the same as the test statistic $\|\sqrt{n}\widehat{A}_{k}\Delta_n\| ^2$ introduced by Cardot et al.~\cite{cardot_test}, that is:
\begin{equation}\label{definition_phi_KL_alternative}
\phi_k({\bf Y},{\bf X})=\frac{\|\widehat{\Pin}_k {\bf
Y}\|^2_n}{\|{\bf Y}-\widehat{\Pin}_k {\bf Y}\|_n^2/(n-\hat{k}^{KL})}=
\frac{\|\sqrt{n}\widehat{A}_{k}\Delta_n\| ^2}{\|{\bf Y}
-\widehat{\Pin}_k {\bf
Y}\|^2_n/(n-\hat{k}^{KL})} \ . 
\end{equation}

\begin{proof}[Proof of Equation (\ref{definition_phi_KL_alternative})]
Consider the least-squares $\widehat{\theta}_k$ estimator of $\theta$ in the space generated by $\widehat{V}_j$, $j=1,\ldots, \hat{k}^{KL}$. It follows that $\|\widehat{\Pin}_k{\bf Y}\|_n^2 =
n\langle\widehat{\theta}_k,\widehat{\Gamma}_n\widehat{\theta}_{k}\rangle $.
Since $\widehat{\theta}_k= \widehat{\Gamma}_{n,k}^{-}\Delta_n$ where $\widehat{\Gamma}_{n,k}^{-}$ is the Moore-Penrose pseudo-inverse of $\widehat{\Gamma}_{n,k}$, we obtain
$$
\|\widehat{\Pin}_k{\bf Y}\|_n^2 =  n \langle\widehat{\Gamma}_{n,k}^{-}\Delta_n,\widehat{\Gamma}_n\widehat{\Gamma}_{n,k}^{-}\Delta_n\rangle 
=  n\langle\widehat{A}_{k}\Delta_n,\widehat{A}_{k}\widehat{\Gamma}_n\widehat{\Gamma}_{n,k}^{-}\Delta_n\rangle 
=  n\|\widehat{A}_{k}\Delta_n\| ^2\  .$$
\end{proof}

\subsection{Proof of the type I error bounds}

We first prove Propositions \ref{prop_niveau_gaussien_T.alpha.k} and \ref{prop_niveau_gaussien}. Afterwards, we derive Theorem  \ref{thrm_niveau}. Finally,  we explain how to adapt the arguments for Theorem \ref{thrm_niveau_T.alpha.k}.

\begin{proof}[Proof of Propositions \ref{prop_niveau_gaussien_T.alpha.k} and  \ref{prop_niveau_gaussien}]
Let us assume that $\epsilon$ follows a Gaussian distribution and that $\theta=0$.
Conditionally on ${\bf X}$, the statistic $\phi_{k}({\bf Y},{\bf X})/\hat{k}$ defined in (\ref{definition_KL_test}) follows a Fisher distribution with $(\hat{k},n-\hat{k})$ degrees of freedom.  Hence, conditionally on ${\bf X}$, the test $T_{\alpha,k}$ has a size exactly $\alpha$. Conditionally on ${\bf X}$,  $T_{\alpha}^{(1)}$ is a Bonferroni procedure of Fisher statistics and its size is smaller than $\alpha$. Reintegrating with respect to  ${\bf X}$, we derive that the size of $T_{\alpha}^{(1)}$ is smaller than $\alpha$. 
Let us turn to the second result. The quantity  $q_{{\bf X},\alpha}$ satisfies 
$$\mathbb{P}_{0}\left(\left. \sup_{k \in \mathcal{K}_n} \left\{  \frac{(n-\hat{k}) \| 
      \widehat{\Pin}_{k} \boldsymbol{\epsilon} \|_n^2 }{ \hat{k}\|\boldsymbol{\epsilon} -  
  \widehat{\Pin}_{k}\boldsymbol{\epsilon} \|_n^2 }  - 
\bar{\F}^{-1}_{\hat{k}, n-\hat{k}}\left(q_{{\bf X},\alpha}\right)\right\} >0 \right|{\bf X}    \right)= \alpha\ ,$$ 
which implies that $\mathbb{P}_{0}(T_{\alpha}^{(2)}|{\bf X})=\alpha \quad {\bf X}$ a.s.
\end{proof}

\begin{proof}[Proof of Theorem \ref{thrm_niveau}]

First, we state that $\hat{k}^{KL}=k$ with large probability.
\begin{lemma}\label{lemme_convergence_valeurs_propres_empirique}
Consider the event $\mathcal{A}_n$ defined by 
\begin{eqnarray}
\mathcal{A}_{n}  &  =&\left\{  \sup_{1\leq j\leq \bar{k}_{n}}\frac{\left\vert
\widehat{\lambda}_{j}-\lambda_{j}\right\vert }{\min\left\{  \lambda_{j}-\lambda_{j+1},\lambda_{j-1}-\lambda_{j}\right\}}\geq 1/2\right\}\ . \label{definition_an}
\end{eqnarray}
 Under Assumptions ${\bf B.2}$ and ${\bf B.3}$, we have 
\begin{equation}\label{eq_controle_an}
 \mathbb{P}(\mathcal{A}_n)\leq C(\gamma)\left[\frac{\bar{k}^3_n\log^{2}(\bar{k}_n\vee e)}{n}\right]\leq C(\gamma)\frac{1}{\log^2(n)}\ ,
\end{equation}
where $\gamma$ is a positive constant involved in Assumption ${\bf B.2}$.
\end{lemma}
This result, proved in Appendix \ref{section_perturb_technical}, relies on the perturbation theory of random operators. Observe that under the event $\overline{\mathcal{A}}_n$, we have $\hat{k}^{KL}=k$ for all $k\leq \bar{k}_n$. Consequently, we can replace $\hat{k}^{KL}$ by $k$ in the definition of the test statistic up to an event of probability less than $C(\gamma)/\log(n)$.
In the sequel, we use the alternative expression (\ref{definition_phi_KL_alternative}) of $\phi_{k}$ and we replace $\hat{k}^{KL}$ by $k$.
The proof is split into three main lemmas \ref{lemma_1} - \ref{lemma_3}. The first lemma, states that $\|\sqrt{n}A_k\Delta_n\| ^2/\sigma^2$ behaves like a $\chi^2$ distribution. Its proof (Appendix \ref{section:berry_essen}) relies on a multivariate Berry-Esseen theorem. The second lemma, which tells us that $\|\sqrt{n}A_k\Delta_n\| ^2/\sigma^2$ is close to $\|\sqrt{n}\widehat{A}_{k}\Delta_n\| ^2/\sigma^2$ is proved below. The third lemma, proved in Appendix \ref{section_technique}, states that $\|{\bf Y}-\widehat{\Pin}_k {\bf Y}\|_n^2/n$ concentrates well around $\sigma^2$. 

\begin{lemma}\label{lemma_1} Assume that ${\bf B.1}$ and ${\bf B'.3}$ hold.
For any $k\geq 1$ and any $x>0$, we have 
\begin{eqnarray*}
|\mathbb{P}\left(\|\sqrt{n}A_k\Delta_n\| ^2\geq
x\right)- \bar{\chi}_k(x/\sigma^2)|&\leq& C
\frac{k^{3/2}}{\sqrt{n}}\frac{\mathbb{E}\left[\epsilon^4\right]^{3/4}}{\sigma^3}
\sup_ { 1\leq j\leq k}\mathbb{E}\left[(\eta^{(j)})^4\right]^{3/4} \leq \frac{C}{\log^2 (n)} \ ,
\end{eqnarray*}
uniformly over all $k\leq \bar{k}_n$.
\end{lemma}

\begin{lemma}\label{lemma_2}
Assume that ${\bf B.1}$--${\bf B'.3}$ hold. Writing  $x_{n,k}=1/(k\log^2(n))$, we have for all  $k\leq \bar{k}_n$, and all $n\geq 5$,
\begin{equation}\label{eq:upper_boundAkAchapeauk}
 \mathbb{P}\left[\|\sqrt{n}\widehat{A}_{k}\Delta_n\| ^2 \geq (1- x_{n,k})^{-1}\|\sqrt{n}A_k\Delta_n\| ^2\right]\leq \mathbb{P}\left[\mathcal{A}_n\right]+  \frac{C(\gamma)}{\log^2 (n)}+\left(\frac{\sqrt{e}}{\log(n)}\right)^k\ . 
\end{equation}
\end{lemma}

\begin{lemma}\label{lemma_3}
Uniformly over all $k\leq \bar{k}_n$, we have
\begin{eqnarray*}
 \mathbb{P}\left[\left|\frac{\|{\bf Y}-\widehat{\Pin}_k {\bf Y}\|_n^2}{n\sigma^2} -
1\right|\geq \frac{k\log^2 (n)}{n}+ 8\sqrt{\frac{\log\log n}{n}}\right]\leq \frac{C}{\log^2(n)}+ \frac{C'}{\sqrt{n}}
\ .
\end{eqnarray*}
\end{lemma}
Let us upper bound the rejection probability due to the statistic $\phi_k$ 
\begin{eqnarray*}
\nonumber
\mathbb{P}\left[\phi_k\left({\bf Y},{\bf X}\right)\geq k\bar{\F}_{k,n-k}^{-1}\left(\alpha/|\mathcal{K}_n|\right)\right]&\leq &
\mathbb{P}\left[\frac{\|\sqrt{n}\widehat{A}_{k}\Delta_n\| ^2}{\|{\bf
Y}
-\widehat{\Pin}_k {\bf Y}\|_n^2/n}\geq 
k\bar{\F}_{k,n-k}^{-1}\left(\alpha/|\mathcal{K}_n|\right)\right]
\end{eqnarray*}
by the three following probabilities
\begin{eqnarray*}
\lefteqn{\mathbb{P}\left[\frac{\|\sqrt{n}A_k\Delta_n\| ^2}{(1-x_{n,k})\sigma^2}\geq k\left(1- 8\sqrt{\frac{\log\log(n)}{n}}-\frac{k\log^2 (n)}{n}
\right)\bar{\F}_{k,n-k}^{-1}\left(\alpha/|\mathcal{K}_n|\right)\right]}&&\\&& +\mathbb{P}\left[\|\sqrt{n}\widehat{A}_{k}\Delta_n\| ^2 \geq (1- x_{n,k})^{-1}\|\sqrt{n}A_k\Delta_n\| ^2\right] \\ &&+ 
\mathbb{P}\left[\left|\frac{\|{\bf Y}-\widehat{\Pin}_k {\bf Y}\|_n^2}{n\sigma^2} -
1\right|\geq \frac{k\log^2 (n)}{n} + 8\sqrt{\frac{\log\log(n)}{n}}\right] \ .
\end{eqnarray*}
Gathering the above results, we obtain that this probability is upper bounded by 
\begin{equation} \label{inegalite_finale}
\frac{C(\gamma)}{\log^2(n)}+\left(\frac{\sqrt{e}}{\log(n)}\right)^k
+ \bar{\chi}_k\left[k\left(1-8\sqrt{\frac{\log\log n}{n}}-\frac{k\log^2 (n)}{n}-x_{n,k}\right)\bar{\F}_{k,n-k}^{-1}\left(\alpha/|\mathcal{K}_n|\right)\right]\ ,
\end{equation}
uniformly over all $k\leq \bar{k}_n$.

\begin{lemma}\label{lemma_bidouillage_proba} Writing $t= 8\sqrt{\frac{\log\log n}{n}}+\frac{k\log^2 (n)}{n}+\frac{1}{k\log^2(n)}$,  we have 
 for $n$ larger than some numerical constant
\begin{eqnarray*}
\bar{\chi}_k\left[k(1- t)
\bar{\F}_{k,n-k}^{-1}\left(\alpha/|\mathcal{K}_n|\right)\right]\leq    \frac{\alpha}{|\mathcal{K}_n|} \left(1+\frac{C(\alpha)}{\log(n)}\right)\ .
\end{eqnarray*}
\end{lemma}
The proof of this technical lemma is postponed to Appendix \ref{section_technique}. We conclude by combining \eqref{inegalite_finale} with Lemma \ref{lemma_bidouillage_proba} and taking an union bound over all $k\in\mathcal{K}_n$ (recall that $|\mathcal{K}_n|\leq \log(n)$).
%
\end{proof}

\begin{proof}[Proof of Theorem \ref{thrm_niveau_T.alpha.k}]
Define $t= 8\sqrt{\log\log n/n}+k\log^2(n)/n+1/k\log^2(n)$. 
Gathering Lemmas \ref{lemma_1}, \ref{lemma_2}, and \ref{lemma_3} as in the proof of Theorem \ref{thrm_niveau} and relying an Condition ${\bf B.3}$, we derive an upper bound analogous to \eqref{inegalite_finale}
\begin{equation*}
\mathbb{P}\left[\phi_k({\bf Y},{\bf X})\geq 
k\bar{\F}_{k,n-k}^{-1}\left(\alpha\right)\right] \leq \bar{\chi}_k\left[k\left(1-t\right)\bar{\F}_{k,n-k}^{-1}\left(\alpha\right)\right]+
\frac{C(\gamma)}{\log(n)}\ ,
\end{equation*}
 for $n$ large enough. Applying the following inequality (proved in Appendix \ref{section_technique}) allows us to conclude.
\begin{lemma}\label{lemma_bidouillage_proba2}
  For $n$ larger than some numerical constant, we have
\begin{equation*}
\bar{\chi}_k\left[k\left(1-t\right)\bar{\F}_{k,n-k}^{-1}\left(\alpha\right)\right]\leq    \alpha \left(1+\frac{C(\alpha)}{\log(n)}\right)\ .
\end{equation*}
\end{lemma}

\end{proof}

\begin{proof}[Proof of Lemma \ref{lemma_2}]
From $\left\Vert b\right\Vert ^{2}-\left\Vert a\right\Vert ^{2}=2\left\langle
a,b-a\right\rangle +\left\Vert b-a\right\Vert ^{2}$, we get%
\[\frac{\left\Vert b\right\Vert ^{2}-\left\Vert a\right\Vert
^{2}}{\left\Vert a\right\Vert ^{2}}\leq\frac{\left\Vert
b-a\right\Vert }{\left\Vert a\right\Vert }\left(  2+\frac{\left\Vert
b-a\right\Vert }{\left\Vert a\right\Vert }\right)\ .\]
Since $x_{n,k}< 1$ for $n\geq 3$, it follows that 
\begin{eqnarray*}
 \mathbb{P}\left[ \left\Vert \sqrt{n}  \widehat{A}_{k}%
\Delta_{n}\right\Vert^2    \geq \frac{\left\Vert \sqrt{n}A_{k}\Delta
_{n}\right\Vert^2}{1-x_{n,k}} \right]\leq  \mathbb{P}\left[  \frac{\left\Vert \sqrt{n}\left(  \widehat{A}_{k}%
-A_{k}\right)  \Delta_{n}\right\Vert  }{\left\Vert \sqrt{n}A_{k}\Delta
_{n}\right\Vert }\geq 
 \frac{x_{n,k}}{4}\right]\nonumber\\ 
\leq \mathbb{P}\left[  \left\Vert \sqrt{n}\left(  \widehat{A}_{k}%
-A_{k}\right)  \Delta_{n}\right\Vert\geq \frac{\sigma\sqrt{k}x_{n,k}}{4\log(n)}
 \right]
+\mathbb{P}\left[
\left\Vert \sqrt{n}A_{k}\Delta_{n}\right\Vert  \leq\frac{\sqrt{k}\sigma}{\log(n)%
}\right]\ .
\end{eqnarray*}
By Lemma 11.1 in \cite{V10}, we know that for any $0<x<1$ and any integer $d\geq 1$, $\mathbb{P}\left[\chi^2(d)\leq de^{-1}x^{2/d}\right]\leq x$.
We get from Lemma \ref{lemma_1} and the last deviation inequality  that
$$\mathbb{P}\left[
\left\Vert \sqrt{n}A_{k}\Delta_{n}\right\Vert  \leq\frac{\sigma\sqrt{k}}{\log(n)%
}\right]\leq \frac{C}{\log^2 (n)}+\left(\frac{\sqrt{e}}{\log(n)}\right)^k\ ,$$
uniformly over all $k\leq \bar{k}_n$. Let us turn to the other term. By Markov inequality and by definition of $x_{n,k}$, the first probability $\mathbb{P}[  \Vert \sqrt{n}(  \widehat{A}_{k}%
-A_{k})  \Delta_{n}\Vert\geq \frac{\sigma\sqrt{k}x_{n,k}}{4\log(n)} ]$ is smaller than 
\begin{eqnarray*}
\frac{16k\log^6(n) }{\sigma^2}\mathbb{E}\left[\left\Vert \sqrt{n}\left(
\widehat{A}_{k}-A_{k}\right)  \Delta_{n}\right\Vert  ^{2}\mathbf{1}_{\overline{\mathcal{A}}_n}\right]+ \mathbb{P}\left[\mathcal{A}_n\right]\ .
\end{eqnarray*}
In order to conclude, we only need to bound $\mathbb{E}[\Vert \sqrt{n}(  \widehat{A}%
_{k}-A_{k})  \Delta_{n}\Vert  ^{2}\mathbf{1}_{\overline{\mathcal{A}}_n}]$. If we prove 
\begin{equation}\label{eq:equation_objective}
 \mathbb{E}\left[\left\Vert \sqrt{n}\left(  \widehat{A}%
_{k}-A_{k}\right)  \right)\Delta_{n}\Vert^{2}\mathbf{1}_{\overline{\mathcal{A}}_n}\right]\leq C(\gamma)\left[\frac{\bar{k}_n^3\log^2(n)}{n} + \frac{\bar{k}_n}{\sqrt{n}}\right] \ ,
\end{equation}
then we get \[\mathbb{P}\left[  \left\Vert \sqrt{n}\left(  \widehat{A}_{k}%
-A_{k}\right)  \Delta_{n}\right\Vert\geq \frac{\sigma\sqrt{k}x_{n,k}}{4\log(n)} \right]\leq \frac{C(\gamma)}{\log^2(n)}+ \mathbb{P}\left[\mathcal{A}_n\right] \ ,\]
by Assumption ${\bf B'.3}$. Thus, it only remains to prove \eqref{eq:equation_objective}.

 Noticing that $\widehat{A}_{k}-A_{k}$ only
depends on the $X_{i}$'s, we derive that 
\begin{align*}
\mathbb{E}\left[\left\Vert \left(  \widehat{A}_{k}-A_{k}\right)  \Delta
_{n}\right\Vert  ^{2}\mathbf{1}_{\overline{\mathcal{A}}_n}\right]  &  =\frac{1}{n}\mathbb{E}\left[\left\Vert \left(  \widehat
{A}_{k}-A_{k}\right)  X_{1}{\boldsymbol{\varepsilon}}_{1}\right\Vert  ^{2}\mathbf{1}_{\overline{\mathcal{A}}_n}\right]\\ &=\frac{\sigma^{2}%
}{n}\mathbb{E}\left[\left\Vert \left(  \widehat{A}_{k}-A_{k}\right)   X_{1}%
\right\Vert  ^{2}\mathbf{1}_{\overline{\mathcal{A}}_n}\right]\\
&  =\frac{\sigma^{2}}{n}\mathbb{E}\left[  \left(\left\Vert \widehat{A}_{k}%
X_{1}\right\Vert  ^{2}+\left\Vert A_{k}X_{1}\right\Vert  ^{2}-2\left\langle
\widehat{A}_{k}X_{1},A_{k}X_{1}\right\rangle \right) \mathbf{1}_{\overline{\mathcal{A}}_n}\right]\ .
\end{align*}
We deal with each term separately:
\begin{eqnarray*}
\mathbb{E}\left[\Vert \widehat{A}_{k}X_{1}\Vert  ^{2}\mathbf{1}_{\overline{\mathcal{A}}_n}\right]&=&\mathbb{E}\left[
\mathrm{tr}\left(  \widehat{A}_{k}\left(  X_{1}\otimes X_{1}\right)
\widehat{A}_{k}\right)\mathbf{1}_{\overline{\mathcal{A}}_n} \right]  =\mathbb{E}\left[  \mathrm{tr}\left(
\widehat{A}_{k}\widehat{\Gamma}_n\widehat{A}_{k}\right) \mathbf{1}_{\overline{\mathcal{A}}_n}\right]  \\ & =&\mathbb{E}\left[
\mathrm{tr}\widehat{\Pi}_{k}\mathbf{1}_{\overline{\mathcal{A}}_n}\right] \leq k\mathbb{P}\left[\overline{\mathcal{A}}_n\right]\ ,\\ 
\mathbb{E}\left[\left\Vert A_{k}X_{1}\right\Vert  ^{2}\mathbf{1}_{\overline{\mathcal{A}}_n}\right]&=&\mathbb{E}\left[  \mathrm{tr}\left(
A_{k}\widehat{\Gamma}_nA_{k}\right) \mathbf{1}_{\overline{\mathcal{A}}_n}\right] \\& =&\mathbb{E}\left[\mathrm{tr}\left(  A_{k}
\Gamma  A_{k}\right)  \mathbf{1}_{\overline{\mathcal{A}}_n}\right]+\mathbb{E}\left[\mathrm{tr}\left(  A_{k}\left(\widehat{\Gamma}_n-
\Gamma\right)  A_{k}\right)  \mathbf{1}_{\overline{\mathcal{A}}_n}\right] \\ &=  &\mathbb{E}\left[\mathrm{tr}\Pi_{k}\mathbf{1}_{\overline{\mathcal{A}}_n}\right]- \mathbb{E}\left[\mathrm{tr}\left(  A_{k}\left(\widehat{\Gamma}_n-
\Gamma\right)  A_{k}\right)  \mathbf{1}_{\mathcal{A}_n}\right]\\ & \leq& k\mathbb{P}\left[\overline{\mathcal{A}}_n\right]+ \sqrt{\mathbb{E}\left[\mathrm{tr}^2\left(  A_{k}\left(\widehat{\Gamma}_n-
\Gamma\right)  A_{k}\right)\right]}\sqrt{\mathbb{P}\left[\mathcal{A}_n\right]}\ ,\\
\mathbb{E}\left[\left\langle \widehat{A}_{k}X_{1},A_{k}X_{1}\right\rangle \mathbf{1}_{\overline{\mathcal{A}}_n}\right]  &
=&\mathbb{E}\left[  \mathrm{tr}\left(  \widehat{A}_{k}\widehat{\Gamma}_nA_{k}\right)\mathbf{1}_{\overline{\mathcal{A}}_n}
\right]  =\mathbb{E}\left[\mathrm{tr}\left(  \Gamma_{k}^{-1/2}\widehat{\Gamma}_{n,k}^{1/2}\right)\mathbf{1}_{\overline{\mathcal{A}}_n}\right] \\ &=&k\mathbb{P}(\overline{\mathcal{A}}_n)+\mathbb{E}\left[\mathrm{tr}\left\{  \Gamma_{k}^{-1/2}\left(  \widehat{\Gamma
}_{n,k}^{1/2}-\Gamma_{k}^{1/2}\right) \mathbf{1}_{\overline{\mathcal{A}}_n} \right\}\right]\ .
\end{eqnarray*}
It follows that %
\begin{eqnarray}\nonumber
\mathbb{E}\left[\left\Vert \sqrt{n}\left(  \widehat{A}_{k}-A_{k}\right)  \Delta
_{n}\right\Vert  ^{2}\mathbf{1}_{\overline{\mathcal{A}}_n}\right]&\leq& 2\sigma^{2}\mathbb{E}\left[\mathrm{tr}\left\{  \Gamma_{k}%
^{-1/2}\left(  \Gamma_{k}^{1/2}-\widehat{\Gamma}_{n,k}^{1/2}\right) \right\}\mathbf{1}_{\overline{\mathcal{A}}_n} \right]\\&&+ \sigma^2\sqrt{\mathbb{E}\left[\mathrm{tr}^2\left(  A_{k}\left(\widehat{\Gamma}_n-
\Gamma\right)  A_{k}\right)\right]}\sqrt{\mathbb{P}\left[\mathcal{A}_n\right]}\ .\label{eq_defi_esp}
\end{eqnarray}

\begin{lemma}\label{lemlem1}
Under  Assumptions ${\bf B.1}$ and ${\bf B.2}$, we have for all $n\geq 1$,
\begin{equation}
\mathbb{E}\left[\mathrm{tr}\left\{ \Gamma_{k}^{-1/2}\left(  \Gamma_{k}%
^{1/2}-\widehat{\Gamma}_{n,k}^{1/2}\right)  \right\}\mathbf{1}_{\overline{\mathcal{A}}_n} \right]  \leq C(\gamma) \frac
{k^{3} [\log^2 (k)\vee 1]}{n}+C'\frac{k}{\sqrt{n}}\label{S2}%
\end{equation}
uniformly over all $k\leq \bar{k}_n$.
\end{lemma}
Lemma \ref{lemlem1} is the core argument to control the behavior of the statistic. Its proof relies on perturbation theory and is   postponed to Section \ref{section_perturb}. Let us compute the last term
\begin{eqnarray*}
\mathbb{E}\left[\mathrm{tr}^2\left(  A_{k}\left(\widehat{\Gamma}_n-
\Gamma\right)  A_{k}\right)\right] &= &\mathbb{E}\left[\left(\sum_{j=1}^k\left(\sum_{i=1}^n\frac{[\boldsymbol{\eta}_i^{(j)}]^2}{n}-1\right)\right)^2\right] \\
&\leq & \frac{k^2}{n}\sup_{j\geq 1}\var\left[(\eta^{(j)})^2\right]\leq C\frac{k^2}{n}\ ,
\end{eqnarray*}
by Assumption ${\bf B.1}$. Combining this bound with (\ref{eq_controle_an}), we get 
\begin{equation}\label{eq_fin_esp}
\sqrt{\mathbb{E}\left[\mathrm{tr}^2\left(  A_{k}\left(\widehat{\Gamma}_n-
\Gamma\right)  A_{k}\right)\right]}\sqrt{\mathbb{P}\left[\mathcal{A}_n\right]}\leq C(\gamma)\frac{\bar{k}_n^{5/2}\log(\bar{k}_n\vee e)}{n}\ . 
\end{equation}
Gathering Lemma  \ref{lemlem1} with  (\ref{eq_defi_esp}), and (\ref{eq_fin_esp}) allows us to prove \eqref{eq:equation_objective}. 
\end{proof}

\subsection{Proofs of the type II error bounds}

\begin{proof}[Proof of Proposition \ref{prop_puissance_comparaison}]
This proof follows the same steps as the proof of Proposition 3.2 in~\cite{VV10}.
\end{proof}

We first derive Theorem  \ref{thrm_power_KL_fixed_multiple} and then explain how to adapt the arguments for Theorem \ref{thrm_power_KL_fixed}.

\begin{proof}[Proof of Theorem \ref{thrm_power_KL_fixed_multiple}]
Arguing as in the beginning of the proof of Theorem \ref{thrm_niveau}, we can replace $\hat{k}^{KL}$ by $k$ in the definition of the statistic (\ref{definition_procedure_KL}). 
Consider some $k\in\mathcal{K}_n$ and take $n\geq 8$, the numerator of $\phi_{k}({\bf Y},{\bf X})$ (\ref{definition_phi_KL_alternative}) is lower bounded as follows 
\begin{equation*}
\|\sqrt{n}\widehat{A}_{k}\Delta_n \| ^2\geq \|\sqrt{n}A_{k}\Delta_n \|^2 \left[1-(\sqrt{k}\log(n))^{-1}\right] - \sqrt{k}\log(n)\|\sqrt{n}(A_k -\widehat{A}_{k})\Delta_{n} \| ^2\ ,
\end{equation*}
since $2ab\leq a^2+b^2$. Observe that $\Delta_n=  \widehat{\Gamma}_n\theta + \Delta_{n,1}$, where $\Delta_{n,1}=\sum_{i=1}^n\langle X_i,.\rangle \epsilon_i/n$.
 The proof is based on the two main following lemmas.
\begin{lemma}\label{lemma_A1}
For any $\beta\in (0,1)$, we have
\begin{eqnarray*}
\left\|\sqrt{n}A_k\Delta_n\right\| ^2\geq
k\sigma^2+ \frac{n}{5}\|\Gamma^{1/2}_k\theta\| ^2-2\sigma^2\sqrt{k\log\left(\frac{2}{\beta}\right)}-10\sigma^2\log\left(\frac{2}{\beta}\right)\ ,
\end{eqnarray*}
with probability larger than  $1- \beta/2-C/\log(n)$ uniformly over all $k\leq \bar{k}_n$.
\end{lemma}

\begin{lemma}\label{lemme_principal_puissance} Assume that ${\bf B.1}$--${\bf B'.3}$ hold. For any   $n\geq 1$ we have
\begin{eqnarray*}\nonumber
\mathbb{P}\left[\|\sqrt{n}(\widehat{A}_k-A_k)\widehat{\Gamma}_n\theta\| \geq \frac{\sqrt{n}\|\Gamma^{1/2}\theta\|}{k^{1/4}\log(n)} \right]&\leq& \frac{C(\gamma)}{\log(n)}
\\
\mathbb{P}\left[\|\sqrt{n}(\widehat{A}_k-A_k)\Delta_{n,1}\| \geq \frac{\sigma}{\log(n)} \right]&\leq &\frac{C(\gamma)}{\log(n)}  \ ,
 \end{eqnarray*}
uniformly over all $k\leq \bar{k}_n$.
\end{lemma}
Lemma \ref{lemma_A1} is based on a multivariate Berry-Esseen inequality and is proved in Appendix \ref{section:berry_essen}. The second lemma proceeds from the same kind of arguments as Lemma \ref{lemma_2}. Thus, its proof is postponed to Appendix \ref{section_technique}. We get by gathering Lemmas \ref{lemma_A1} and \ref{lemme_principal_puissance} and since $\sqrt{k}\log(n)\geq 2$ for $n\geq 8$,
\begin{eqnarray*}
\lefteqn{\|\sqrt{n}\widehat{A}_{k}\Delta_n \| ^2 \geq k\sigma^2  - 3\sigma^{2}\frac{\sqrt{k}}{\log(n)}- 2n\frac{\|\Gamma^{1/2}\theta\|^2}{\log(n)}} \\ &&+ C _1 n\|\Gamma_k^{1/2}\theta\| ^2- C_2\sigma^2\left(\sqrt{k\log\left(\frac{2}{\beta}\right)}+ \log\left(\frac{2}{\beta}\right)\right)\ ,
 \end{eqnarray*}
 with probability larger than $1- \beta/2 -C(\gamma)/\log(n)$. Next, we use a rough control of the denominator, proved in Section \ref{section_technique}.

\begin{lemma}[Control of the denominator]\label{lemma_A3} We have
\begin{eqnarray*}
 \frac{\|{\bf Y}-\widehat{\Pin}_k{\bf Y}\|_n^2}{n-k} \leq \sigma^2\left[1+ C\left(\frac{k}{n}+\sqrt{\frac{\log(n)}{n}}\right)\right]+ C'\|\Gamma^{1/2}\theta\|^2/\beta\ ,
\end{eqnarray*}
with probability larger than $1-2/\log(n)-\beta/4$.
\end{lemma}
Since  $\sqrt{\log(2/\beta)}\geq 1/\log(n)$ and $C_1\geq 2/\log(n)$ for $n$ large enough,
we derive from the previous results that with probability larger than $1-3\beta/4-C(\gamma)/\log(n)$, the statistic 
$\phi_k({\bf Y},{\bf X}) $ is lower bounded by
\begin{equation}\label{eq:lower_phi_k_alternative}
  \frac{k\sigma^2+n\|\Gamma_k^{1/2}\theta\| ^2nC'_1  - C'_2\sigma^2\left(\sqrt{k\log(1/\beta)}+ \log(1/\beta)\right)- 2\frac{n}{\log(n)} \|(\Gamma^{1/2}-\Gamma_k^{1/2})\theta\|^2  }{\sigma^2\left[1+ C\left(\frac{k}{n}+\sqrt{\frac{\log(n)}{n}}\right)\right]+ C'\|\Gamma^{1/2}\theta\|^2/\beta}\ . 
\end{equation}
By Lemma 1 in \cite{baraud03}, we can upper bound the quantile of Fisher distribution
\begin{equation}
 k\bar{\F}^{-1}_{k,n-k}(\alpha/|\mathcal{K}_n|)\leq k + C \left[\sqrt{k\log\left(\frac{|\mathcal{K}_n|}{\alpha}\right)}+\log\left(\frac{|\mathcal{K}_n|}{\alpha}\right)\right]\ ,\label{eq:quantile}
\end{equation}
since we assume that $\log(|\mathcal{K}_n|/\alpha)\leq \log(n)+\log(1/\alpha)\leq 2\sqrt{n}$.  Comparing the lower bound \eqref{eq:lower_phi_k_alternative} with \eqref{eq:quantile} allows us to conclude. We refer to Appendix \ref{section_technique} for the details.  
\end{proof}

\begin{proof}[Proof of Theorem \ref{thrm_power_KL_fixed}]
We have shown in the last proof that 
$\phi_k({\bf Y},{\bf X}) $ is lower bounded by \eqref{eq:lower_phi_k_alternative}
with probability larger than $1-3\beta/4-C(\gamma)/\log(n)$. By Lemma 1 in \cite{baraud03}, we upper bound the quantile of Fisher distribution
$$ k\bar{\F}^{-1}_{k,n-k}(\alpha)\leq k + C \left[\sqrt{k\log\left(1/\alpha\right)}+\log\left(1/\alpha\right)\right]\ ,$$
since $\log(1/\alpha)\leq \sqrt{n}$. Comparing these two bounds leads us to the desired result.
\end{proof}

\section{Arguments based on perturbation theory}\label{section_perturb}

\subsection{Preliminary facts}\label{section_prelim_perturbation}
 
Roughly speaking, several results
mentioned below are based on an extension of the classical residue formula on the complex
plane (see Rudin~\cite{rudin}) to analytic functions still defined on the
complex plane but with values in the space of operators. We refer to Dunford and Schwartz~\cite[Chapter VII.3]{dunford} or to Gohberg et al.~\cite{gohberg,gohberg2} for an introduction to functional calculus for operators related with Riesz integrals.  Let us denote  $\mathcal{B}_{j}$ the oriented circle of the complex plane
with center $\lambda_{j}$ and radius $\delta_{j}/2$ where $\delta_j$ is defined by
\begin{eqnarray}\label{definition_delta_j}
\delta_{j}%
=\min\left\{  \lambda_{j}-\lambda_{j+1},\lambda_{j-1}-\lambda_{j}\right\}\ .
\end{eqnarray}
The open domain whose boundary is $\mathcal{C}_{k}:= \cup_{j=1}^{k}\mathcal{B}_{j}$ is not connected but we
can apply the functional calculus for bounded operators (see Dunford and Schwartz~\cite[Section VII.3, Definitions 8 and 9]{dunford}
). Using this formalism it is easy to
prove the following formulas :%
\begin{align}
\Pi_{k}   =\frac{1}{2\pi\iota}\int_{\mathcal{C}_{k}}\left(
zI-\Gamma\right)  ^{-1}dz\quad \text{ and }\quad 
\Gamma^{1/2}_k   =\frac{1}{2\pi\iota}\int_{\mathcal{C}_{k}}z^{1/2}\left(
zI-\Gamma\right)  ^{-1}dz.\nonumber
\end{align}

The same is true with the random operator $\widehat{\Gamma}_n$, but the contour $\mathcal{C}%
_{k}$ must be replaced by its random counterpart $\widehat{\mathcal{C}}%
_{k}=\bigcup_{j=1}^{k\wedge \mathrm{Rank}(\widehat{\Gamma}_n)}\widehat{\mathcal{B}}_{j}$ where each $\widehat
{\mathcal{B}}_{j}$ is a random ball of the complex plane with center
$\widehat{\lambda}_{j}$ and a radius $\widehat{\delta}_{j}/2=\min\{  \widehat{\lambda}_{j}-\widehat{\lambda}_{j+1},\widehat{\lambda}_{j-1}-\widehat{\lambda}_{j}\}$. We start with some lemmas.

\begin{lemma}
\label{DomEigen}Assume that for some $\gamma>0$,  the sequence $\left(  j\lambda_{j}\log^{1+\gamma
}(j\vee 2)\right)  _{j\in\mathbb{N}^*}$ decreases. Then, we have
\[
\sum_{j\geq1,j\neq k}\frac{\lambda_{j}}{\left\vert \lambda_{k}-\lambda
_{j}\right\vert }\leq C(\gamma)k\left[\log k\vee 1\right]\ .
\]
\end{lemma}
For any positive integer $j$, let us define the event \[\mathcal{E}_{j,n}:=\left\{  \sup_{z\in\mathcal{B}%
_{j}}\left\Vert \left(  zI-\Gamma\right)  ^{-1/2}\left(  \widehat{\Gamma}_{n}%
-\Gamma\right)  \left(  zI-\Gamma\right)  ^{-1/2}\right\Vert _{\infty}%
\geq1/2\right\}  .\]
\begin{lemma} \label{perturb1} 
Suppose that Assumption ${\bf B.1-2}$ holds. For any $j\geq 1$,
We have the two following bounds 
\begin{align*}
\mathbb{E}\sup_{z\in\mathcal{B}_{j}}\left\Vert \left(  zI-\Gamma\right)
^{-1/2}\left(  \widehat{\Gamma}_{n}-\Gamma\right)  \left(  zI-\Gamma\right)
^{-1/2}\right\Vert _{HS}^{2}  & \leq\frac{C(\gamma)}{n}\left[  j(\log j\vee 1)\right]  ^{2}\ ,\\
\mathbb{P}\left(  \mathcal{E}_{j,n}\right)    & \leq\frac{C(\gamma)}{n}\left[  j(\log
j\vee 1)\right]  ^{2}\ .%
\end{align*}
\end{lemma}
The proof of Lemma \ref{DomEigen} (resp. \ref{perturb1}) is postponed to Appendix \ref{section_technique} (\ref{section_perturb_technical}).
\subsection{Proof of Lemma \ref{lemlem1}}
In order to upper bound this expectation, we set $\widehat{\lambda}_j=0$ for any $j> \mathrm{Rank}(\widehat{\Gamma}_n)$. We have
\begin{align}
\mathrm{tr}\left[  \Gamma_{k}^{-1/2}\left(  \Gamma_{k}^{1/2}-\widehat{\Gamma
}_{n,k}^{1/2}\right)  1_{\overline{\mathcal{A}}_{n}}\right]    &
=k1_{\overline{\mathcal{A}}_{n}}-\sum_{j=1}^{k}\sum_{l=1}^{k}\sqrt
{\frac{\widehat{\lambda}_{j}}{\lambda_{l}}}\left\langle V_{l},\widehat{V}%
_{j}\right\rangle^{2} 1_{\overline{\mathcal{A}}_{n}}\nonumber\\
& \leq k1_{\overline{\mathcal{A}}_{n}}-\sum_{j=1}^{k}\sqrt{\frac
{\widehat{\lambda}_{j}}{\lambda_{j}}}\left\langle V_{j},\widehat{V}%
_{j}\right\rangle ^{2}1_{\overline{\mathcal{A}}_{n}}\nonumber\\
& \leq\sum_{j=1}^{k}\left(  1-\left\langle V_{j},\widehat{V}_{j}\right\rangle 
^{2}\right)  1_{\overline{\mathcal{A}}_{n}}+\sum_{j=1}^{k}\left\vert
\sqrt{\frac{\widehat{\lambda}_{j}}{\lambda_{j}}}-1\right\vert 1_{\overline
{\mathcal{A}}_{n}}\nonumber\\  
& \leq\sum_{j=1}^{k}\left(  1-\left\langle V_{j},\widehat{V}_{j}\right\rangle 
^{2}\right)  1_{\overline{\mathcal{A}}_{n}}+\sum_{j=1}^{k}\frac{|\widehat{\lambda}_j-\lambda_j|}{\lambda_j} 1_{\overline
{\mathcal{A}}_{n}}\ ,\nonumber%
\end{align}
where the last equation follows from the upper bound $|\sqrt{1+x}-1|\leq |x|$ for any $x\geq -1$. Observe that under the event $\overline{\mathcal{A}}_n$, $(\widehat{\lambda}_j-\lambda_j)/\lambda_j\leq 1/2$. Applying Lemma \ref{perturb1}, we obtain the following bound
\begin{eqnarray}
\lefteqn{\mathbb{E}\left[\mathrm{tr}\left[  \Gamma_{k}^{-1/2}\left(  \Gamma_{k}^{1/2}-\widehat{\Gamma
}_{n,k}^{1/2}\right)  1_{\overline{\mathcal{A}}_{n}}\right]\right]}\nonumber&&\\&  &\leq \sum_{j=1}^k
\mathbb{E}\left[\left(  1-\left\langle V_{j},\widehat{V}_{j}\right\rangle 
^{2}\right)  1_{\overline{\mathcal{A}}_{n}\cap\overline{\mathcal{E}}_{j,n}}\right]+\sum_{j=1}^{k}\mathbb{E}\left[\frac{|\widehat{\lambda}_j-\lambda_j|}{\lambda_j} 1_{\overline
{\mathcal{A}}_{n}\cap\overline{\mathcal{E}}_{j,n}}\right]\nonumber  + \sum_{j=1}^k \frac{3}{2}\mathbb{P}\left[\mathcal{E}_{j,n}\right]\\ 
&&\leq \sum_{j=1}^k
\mathbb{E}\left[\left(  1-\left\langle V_{j},\widehat{V}_{j}\right\rangle 
^{2}\right)  1_{\overline{\mathcal{A}}_{n}\cap\overline{\mathcal{E}}_{j,n}}\right]+\sum_{j=1}^{k}\mathbb{E}\left[\frac{|\widehat{\lambda}_j-\lambda_j|}{\lambda_j} 1_{\overline
{\mathcal{A}}_{n}\cap\overline{\mathcal{E}}_{j,n}}\right] + C(\gamma)\frac{k^3(\log^2(k)\vee 1)}{n}\ .\nonumber\\ \label{trace}
\end{eqnarray}
In the sequel,  $\pi_{j}$ stands for  the orthogonal projector associated to the single $j-th$ eigenvector $V_j$ while $\widehat{\pi}_{j}$ refers to its empirical counterpart. Applying functional calculus tools  for linear
operators, we get for any $1\leq j\leq k$
\begin{align*}
1-\left\langle V_{j},\widehat{V}_{j}\right\rangle  ^{2} &
=  \left\langle V_{j},V_{j}\right\rangle  ^{2}-\left\langle
V_{j},\widehat{V}_{j}\right\rangle  ^{2}  =\left\langle
\left( \pi_{j} -\widehat{\pi}_{j}\right)  V_{j},V_{j}\right\rangle  \\
&  =\frac{1}{2\pi \iota}\left[\int_{\mathcal{B}_{j}}\left\langle \left(  zI-\Gamma\right)  ^{-1}V_j,V_j\right\rangle dz- \int_{\hat{\mathcal{B}}_{j}}\left\langle\left(  zI-\widehat{\Gamma}
_{n}\right)  ^{-1}V_{j},V_{j}\right\rangle  dz\right]\ ,
\end{align*}
which looks like the definition of $\Pi_{k}$ given in the first paragraph of Section \ref{section_prelim_perturbation} (note that only the contour changed). Under the event $\overline{\mathcal{A}}_n$, $\widehat{\lambda}_j$ lies inside the circle $\mathcal{B}_j$. In fact,  $(zI-\widehat{\Gamma}_{n})^{-1}$ has only one pole inside the circle $\mathcal{B}_j$ at $z=\widehat{\lambda}_j$. As a consequence, we have almost surely
\[\int_{\hat{\mathcal{B}}_{j}}\left\langle\left(  zI-\widehat{\Gamma}
_{n}\right)  ^{-1}V_{j},V_{j}\right\rangle  dz\mathbf{1}_{\overline{\mathcal{A}}_n}= \int_{\mathcal{B}_{j}}\left\langle\left(  zI-\widehat{\Gamma}
_{n}\right)  ^{-1}V_{j},V_{j}\right\rangle  dz\mathbf{1}_{ \overline{\mathcal{A}}_n}\ ,\]
so that 
\begin{equation}\label{eq_integration}
\left(1-\left\langle V_{j},\widehat{V}_{j}\right\rangle  ^{2}\right)\mathbf{1}_{\overline{\mathcal{E}}_{j,n}\cap \overline{\mathcal{A}}_n}=	 \frac{1}{2\pi \iota}\int_{\mathcal{B}_{j}}\left\langle \left\{\left(  zI-\Gamma\right)  ^{-1}-\left(  zI-\widehat{\Gamma}_n\right)  ^{-1}\right\}V_j,V_j\right\rangle dz \mathbf{1}_{\overline{\mathcal{E}}_{j,n}\cap \overline{\mathcal{A}}_n}\ . 
\end{equation}
Working out  this integral, we get
\begin{eqnarray*}
 \lefteqn{\int_{\mathcal{B}_{j}}\left\langle \left\{\left(  zI-\Gamma\right)  ^{-1}-\left(  zI-\widehat{\Gamma}_n\right)  ^{-1}\right\}V_j,V_j\right\rangle dz}&&\\
&  =&-\int_{\mathcal{B}_{j}}\left\langle \left(  zI-\Gamma\right)  ^{-1}\left(
\widehat{\Gamma}_n-\Gamma\right)  \left(  zI-\Gamma\right)  ^{-1}V_{j},V_{j}%
\right\rangle  \\ 
&&-\int_{\mathcal{B}_{j}}\left\langle \left(  zI-\widehat{\Gamma}_n\right)
^{-1}\left(  \widehat{\Gamma}_n-\Gamma\right)  \left(  zI-\Gamma\right)  ^{-1}\left(
\widehat{\Gamma}_n-\Gamma\right)  \left(  zI-\Gamma\right)  ^{-1}V_{j},V_{j}%
\right\rangle  dz\ .
\end{eqnarray*}
The first term is $\int_{\mathcal{B}_j}(z-\lambda_j)^{-2}\langle (\widehat{\Gamma}_n-\Gamma)V_j,V_j\rangle dz$. Thus, it is null almost surely by the Cauchy integration theorem. Define 
$S_{n}\left(  z\right)  =\left(  zI-\Gamma\right)  ^{1/2}(
zI-\widehat{\Gamma}_n)  ^{-1}\left(  zI-\Gamma\right)  ^{1/2}$ and
$T_{n}\left(  z\right)  =\left(  zI-\Gamma\right)  ^{-1/2}(
\widehat{\Gamma}_n-\Gamma)  \left(  zI-\Gamma\right)  ^{-1/2} $. %
For any fixed $z$, we have $S_{n}\left(  z\right)
=\left[  I-T_{n}\left(  z\right)  \right]  ^{-1}$. Thus, it comes from (\ref{eq_integration}) that 
\begin{eqnarray}
 \lefteqn{\mathbb{E} \left[\left(1-\left\langle V_{j},\widehat{V}_{j}\right\rangle  ^{2}\right)\mathbf{1}_{\overline{\mathcal{E}}_{j,n}\cap \overline{\mathcal{A}}_n}\right]}\nonumber&&\\ &\leq &   \mathbb{E}\left[\left\vert\frac{1}{2\pi \iota}\int_{\mathcal{B}_{i}}\left\langle \left(
zI-\Gamma\right)  ^{-1/2}S_{n}\left(  z\right)  %
T_{n}^{2}\left(  z\right)  \left(  zI-\Gamma\right)  ^{-1/2}V_{j}%
,V_{j}\right\rangle  \mathbf{1}_{\overline{\mathcal{E}}_{j,n}\cap \mathcal{A}_n}dz\right\vert\right] \nonumber\\
&  \leq & C\delta_{j}\mathbb{E}\left[\sup_{z\in\mathcal{B}_{j}}\left\{  \left\Vert
S_{n}\left(  z\right)  \right\Vert _{\infty}\mathbf{1}_{\overline{\mathcal{E}}_{j,n}%
}\left\Vert T_{n}\left(  z\right)  \right\Vert _{\infty}^{2}\left\Vert \left(
zI-\Gamma\right)  ^{-1}\right\Vert _{\infty}\right\}\right]  \nonumber\\
&  \leq &C\delta_{j}\sup_{z\in\mathcal{B}_{j}}\left\Vert \left(  zI-\Gamma
\right)  ^{-1}\right\Vert _{\infty}\mathbb{E}\left[\sup_{z\in\mathcal{B}_{j}}\left\Vert T_{n}\left(  z\right)  \right\Vert _{\infty}^{2}\right] \leq C(\gamma)\frac{j^2(\log^2(j)\vee 1)}{n}\ ,\label{RF3}
\end{eqnarray}
since $\sup_{z\in\mathcal{B}_{j}}\left\Vert \left(  zI-\Gamma\right)
^{-1}\right\Vert _{\infty}\leq2\delta_{i}^{-1}$, $\sup_{z\in\mathcal{B}_{j}}\left\Vert S_{n}\left(
z\right)  \right\Vert _{\infty}\mathbf{1}_{\overline{\mathcal{E}}_{j,n}}\leq 2$ and $\mathbb{E}[\sup_{z\in\mathcal{B}_{i}}\left\Vert T_{n}\left(
z\right)  \right\Vert _{\infty}^{2}]\leq \frac{C(\gamma)}{n}j^2(\log^2(j)\vee 1)$  by Lemma \ref{perturb1}. Hence, we obtain an upper bound for the first term in (\ref{trace})
\begin{equation}\label{eq_majoration_principale_lemme}
\mathbb{E}\left[ \left(k-\sum_{j=1}^{k}\left\langle V_{j},\widehat{V}%
_{j}\right\rangle  ^{2}\right)\mathbf{1}_{\overline{\mathcal{A}}_n}\right] \leq C(\gamma)\frac{k^3(\log^2(k)\vee 1)}{n}\ .
\end{equation}

Turning to the second term in (\ref{trace}), we only provide a sketch of the proof since the approach is the same as the first term in (\ref{trace}).  We have
\[
\widehat{\lambda}_{j}-\lambda_{j}=\mathrm{tr}\left(  \widehat{\Gamma}_{n}\widehat{\pi
}_{j}-\Gamma\pi_{j}\right)  =\mathrm{tr}\left(  \widehat{\Gamma}_{n}\left(  \widehat
{\pi}_{j}-\pi_{j}\right)  \right)  +\mathrm{tr}\left(  \left(  \widehat{\Gamma}
_{n}-\Gamma\right)  \pi_{j}\right)\ ,
\]
so that 
\begin{equation}
 \frac{|\widehat{\lambda}_j-\lambda_j|}{\lambda_j}\mathbf{1}_{\overline{\mathcal{A}}_n\cap \overline{\mathcal{E}}_{j,n}}\leq \frac{\left|\mathrm{tr}\left(  \widehat{\Gamma}_{n}\left(  \widehat
{\pi}_{j}-\pi_{j}\right)  \right)\right|}{\lambda_j}\mathbf{1}_{\overline{\mathcal{A}}_n\cap \overline{\mathcal{E}}_{j,n}}+ \frac{\left|\mathrm{tr}\left(  \left(  \widehat{\Gamma}
_{n}-\Gamma\right)  \pi_{j}\right)\right|}{\lambda_j}\mathbf{1}_{\overline{\mathcal{A}}_n\cap \overline{\mathcal{E}}_{j,n}}\ .\label{eq_majoration_1_ratio_vp}
\end{equation}
The second term in this decomposition is bounded as follows 
\begin{equation}
\mathbb{E}\left[\frac{\left\vert \mathrm{tr}\left(  \left(  \widehat{\Gamma}_{n}%
-\Gamma\right)  \pi_{j}\right)  \right\vert}{\lambda_j}\mathbf{1}_{\overline{\mathcal{A}}_n\cap \overline{\mathcal{E}}_{j,n}}\right]\leq \mathbb{E}\left[\frac{\left\vert \left\langle
\left(  \widehat{\Gamma}_{n}-\Gamma\right)    V_{j}  ,V_{j}\right\rangle 
\right\vert}{\lambda_j}\right] \leq\frac{1}{\sqrt{n}}\ .\label{eq_majoration_2_ratio_vp}
\end{equation}
 We turn to $\mathbb{E}\left[
\left\vert \mathrm{tr}\left(  \widehat{\Gamma}_{n}\left(  \widehat{\pi}_{j}-\pi
_{j}\right)  \right)  \right\vert \mathbf{1}_{\overline{\mathcal{A}}_n\cap \overline{\mathcal{E}}_{j,n}}\right] $ and we use the same method as above for
bounding $(  1-\langle V_{j},\widehat{V}_{j}\rangle ^{2})  $. 
\begin{align*}
\lefteqn{\frac{1}{\lambda_j}\mathrm{tr}\left[  \widehat{\Gamma}_{n}\left(  \widehat{\pi}_{j}-\pi_{j}\right)
\right]  \mathbf{1}_{\overline{\mathcal{A}}_n\cap \overline{\mathcal{E}}_{j,n}}} &\\  & =\frac{1}{\lambda_j2\pi\iota}\mathrm{tr}\left[  \int_{\mathcal{B}_{j}}\widehat{\Gamma}_{n}\left(
zI-\widehat{\Gamma}_{n}\right)  ^{-1}\left(  \widehat{\Gamma}_{n}-\Gamma\right)  \left(
zI-\Gamma\right)  ^{-1}dz\right] \mathbf{1}_{\overline{\mathcal{A}}_n\cap \overline{\mathcal{E}}_{j,n}}  \\
& =\frac{1}{\lambda_j2\pi\iota}\mathrm{tr}\left[  \int_{\mathcal{B}_{j}}\left(  \widehat{\Gamma}_{n}\left(
zI-\widehat{\Gamma}_{n}\right)  ^{-1}-\Gamma\left(  zI-\Gamma\right)  ^{-1}\right)
\left(  \widehat{\Gamma}_{n}-\Gamma\right)  \left(  zI-\Gamma\right)  ^{-1}dz\right]\mathbf{1}_{\overline{\mathcal{A}}_n\cap \overline{\mathcal{E}}_{j,n}}   \\
& =\frac{1}{\lambda_j2\pi\iota}\mathrm{tr}\left[  \int_{\mathcal{B}_{j}}z\left(  zI-\widehat{\Gamma}
_{n}\right)  ^{-1}\left(  \widehat{\Gamma}_{n}-\Gamma\right)  \left(  zI-\Gamma\right)
^{-1}\left(  \widehat{\Gamma}_{n}-\Gamma\right)  \left(  zI-\Gamma\right)
^{-1}dz\right]  \mathbf{1}_{\overline{\mathcal{A}}_n\cap \overline{\mathcal{E}}_{j,n}} \\
& =\frac{1}{\lambda_j2\pi\iota}\mathrm{tr}\left[  \int_{\mathcal{B}_{j}}z\left(  zI-\Gamma\right)
^{-1/2}S_{n}\left(  z\right)  T_{n}^{2}\left(  z\right)  \left(
zI-\Gamma\right)  ^{-1/2}dz\right]\mathbf{1}_{\overline{\mathcal{A}}_n\cap \overline{\mathcal{E}}_{j,n}} \ .
\end{align*}
From the upper bound
\begin{eqnarray*}
\lefteqn{\left|\mathrm{tr}\left[\left(  zI-\Gamma\right)
^{-1/2}S_{n}\left(  z\right)  T_{n}^{2}\left(  z\right)  \left(
zI-\Gamma\right)  ^{-1/2}\right]\right|}&&\\ &&\leq  \|(zI-\Gamma)^{-1/2}S_nT_n\|_{HS} \|T_n(zI-\Gamma)^{-1/2}\|_{HS}\leq \|(zI-\Gamma)^{-1}\|_{\infty}\|S_n\|_{\infty}\|T_n\|^2_{HS}\ ,
\end{eqnarray*}
we derive as in the proof of (\ref{eq_majoration_principale_lemme})%
\begin{align*}
\lefteqn{\frac{1}{\lambda_j}\mathbb{E}\left[\left\vert \mathrm{tr}\left(  \widehat{\Gamma}_{n}\left(  \widehat{\pi}%
_{j}-\pi_{j}\right)  \right)  \right\vert \mathbf{1}_{\overline{\mathcal{A}}_n\cap \overline{\mathcal{E}}_{j,n}} \right]}&\\ & \leq \frac{C}{\lambda_j}\mathbb{E}\left[\int_{\mathcal{B}_{j}%
}\left\vert z\right\vert \left\Vert \left(  zI-\Gamma\right)  ^{-1}\right\Vert
_{\infty}\left\Vert S_n(z)\right\Vert_{\infty}\left\Vert T_{n}\left(  z\right)  \right\Vert _{HS}^{2}dz\mathbf{1}_{\overline{\mathcal{E}}_{j,n}}\right]\\
& \leq C\mathbb{E}\left[\sup_{z\in\mathcal{B}_{j}}\left\Vert T_{n}\left(
z\right)  \right\Vert _{HS}^{2}\mathbf{1}_{\overline{\mathcal{E}}_{j,n}}\right] \leq C(\gamma)\frac{j^{2}(\log^{2}j\vee 1)}{n}.
\end{align*}
Gathering (\ref{eq_majoration_1_ratio_vp}) and (\ref{eq_majoration_2_ratio_vp}) with this last bound, we get
\[
\sum_{j=1}^{k}\mathbb{E}\left[\left\vert \frac{\widehat{\lambda}_{j}-\lambda_{j}%
}{\lambda_{j}}\right\vert\mathbf{1}_{\overline{\mathcal{A}}_n\cap \overline{\mathcal{E}}_{j,n}}\right] \leq C(\gamma) \frac{k^{3}(\log^{2}(k)\vee 1)}{n}+ C'\frac
{k}{\sqrt{n}},
\]
Combining this last bound with (\ref{trace}) and (\ref{eq_majoration_principale_lemme}) allows us to conclude.

\section*{Acknowledgements} The research of N. Verzelen is partly supported by the french Agence Nationale
de la Recherche (ANR 2011 BS01 010 01 projet Calibration).
We would like to thank two anonymous referees for their insightful remarks that lead us to significantly improve the presentation of the paper.

\bibliographystyle{acmtrans-ims}
\bibliography{estimation}

\vfill\eject

\appendix
\numberwithin{equation}{section}
\section{Power under Gaussian Noise}\label{sec:power:gaussian}

\subsection{Power of $T_{\alpha,k}$}

\begin{prop}[Power under Gaussian errors]\label{prte_power_KL_gaussian} 
There exists positive constants $C$, $C_1(\beta)$, and $C_2$ such that the following holds.
Suppose that $\alpha\geq \exp(-n/20)$, $\beta\geq C/n$ and that Assumptions ${\bf B.1}$ and ${\bf A.1}$ are true.
Then, $\mathbb{P}_{\theta}(T_{\alpha,k}>0)\geq 1-\beta$ for any $\theta$  satisfying 
\begin{equation}\label{definition_puissance_complete_gaussien}
\|\Gamma^{1/2}\theta\| ^2 \geq 
C_1(\beta)\left(\lambda_{k+1}+\sum_{j\geq k+1}\frac{\lambda_j}{\sqrt{n}}\right)\|\theta\| ^2 + 
C_2\frac{\sigma^2}{n}\left[\sqrt{k\log\left(\frac{\log n}{\alpha \beta}\right)}+ \log\left(\frac{\log n}{\alpha \beta	}\right)\right]\ .
\end{equation} 
\end{prop}

\noindent 
\begin{remark}
If this result requires very weak assumptions on the process $X$ (only a fourth moment assumption), the bound \eqref{definition_puissance_complete_gaussien} is slightly looser than \eqref{definition_puissance_complete} in Theorem \ref{thrm_power_KL_fixed} because 
\[\|(\Gamma^{1/2}-\Gamma_k^{1/2})\theta\| ^2\leq \lambda_{k+1}\|\theta\| ^2\ .\]
\end{remark}

\subsection{Power of $T_{\alpha}^{(1)}$}

A similar result holds for $T_{\alpha}^{(1)}$.

\begin{prop}\label{prte_power_KL_gaussian_multiple} 
There exists positive constants $C$, $C_1(\beta)$, and $C_2$ such that the following holds.
Suppose that $\alpha\geq \exp(-n/20)$, $\beta\geq C/n$ and that Assumptions ${\bf B.1}$ and ${\bf A.1}$ are true.
Then, $\mathbb{P}_{\theta}(T_{\alpha}^{(1)}>0)\geq 1-\beta$ for any $\theta$  satisfying 
\begin{equation*}
\|\Gamma^{1/2}\theta\| ^2 \geq \inf_{k\in\mathcal{K}_n}
C_1(\beta)\left(\lambda_{k+1}+\sum_{j\geq k+1}\frac{\lambda_j}{\sqrt{n}}\right)\|\theta\| ^2 + 
C_2\frac{\sigma^2}{n}\left[\sqrt{k\log\left(\frac{\log n}{\alpha \beta}\right)}+ \log\left(\frac{\log n}{\alpha \beta	}\right)\right]\ .
\end{equation*} 
\end{prop}

\subsection{Proofs of Propositions \ref{prte_power_KL_gaussian} and \ref{prte_power_KL_gaussian_multiple}}\label{section:remaining_proofs}

We first prove Proposition \ref{prte_power_KL_gaussian_multiple} and then adapt the arguments to Proposition \ref{prte_power_KL_gaussian}.

\begin{proof}[Proof of Proposition \ref{prte_power_KL_gaussian_multiple}]
Let us first work conditionally to ${\bf X}$. In this case, the design ${\bf X}$ and the projection $\widehat{\Pin}_k$  are considered as fixed. Thus, the statistic $T^{(1)}_{\alpha}$ is analogous to the procedure of Baraud et al.~\cite{baraud03}. By Theorem 1 in \cite{baraud03}, we have $\mathbb{P}_{\theta}(T_{\alpha}^{(1)}>0)\geq 1-\beta/2$ if $\theta$ satisfies $n\langle\theta,\widehat{\Gamma}_n\theta\rangle  \geq \inf_{k\in\mathcal{K}_n}\Delta(\theta,k,{\bf X})$ where $\Delta(\theta,k,{\bf X})$ is defined by 
\begin{eqnarray}\label{definition_delta}
 C_1 \langle\theta,\widehat{\Pi}^{\perp}_{\hat{k}^{KL}} \widehat{\Gamma}_n\theta\rangle + C_2\sqrt{\hat{k}^{KL}\log\left(\frac{2\log n}{\alpha \beta	}\right)}\sigma^2+ C_3\log\left(\frac{2\log n}{\alpha \beta	}\right)\sigma^2\ .
\end{eqnarray}
since $\alpha\geq \exp(-n/20)$, $\beta\geq C/n$.
We have $n\langle \theta,\widehat{\Gamma}_n\theta\rangle =\sum_{i=1}^n \langle  X_i,\theta\rangle ^2$. 
By Assumption ${\bf B.1}$, we get 
$$\mathbb{E}[\langle X,\theta\rangle ^4]=\mathbb{E}\left[\left(\sum_{j=1}^{\infty} \sqrt{\lambda_j}\theta_j \eta^{(j)}\right)^4\right]\leq C\langle \theta,\Gamma \theta\rangle ^4\ .$$
Applying Chebychev's inequality, we have 
\begin{equation}\label{eq2_type2_gaussien}
 \langle \theta,\widehat{\Gamma}_n\theta\rangle \geq  \mathbb{E}[\langle X,\theta\rangle ^2]/2\text{, with probability larger than }1-\beta/4\text{ as long as }\beta\geq C/n\ .
\end{equation}

Let us fix some $k\in\mathcal{K}_n$. We have $\langle \theta,\widehat{\Pi}^{\perp}_{\hat{k}^{KL}} \widehat{\Gamma}_n\theta\rangle^2 \leq \|\theta\| ^2\widehat{\lambda}_{\hat{k}^{KL}+1}$. Observe that $\hat{k}^{KL}+1 < k+1$ only if $\widehat{\lambda}_{\hat{k}^{KL}+1}=0$. Consequently, we also have $\langle \theta,\widehat{\Pi}^{\perp}_{\hat{k}^{KL}} \widehat{\Gamma}_n\theta\rangle \leq \|\theta\| ^2\widehat{\lambda}_{k+1}$.

To conclude it is sufficient to provide an upper bound of $\widehat{\lambda}_{k+1}$ with high probability.
By definition of $\widehat{\lambda}_{k+1}$, we have 
$$\widehat{\lambda}_{k+1}= \inf_{W,\ \textrm{Codim(W)=k}}\sup_{z\in W^{\perp},\ \|z\| =1}\langle z,\widehat{\Gamma}_nz\rangle \leq \sup_{z\in \text{Vect}(V_{k+1},...),\  \|z\| =1}\langle z,\widehat{\Gamma}_nz\rangle  ,$$
implying that 
$$\widehat{\lambda}_{k+1}\leq \|\Pi^\perp_{k}\widehat{\Gamma}_n\Pi^\perp_{k}\|_{\infty}\leq \lambda_{k+1}+ \|\Pi^\perp_{k}(\Gamma-\widehat{\Gamma}_n)\Pi^\perp_{k}\|_{\infty}\leq \lambda_{k+1} + \|\Pi^\perp_{k}(\Gamma-\widehat{\Gamma}_n)\Pi^\perp_{k}\|_{HS}\ . $$
Hence, it is sufficient to bound the Hilbert Schmidt norm $\|\Pi^\perp_{k}(\Gamma-\widehat{\Gamma}_n)\Pi^\perp_{k}\|_{HS}$ in probability.
 By Jensen's inequality, we have $\mathbb{E}[\|\Pi^\perp_{k}(\Gamma-\widehat{\Gamma}_n)\Pi^\perp_{k}\|_{HS}]\leq \mathbb{E}[\|\Pi^\perp_{k}(\Gamma-\widehat{\Gamma}_n)\Pi^\perp_{k}\|^2_{HS}]^{1/2}$ and simple calculations lead to 
\[\mathbb{E}[\|\Pi^\perp_{k}(\Gamma-\widehat{\Gamma}_n)\Pi^\perp_{k}\|^2_{HS}]= \frac{1}{n}\mathbb{E}[\|\Pi^\perp_{k}\Gamma\Pi^\perp_{k}-(\Pi^\perp_{k}X)\otimes (\Pi^\perp_{k}X)]\|^2_{HS}]\ .\]
By Assumption ${\bf B.1}$, we conclude that 
 $$\mathbb{E}[\|\Pi^\perp_{k}\Gamma\Pi^\perp_{k}-(\Pi^\perp_{k}X)\otimes (\Pi^\perp_{k}X)]\|^2_{HS}]\leq \mathbb{E}[\|\Pi^\perp_{k}X\|^4 ]\leq C(\sum_{j\geq k+1}\lambda_j)^2\ .$$
By Markov inequality, we conclude that 
$\widehat{\lambda}_{k+1}\leq \lambda_k+C(\beta)\sum_{j\geq k+1}\frac{\lambda_j}{\sqrt{n}}$ with probability larger than $1-\beta/4$. Gathering this probability bound with (\ref{definition_delta}) and (\ref{eq2_type2_gaussien}), we derive that ${P}_{\theta}(T_{\alpha}^{(1)}>0)\geq 1-\beta$ if $\theta$ satisfies for some $k\in\mathcal{K}_n$, 
\[\|\Gamma^{1/2}\theta\| ^2 \geq \frac{C_1}{n}\|\theta\| ^2 \left[\lambda_k+ C(\beta)\sum_{j\geq k+1}\lambda_j/\sqrt{n}\right]+ C_2\frac{\sigma^2}{n}\left[\sqrt{k\log\left(\frac{2\log n}{\alpha \beta	}\right)}+\log\left(\frac{2\log n}{\alpha \beta	}\right)\right]\ .\]

\end{proof}

\begin{proof}[Proof of Proposition \ref{prte_power_KL_gaussian}]
As in the previous proof, we apply Theorem 1 in \cite{baraud03} except that there is now only one test (instead of $|\mathcal{K}_n|$ tests). We have $\mathbb{P}_{\theta}(T_{\alpha,k}>0)\geq 1-\beta/2$ if $\theta$ satisfies 
\[n\langle\theta,\widehat{\Gamma}_n\theta\rangle \geq C_1 \langle\theta,\widehat{\Pi}^{\perp}_{\hat{k}^{KL}} \widehat{\Gamma}_n\theta\rangle + C_2\sqrt{\hat{k}^{KL}\log\left(\frac{2\log n}{\alpha \beta	}\right)}\sigma^2+ C_3\log\left(\frac{2\log n}{\alpha \beta	}\right)\sigma^2\ .\]
Furthermore, we have shown that
\begin{eqnarray*}
\langle \theta,\widehat{\Gamma}_n\theta\rangle &\geq&  \mathbb{E}[\langle X,\theta\rangle ^2]/2\\
\langle \theta,\widehat{\Pi}^{\perp}_{\hat{k}^{KL}} \widehat{\Gamma}_n\theta\rangle&\leq &\|\theta\| ^2\left[ \lambda_k+C(\beta)\sum_{j\geq k+1}\frac{\lambda_j}{\sqrt{n}}\right]\ .
\end{eqnarray*}
 with probability larger than $1-\beta/2$. Gathering these three bounds leads to the desired result.
\end{proof}

\section{Proofs of the minimax lower bounds}\label{section_proof_minimax}

\begin{proof}[Proof of Proposition \ref{prop_minoration}]
For any dimension $k\geq 1$, we define $r_k^2=C(\alpha,\beta)\frac{\sqrt{k}}{n}\wedge 
\lambda_k a_k^2R^2$, where the constant $C(\alpha,\beta)$ will be fixed later. For any $\theta\in \text{Vect}(V_1,\ldots, V_k)$ such that $\|\Gamma^{1/2}\theta\| ^2/\sigma^2\leq r_k^2$, we have 
$$\sum_{j=1}^k\frac{\langle \theta,V_j\rangle^2}{a_j^2}\leq \frac{1}{\lambda_ka_k^2}\sum_{j=1}^k \langle \theta,V_j\rangle^2\lambda_i\leq \frac{r_k^2\sigma^2}{a_k^2\lambda_k}\leq R^2\sigma^2$$
since $r_k^2\leq \lambda_k a_k^2R^2$ and since the $\lambda_ja^2_j$'s are 
non increasing. As a consequence,  
\begin{eqnarray*} 
\left\{\theta\in \text{Vect}(V_1,\ldots, V_k),\ \|\Gamma^{1/2}\theta\| ^2/\sigma^2=r^2_k\right\}\subset \left\{\theta\in \mathcal{E}_a(R),\ \|\Gamma^{1/2}\theta\| ^2/\sigma^2\geq r^2_k\right\}\ . 
\end{eqnarray*} 
Since $X$ is a centered Gaussian process, $(\langle X,V_1\rangle ,\ldots, \langle X,V_k\rangle )$ is a centered Gaussian vector. Assuming that $\theta$ belongs to $\text{Vect}(V_1,\ldots, V_k)$ and  that $(V_1,\ldots, V_k)$ is known, the functional linear model translates as a linear Gaussian model with Gaussian design as studied in \cite{VV10}:
$$Y=\sum_{j=1}^k \langle X,V_j\rangle \langle\theta,V_j\rangle +\epsilon\ .$$
By Proposition 4.2 in \cite{VV10}, there exists a constant $C(\alpha, \beta)$, such that 
for any test $T$ of level $\alpha$, we have
\begin{eqnarray*}
 \boldsymbol{\beta}\left[T; \left\{\theta\in \text{Vect}(V_1,\ldots, V_k)  , \sigma>0,\  \|\Gamma^{1/2}\theta\| ^2 \geq C(\alpha,\beta)\frac{\sqrt{k}}{n}\sigma^2\right\}\right]\geq \beta\ .
\end{eqnarray*}
Gathering this last  bound for all $k\geq 1$ allows us to conclude.
\end{proof}
\begin{proof}[Proof of Proposition \ref{prop_minoration_adaptation}]
 As in the last proof, we shall adapt results for the Gaussian linear regression model with Gaussian design.
Let $k^*_n(R) \in \mathbb{N}^*$ be an integer 
that achieves the supremum of $\tilde{r}^2_{k}=C(\alpha,\beta)\sqrt{k\log\log(k\vee 3) }/n\wedge R^2a_k^2\lambda_k $. We note as in the last proof that 
for any $R>0$ and $k^*_n(R)$ in $\mathbb{N}^*$, 
\begin{eqnarray*} 
\left\{\theta\in \text{Vect}(V_1,\ldots, V_{k^*_n(R)}),\ \frac{\|\Gamma^{1/2}\theta\| ^2}{\sigma^2}=\tilde{r}^2_{k^*_n(R)}\right\}\subset \left\{\theta\in \mathcal{E}_a(R),\ \frac{\|\Gamma^{1/2}\theta\| ^2}{\sigma^2}\geq \tilde{r}^2_{k^*_n(R)}\right\}.\
\end{eqnarray*} 
Thus, we obtain
\begin{eqnarray*} 
\lefteqn{\bigcup_{k\geq 1}\left\{\theta\in \text{Vect}(V_1,\ldots, V_{k}),\ 
\frac{\|\theta\| ^2}{\var(Y)-\|\theta\| ^2}= C(\alpha,\beta)\sqrt{k\log\log(k\vee 3) }/n\right\}}&&\\&\subset& 
\bigcup_{R>0}\left\{\theta\in \text{Vect}(V_1,\ldots, V_{k^*_n(R)})\ \frac{\|\Gamma^{1/2}\theta\| ^2}{\sigma^2}= r^2_{k^*_n(R)}\right\} 
\\  &\subset&  \bigcup_{R>0}\left\{\theta\in \mathcal{E}_a(R),\ \frac{\|\Gamma^{1/2}\theta\| ^2}{\sigma^2}\geq r^2_{D^*(R)} \right\}\ .
\end{eqnarray*}
Hence, we only have to provide a minimax lower bound for simultaneously testing over a family of nesting linear spaces. Letting $p$ go to infinity in Proposition 5.5 in \cite{VV10}, we obtain that
$$\boldsymbol{\beta}\left[\bigcup_{k\geq 1}\left\{\theta\in \text{Vect}(V_1,\ldots, V_{k})\ ,
\frac{\|\Gamma^{1/2}\theta\| ^2}{\var(Y)-\|\Gamma^{1/2}\theta\| ^2}= C(\alpha,\beta)\sqrt{k\log\log(k\vee 3) }/n\right\}\right]\geq \beta\ ,$$
which allows us to conclude.

\end{proof}

\section{Proofs based on Berry-Esseen type inequalities}\label{section:berry_essen}

\begin{proof}[Proof of Lemma \ref{lemma_1}]
Let us fix some $k>0$. For any $1\leq j\leq k$, we have
\begin{eqnarray*}
\sqrt{n} \langle\Delta_n,A_kV_j\rangle  =
\frac{1}{\sqrt{n}}\sum_{i=1}^n\langle
X_i,\frac{V_j}{\sqrt{\lambda_j}}\rangle {\epsilon}_i = \frac{1}{\sqrt{n}}\sum_{i=1}^n
\boldsymbol{\eta}_{i}^{(j)}\epsilon_i
\end{eqnarray*}
For any $1\leq j_1<j_2\leq k$ and $1\leq i\leq j$, the random variables
$\boldsymbol{\eta}_{i}^{(j_1)}\epsilon_i$ and $\boldsymbol{\eta}_{i}^{(j_2)}\epsilon_i$ are uncorrelated.
By the central limit theorem, we conclude that
$\|\sqrt{n}A_k\Delta_n\| ^2/\sigma^2$ converges in distribution towards a
$\chi^2(k)$ random variable, at least when $k$ is fixed.\\

In order to precisely control the tails of $\|\sqrt{n}A_k\Delta_n\| ^2$, the
central limit theorem is not sufficient. We need
a Berry-Esseen type inequality. 
Let us call $W_i$ the vector of size $k$ whose $j$-th component is
$\boldsymbol{\eta}_{i}^{(j)}\epsilon_i$. We note $\|W_i\|_k$ its Euclidean norm. By
Assumption ${\bf B.1}$, we have 
\begin{eqnarray*}
 \mathbb{E}\left[\|W_i\|_k^3\right]\leq
k^{3/2}\mathbb{E}\left[\epsilon^4\right]^{3/4}\sup_ { 1\leq j\leq
k}\mathbb{E}\left[(\eta^{(j)})^4\right]^{3/4}\ .
\end{eqnarray*}
Applying the second part of Theorem 1.1 in
Bentkus \cite{bentkus03}, we obtain 
\begin{eqnarray*}
 \sup_{x>0}\left|\mathbb{P}\left(\|\sqrt{n}A_k\Delta_n\| ^2\geq
x\right)- \bar{\chi}_k(x/\sigma^2)\right|\leq C
\frac{k^{3/2}}{\sqrt{n}}\frac{\mathbb{E}\left[\epsilon^4\right]^{3/4}}{\sigma^3}
\sup_ { 1\leq j\leq k}\mathbb{E}\left[(\eta^{(j)})^4\right]^{3/4} \ .
\end{eqnarray*}
 We conclude by applying Assumption ${\bf B'.3}$.

\end{proof}

\begin{proof}[Proof of Lemma \ref{lemma_A1}]

As explained in the proof of Lemma \ref{lemma_1},  $\sqrt{n}/\sigma A_k\Delta_{n,1}$
converges to a Gaussian process whose covariance operator $\Sigma_k$ is defined by  $\Sigma_k=
\sum_{j=1}^k\langle V_j,.\rangle V_j$. For $j=1,\ldots, k$, we  define $\xi_j= (\lambda_j^{1/2}\sqrt{\var([\eta^{(j)}]^2)}\langle\theta,V_j\rangle )^{-1}$ if $\langle\theta,V_j\rangle^2 \neq 0$ and $\xi_j=0$ else. Consider the operator $D_k=\sum_{j=1}^k \xi_j\langle V_j,.\rangle V_j$. 
For any $j=1,\ldots, k$ such that $\xi_j\neq 0$, we have 
$$\sqrt{n}\langle D_kA_k\widehat{\Gamma}_n\theta,V_j\rangle -\sqrt{\frac{n}{\var([\eta^{(j)}]^2)}}=\frac{\sum_{i=1}^n[\boldsymbol{\eta}_i^{(j)}]^2-1}{\sqrt{n\var([\eta^{(j)}]^2)}}\ .$$
As a consequence, $\sqrt{n}(D_kA_k\widehat{\Gamma}_n\theta-D_k\Gamma_k^{1/2}\theta)$ converges in distribution towards a Gaussian process whose covariance operator $\Sigma'_k$ is defined by $\Sigma'_k=\sum_{j=1}^k \langle V_j,.\rangle V_j\mathbf{1}_{\xi_j\neq 0}$. Furthermore, the processes $\sqrt{n}/\sigma A_k\Delta_{n,1}$ and $\sqrt{n}(D_kA_k\widehat{\Gamma}_n\theta-D_k\Gamma_k^{1/2}\theta)$ are asymptotically independent. Let us consider the random vector $Z$ of size $k':=k+\# \{j\in\{1,\ldots,k\}: \xi_j\neq 0\}$ such that
$Z_j=\epsilon/\sigma \eta^{(j)}$ if $j=1,\ldots, k$ and $Z_j=([\eta^{(j)}]^2-1)/\sqrt{\var([\eta^{(j)}]^2)}$ if $j>k$.
Let us upper bound  $\mathbb{E}[\|Z\|^3_{k'}]$
\begin{eqnarray*}
 \mathbb{E}[\|Z\|^3_{k'}]\leq C k^{3/2}\left(\frac{\mathbb{E}[\epsilon^4]^{3/4}}{\sigma^3}\max_{1\leq j\leq k}\mathbb{E}\left[(\eta^{(j)})^4\right]^{3/4}\vee \max_{1\leq j\leq k}\mathbb{E}\left[(\eta^{(j)})^8\right]^{3/4}\right)\ .
\end{eqnarray*}
We note ${\bf Z}_1,\ldots, {\bf Z}_n$ the $n$ observations of the vector $Z$, based on $\boldsymbol{\eta}_i^{(j)}$ and $\boldsymbol{\epsilon}_i$ for $i=1,\ldots, n$.
By Assumptions ${\bf B.1}$ and  ${\bf B.4}$, we can  apply the Berry-Esseen type inequality of
Bentkus (Theorem 1.1 in \cite{bentkus03}) in dimension $k'$. For any convex set $\mathcal{A}$, we obtain 
\begin{eqnarray*}
 \lefteqn{\left|\mathbb{P}\left(\sum_{i=1}^n\frac{{\bf Z}_i}{\sqrt{n}}\in\mathcal{A}\right)- \mathbb{P}\left[\mathcal{N}_{k'}(0,I_{k'})\in \mathcal{A}\right]\right|
\leq C\frac{k^{7/4}}{\sqrt{n}}}&&\\&\hspace{3cm}&\times \left(\frac{\mathbb{E}[\epsilon^4]^{3/4}}{\sigma^3}\max_{1\leq j\leq k}\mathbb{E}\left[(\eta^{(j)})^4\right]^{3/4}\vee \max_{1\leq j\leq k}\mathbb{E}\left[(\eta^{(j)})^8\right]^{3/4}\right)\ .
\end{eqnarray*}
 Moreover, this last quantity is smaller than $Cn^{-1/16}\log^{-7}(n)$ uniformly over all $k\leq \bar{k}_n$ by Assumption ${\bf B'.3}$.
Consider a standard Gaussian vector $(u_1,\ldots , u_{2k})$. We  define the  random vector $W$ by
$$W= \sum_{j=1}^k \left(\sqrt{n\lambda_j}\langle\theta,V_j\rangle + \sqrt{\lambda_j}\langle\theta,V_j\rangle \sqrt{\var([\eta^{(j)}]^2)} u_j+\sigma u_{j+k}\right)^2\ .$$
We derive from the definition of $W$ and the previous Berry-Esseen inequality that 
\begin{eqnarray*}
\sup_{x>0}\left|\mathbb{P}\left(\left\|\sqrt{n}A_{k,n}\left(\widehat{\Gamma}_n\theta+\Delta_{n,1}\right)\right\| ^2\geq
x\right)- \mathbb{P}(W\geq x)\right| \leq \frac{C}{\log^7 (n)}\ . 
\end{eqnarray*}
Conditionally to $(u_1,\ldots, u_k)$, $W/\sigma^2$ follows a non-central $\chi^2$ distribution with $k$ degrees of freedom and non-centrality parameter \[V:= \sum_{j=1}^k \left(\sqrt{n\lambda_j}\langle \theta,V_j\rangle + \sqrt{\lambda_j}\langle\theta,V_j\rangle \sqrt{\var([\eta^{(j)}]^2)} u_j\right)^2/\sigma^2\ .\] By a deviation inequality on non-central $\chi^2$ distributions (e.g. Eq.18 in \cite{baraud03}), we derive that, conditionally to $(u_1,\ldots, u_k)$,
$$W\geq k\sigma^2 +\frac{4}{5}V\sigma^2-2\sigma^2\sqrt{k\log(2/\beta)}-10\sigma^2\log(2/\beta)\ ,$$
with probability larger than $1-\beta/2$.
The non-centrality parameter $V$ is a polynomial function of independent normal variables. Applying a deviation inequality for normal variables, we derive that  $V\geq n/4\|\Gamma^{1/2}_k\theta\| ^2/\sigma^2$ with probability larger than $1-\sum_{j=1}^k \exp\left[-n\var([\eta^{(j)}]^2)/8\right]$. All in all, we conclude that 
$$\left\|\sqrt{n}A_{k,n}\left(\widehat{\Gamma}_n\theta+\Delta^2_n\right)\right\| ^2\geq k\sigma^2+ \frac{n}{5}\|\Gamma^{1/2}_k\theta\| ^2-2\sigma^2\sqrt{k\log(2/\beta)}-10\sigma^2\log(2/\beta)\ ,$$
with probability larger than $1-\beta/2-C/\log^7(n)- n\exp[-C'n]$.

\end{proof}

\section{Remaining proofs based on perturbation theory}\label{section_perturb_technical}

\subsection{Proof of Lemma \ref{perturb1}}

The second bound straightforwardly follows from the first bound by Markov inequality.
Fix $z\in\mathcal{B}_{j}$. We have%
\begin{align*}
&  \left\Vert \left(  zI-\Gamma\right)  ^{-1/2}\left(  \widehat{\Gamma}_{n}%
-\Gamma\right)  \left(  zI-\Gamma\right)  ^{-1/2}\right\Vert _{HS}^{2}\\
&  =\sum_{l=1}^{+\infty}\sum_{k=1}^{+\infty}\left\langle \left(
zI-\Gamma\right)  ^{-1/2}\left(  \widehat{\Gamma}_{n}-\Gamma\right)  \left(
zI-\Gamma\right)  ^{-1/2} V_{l}  ,V_{k}\right\rangle  ^{2}%
=\sum_{l,k=1}^{+\infty}\frac{\left\langle \left(  \widehat{\Gamma}_{n}-\Gamma\right)
V_{l}  ,V_{k}\right\rangle  ^{2}}{\left\vert z-\lambda
_{l}\right\vert \left\vert z-\lambda_{k}\right\vert }.%
\end{align*}
Since for $z=\lambda_{j}+\frac{\delta_{j}}{2}e^{\iota\theta}\in\mathcal{B}%
_{j}$ and $i\neq j$
\[
\left\vert z-\lambda_{i}\right\vert =\left\vert \lambda_{j}-\lambda_{i}%
+\frac{\delta_{j}}{2}e^{\iota\theta}\right\vert \geq\left\vert \lambda
_{j}-\lambda_{i}\right\vert -\frac{\delta_{j}}{2}\geq\left\vert \lambda
_{j}-\lambda_{i}\right\vert /2,
\]
we have
\begin{align*}
\sum_{l,k=1}^{+\infty}\frac{\left\langle \left(  \widehat{\Gamma}_{n}-\Gamma\right)
 V_{l}  ,V_{k}\right\rangle  ^{2}}{\left\vert z-\lambda
_{l}\right\vert \left\vert z-\lambda_{k}\right\vert }
&  \leq4\sum_{\substack{l,k=1,\\l,k\neq j}}^{+\infty}\frac{\left\langle
\left(  \widehat{\Gamma}_{n}-\Gamma\right)   V_{l}  ,V_{k}\right\rangle 
^{2}}{\left\vert \lambda_{j}-\lambda_{l}\right\vert \left\vert \lambda
_{j}-\lambda_{k}\right\vert }+2\sum_{\substack{k=1,\\k\neq j}}^{+\infty}%
\frac{\left\langle \left(  \widehat{\Gamma}_{n}-\Gamma\right)    V_{j}
,V_{k}\right\rangle  ^{2}}{\delta_{j}\left\vert \lambda_{j}-\lambda
_{k}\right\vert }\\ &+\frac{\left\langle \left(  \widehat{\Gamma}_{n}-\Gamma\right)  
V_{j} ,V_{j}\right\rangle  ^{2}}{\delta_{j}^{2}}\ .
\end{align*}
Applying Assumption {\bf B.1}, we derive
\begin{eqnarray*}
 \mathbb{E}\left[\sum_{l,k=1}^{+\infty}\frac{\left\langle \left(  \widehat{\Gamma}_{n}-\Gamma\right)
V_{l}  ,V_{k}\right\rangle  ^{2}}{\left\vert z-\lambda
_{l}\right\vert \left\vert z-\lambda_{k}\right\vert }
\right]&\leq& \frac{C}{n}\left[\sum_{\substack{l,k=1,\\l,k\neq j}}^{+\infty}\frac{\lambda_k\lambda_l}{|\lambda_j-\lambda_k|)(|\lambda_j-\lambda_k|)}+ \sum_{\substack{k=1,\\k\neq j}}^{+\infty}
\frac{\lambda_k\lambda_j}{\delta_j|\lambda_j-\lambda_k|}+ \frac{\lambda_j^2}{\delta_j^2}\right]\\
&\leq & \frac{C'}{n}\left[\left(\sum_{k\geq 1,\ k\neq j}^{\infty}\frac{\lambda_k}{|\lambda_k-\lambda_j|}\right)^2+ \frac{\lambda^2_j}{|\lambda_j-\lambda_{j+1}|^2}+\frac{\lambda^2_j}{|\lambda_{j-1}-\lambda_{j}|^2}\right]\ .
\end{eqnarray*}
Applying Lemma \ref{DomEigen} and Assumption ${\bf B.2}$ allows us to conclude.

\subsection{Proof of Lemma \ref{lemme_convergence_valeurs_propres_empirique}}

For any $2\leq j\leq \bar{k}_n$, we define $\delta'_j:=\max(\lambda_j-\lambda_{j+1}, \lambda_{j-1}-\lambda_j)$. Then, we build an oriented circle $\mathcal{B}'_j$ on the complex plane of radius $(\delta'_j-\delta_j)/4$ in such a way that any real number between $(\lambda_j+\lambda_{j+1})/2$ and $(\lambda_j+\lambda_{j-1})/2$ is either inside  $\mathcal{B}_j$ or $\mathcal{B}'_j$.
See Figure \ref{fig-contour} for an example of $\mathcal{B}_j$ and $\mathcal{B}'_j$.\\

\begin{figure}[hptb]
\begin{center}
{\includegraphics[width=11cm,angle=0]{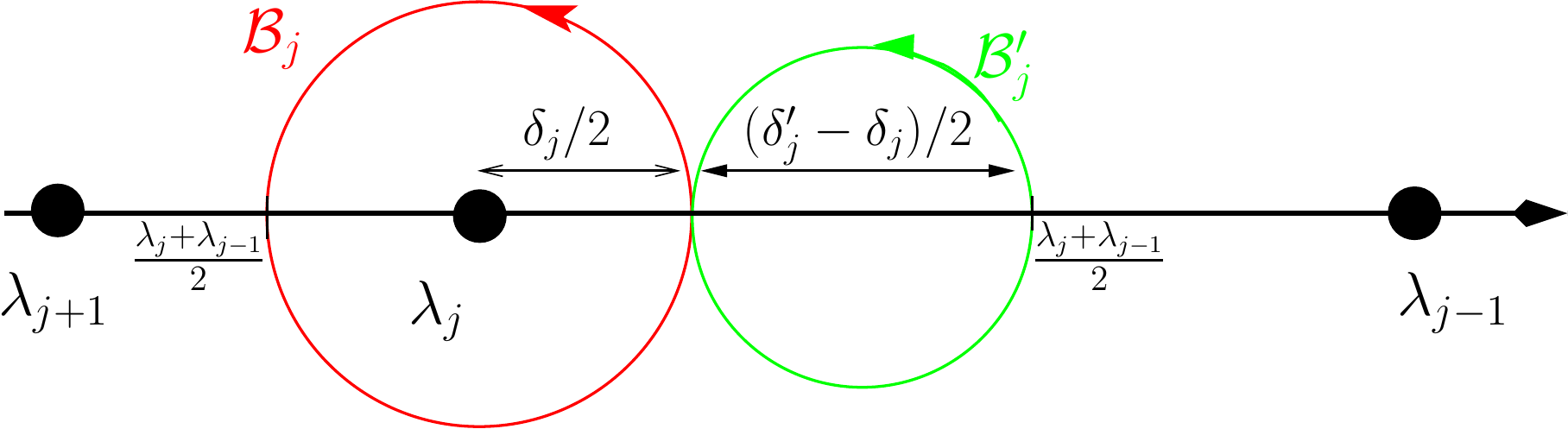}}
\caption{Contours $\mathcal{B}_j$ and $\mathcal{B}'_j$}\label{fig-contour}
\end{center}
\end{figure}

Lemma \ref{lemme_convergence_valeurs_propres_empirique} is a straightforward consequence of the two following lemmas. Let us define $T_n(z)=( zI-\Gamma ) ^{-1/2}(\widehat{\Gamma}_{n}-\Gamma)( zI-\Gamma ) ^{-1/2} $ and $S_{n}( z)=( zI-\Gamma ) ^{1/2}( zI-\widehat{\Gamma}_{n})
^{-1}( zI-\Gamma ) ^{1/2}$.
\begin{lemma}\label{lemma00}
 We have $\mathcal{A}_{n}\subset \mathcal{E}_{n}\cup \mathcal{E}'_n \cup\left\{\widehat{\lambda}_1\geq \frac{3\lambda_1-\lambda_2}{2}\right\}$,
where 
\begin{eqnarray*}
 \mathcal{E}_n&:=&\left\{ \sup_{1\leq j\leq \bar{k}_n}\sup_{z\in \mathcal{B}_{j}}\left\Vert T_{n}\left(
z\right) \right\Vert _{\infty }\geq 0.5\right\}\ ,\quad\quad
 \mathcal{E}'_n:=\left\{ \sup_{2\leq j\leq \bar{k}_n}\sup_{z\in \mathcal{B}'_{j}}\left\Vert T_{n}\left(
z\right) \right\Vert _{\infty }\geq 0.5\right\}\ .\\ 
\end{eqnarray*}
\end{lemma}

\begin{lemma}\label{lemma01}
Under Assumptions ${\bf B.1}$ and ${\bf B.2}$, we have \begin{eqnarray*}
\mathbb{P}\left(\mathcal{E}_n\right)\leq C_1(\gamma)\frac{\bar{k}^3_n\log^2(\bar{k}_n\vee e)}{n},\quad
 \mathbb{P}\left(\mathcal{E}'_n\right)\leq C_2(\gamma)\frac{\bar{k}^3_n\log^2(\bar{k}_n\vee e)}{n},\quad
\mathbb{P}\left[\widehat{\lambda}_1\geq \frac{3\lambda_1-\lambda_2}{2}\right]\leq \frac{C_3(\gamma)}{n}\ .
\end{eqnarray*}

\end{lemma}

\begin{proof}[Proof of Lemma \ref{lemma00}]
Suppose that the four following events hold: 1) $\widehat{\Gamma}_n$ has no eigenvalue on all the contours $\mathcal{B}_j$ and $\mathcal{B}'_j$. 2) For each $1\leq j\leq \bar{k}_n$, $\widehat{\Gamma}_n$ has exactly one eigenvalue inside the circle $\mathcal{B}_j$. 3) For each $2\leq j\leq \bar{k}_n$, $\widehat{\Gamma}_n$ has no eigenvalue inside the circle $\mathcal{B}'_j$. 4) $\widehat{\lambda}_1< (3\lambda_1-\lambda_2)/2$. In such a case, the event $\overline{\mathcal{A}}_n$ is true. As a consequence, $\mathcal{A}_n$ is included in the union of the four following events denoted $\mathcal{D}_1$, $\mathcal{D}_2$, $\mathcal{D}_3$ and $\mathcal{D}_4$.
\begin{itemize}
\item For some $1\leq j\leq \bar{k}_n$, $\widehat{\Gamma}_n$ has an eigenvalue that lies on the contours $\mathcal{B}_j$ and $\mathcal{B}'_j$.
 \item For some $1\leq j\leq \bar{k}_n$, $\widehat{\Gamma}_n$ has  either $0$ or more than $2$ eigenvalues inside the circle $\mathcal{B}_j$.
\item For some $2\leq j\leq \bar{k}_n$, $\widehat{\Gamma}_n$ has  at least $1$ eigenvalue inside the circle $\mathcal{B}'_j$.
\item $\widehat{\lambda}_1\geq  (3\lambda_1-\lambda_2)/2$.
\end{itemize}
We shall prove that $\mathcal{D}_1\subset \mathcal{E}_n\cup \mathcal{E}'_n$, that $\mathcal{D}_2\setminus \mathcal{D}_1\subset \mathcal{E}_n$ and that $\mathcal{D}_3\setminus \mathcal{D}_1 \subset \mathcal{E}'_n$. \\

\noindent 
{\bf Event $\mathcal{D}_1$}. Assume that an eigenvalue of $\widehat{\Gamma}_n$ lies exactly on some contour $\mathcal{B}_{j}\cup\mathcal{B}'_j$. Let us call $\widehat{\lambda}$ such an eigenvalue and $\widehat{V}$ a corresponding eigenvector. We have 
\begin{eqnarray*}
 T_n(\widehat{\lambda})(\widehat{\lambda} I-\Gamma)^{1/2}\widehat{V}&=& (\widehat{\lambda} I -\Gamma)^{-1/2}(\widehat{\Gamma}_n-\Gamma)\widehat{V}\\ & =& (\widehat{\lambda} I -\Gamma)^{-1/2}(\widehat{\lambda} I -\Gamma) \widehat{V}=  (\widehat{\lambda}I-\Gamma)^{1/2}\widehat{V}\ .
\end{eqnarray*}
Since $\widehat{\lambda}$ is not an eigenvalue of $\Gamma$, we have $(\widehat{\lambda}I-\Gamma)^{1/2}\widehat{V}\neq 0$
so that $\sup_{z\in \mathcal{B}_{j}\cup\mathcal{B}'_j}\left\Vert T_{n}\left(
z\right) \right\Vert _{\infty }\geq 1$. Hence, $\mathcal{D}_1\subset \mathcal{E}_{n}\cup \mathcal{E}'_{n}$.\\

\noindent 
{\bf Event $\mathcal{D}_2\setminus \mathcal{D}_1$}. Assume that $\mathcal{D}_2\setminus \mathcal{D}_1$ is true. It follows that for some $1\leq j^{\ast}\leq \bar{k}_n$ the operator $(2\pi\iota)^{-1}\int_{\mathcal{B}_{j^{\ast }}}( zI-\widehat{\Gamma}_{n}) ^{-1}dz$ is an orthogonal projector $\pi_{\widehat{W}_{j^{\ast}}}$ on a space $\widehat{W}_{j^{\ast}}$ of dimension different from one. In contrast, $(2\pi\iota)^{-1}\int_{\mathcal{B}_{j^{\ast }}}( zI-\Gamma) ^{-1}dz$ is the orthogonal projector $\pi_{j^{\ast}}$ on $V_{j^{\ast}}$. Consider
\begin{equation*}
\frac{1}{2\pi\iota}\int_{\mathcal{B}_{j^{\ast }}}\left[ \left( zI-\widehat{\Gamma}_{n}\right) ^{-1}-\left(
zI-\Gamma \right) ^{-1}\right] dz=\pi_{\widehat{W}_{j^{\ast}}}-\pi _{j^{\ast }}\ .
\end{equation*}%
If  $\mathrm{dim}(\widehat{W}_{j^{\ast}})=0$, then $\|\pi_{\widehat{W}_{j^{\ast}}}-\pi _{j^{\ast }}\|_{\infty}=1$. If $\mathrm{dim}(\widehat{W}_{j^{\ast}})\geq 2$, then there exists a vector $\widehat{V}$ in $\widehat{W}_{j^{\ast}}$ such that $\pi _{j^{\ast }} \widehat{V}=0$. As a consequence, we have $\|\pi_{\widehat{W}_{j^{\ast}}}-\pi _{j^{\ast }}\|_{\infty}\geq 1$. 
 For any $z\in \mathcal{B}_{j^{\ast}}$, $S_n(z)$ is well defined since no eigenvalue of $\widehat{\Gamma}_n$ lies on $\mathcal{B}_{j^{\ast}}$.
It follows that
\begin{align}
1& \leq \frac{1}{2\pi}\int_{\mathcal{B}_{j^{\ast
}}}\left\Vert \left( zI-\widehat{\Gamma}_{n}\right) ^{-1}\left( \widehat{\Gamma}_{n}-\Gamma
\right) \left( zI-\Gamma \right) ^{-1}\right\Vert _{\infty }dz \nonumber\\
& \leq \frac{1}{2\pi}\int_{\mathcal{B}_{j^{\ast }}}\left\Vert \left( zI-\widehat{\Gamma}_{n}\right)
^{-1}\left( zI-\Gamma \right) ^{1/2}T_{n}(z)\left( zI-\Gamma \right)
^{-1/2}\right\Vert _{\infty }dz \nonumber\\
& \leq\frac{1}{2\pi}\int_{\mathcal{B}_{j^{\ast }}}\left\Vert \left( zI-\Gamma
\right) ^{1/2}\left( zI-\widehat{\Gamma}_{n}\right) ^{-1}\left( zI-\Gamma \right)
^{1/2}\right\Vert _{\infty }\left\Vert T_{n}\left( z\right) \right\Vert
_{\infty }\left\Vert \left( zI-\Gamma \right) ^{-1/2}\right\Vert _{\infty
}^{2}dz \nonumber\\
& \leq  \sup_{z\in \mathcal{B}_{j^{\ast }}}\left\Vert S_{n}\left( z\right) \right\Vert
_{\infty }\left\Vert T_{n}\left( z\right) \right\Vert _{\infty }\ ,\label{eq_majoration_infini}
\end{align}%
since $\Vert \left( zI-\Gamma \right)^{-1}\Vert _{\infty
}\leq 2/\delta_j$.
 Moreover, we have
$S_{n}\left( z\right) \left( I-T_{n}\left( z\right) \right) =I$. We can assume that $\sup_{z\in\mathcal{B}_{j^{\ast}}}\left\Vert T_{n}\left(
z\right) \right\Vert _{\infty }<0.9$, otherwise $\mathcal{E}_{n}$ is true. Then, we have $\left\Vert S_{n}\left(
z\right) \right\Vert _{\infty }\leq ( 1-\left\Vert T_{n}\left( z\right) \right\Vert _{\infty })^{-1}$. Gathering this bound with (\ref{eq_majoration_infini})
leads to
$\sup_{z\in \mathcal{B}_{j^{\ast }}}\left\Vert T_{n}\left( z\right) \right\Vert
_{\infty }\geq 0.5$, 
which allows us to conclude that $\mathcal{D}_2\setminus \mathcal{D}_1\subset \mathcal{E}_n$.\\

\noindent 
{\bf Event $\mathcal{D}_3\setminus \mathcal{D}_1$}.  Assume that $\mathcal{D}_3\setminus \mathcal{D}_1$ is true. 
Arguing as for $\mathcal{D}_2$, we derive that for some $2\leq j^{\ast}\leq \bar{k}_n$, we have $\delta'_{j^{\ast}}>\delta_{j^{\ast}}$ and
\begin{eqnarray}\label{majoration_norme_op}
 \frac{1}{2\pi}\left\Vert\int_{\mathcal{B}'_{j^{\ast
}}} \left( zI-\widehat{\Gamma}_{n}\right) ^{-1}\left( \widehat{\Gamma}_{n}-\Gamma
\right) \left( zI-\Gamma \right) ^{-1}dz\right\Vert _{\infty }\geq 1\ .
\end{eqnarray}
We have proved above that 
\[( zI-\widehat{\Gamma}_{n}) ^{-1}( \widehat{\Gamma}_{n}-\Gamma
) ( zI-\Gamma ) ^{-1} =(zI-\Gamma)^{-1/2} S_n(z)T_n(z)(zI-\Gamma)^{-1/2}dz\ ,\]
where $S_n(z)=(I-T_n(z))^{-1}$ is well defined for any $z\in\mathcal{B}'_{j^{\ast
}}$. By a straightforward induction, we get for any positive integer $p$
\begin{eqnarray*}
 \lefteqn{\int_{\mathcal{B}'_{j^{\ast
}}} \left( zI-\widehat{\Gamma}_{n}\right) ^{-1}\left( \widehat{\Gamma}_{n}-\Gamma
\right) \left( zI-\Gamma \right) ^{-1}dz
= \sum_{k=1}^{p}\int_{\mathcal{B}'_{j^{\ast}}}(zI-\Gamma)^{-1/2} T_n^k(z)(zI-\Gamma)^{-1/2}dz}&& \\ 
&\hspace{5cm}&+ \int_{\mathcal{B}'_{j^{\ast}}}(zI-\Gamma)^{-1/2}  S_n(z)T_n^p(z)(zI-\Gamma)^{-1/2}dz \ .
\end{eqnarray*}
Observe that each integral $\int_{\mathcal{B}'_{j^{\ast}}}(zI-\Gamma)^{-1/2} T_n^k(z)(zI-\Gamma)^{-1/2}dz$ is zero since the operator $(zI-\Gamma)^{-1/2}$ has no pole inside $\mathcal{B}'_{j^{\ast}}$. Assume  that the event $\mathcal{E}'_{n}$ does not hold. Then, we can bound $\|S_n(z)\|_{\infty}$ by $(1-\|T_n(z)\|_{\infty})^{-1}$ as above. As a consequence, we obtain  that for any positive integer $p$ ,
\begin{eqnarray*}
 \lefteqn{\frac{1}{2\pi}\left\Vert\int_{\mathcal{B}'_{j^{\ast
}}} \left( zI-\widehat{\Gamma}_{n}\right) ^{-1}\left( \widehat{\Gamma}_{n}-\Gamma
\right) \left( zI-\Gamma \right) ^{-1}dz\right\Vert _{\infty}}\\ &\hspace{2cm}&\leq  \frac{1}{2\pi}\int_{\mathcal{B}'_{j^{\ast}}}\left\Vert \left( zI-\Gamma \right) ^{-1/2}\right\Vert _{\infty
}^{2}\frac{\|T_n(z)\|^p_{\infty}}{1-\|T_n(z)\|_{\infty}} dz\leq  \frac{\delta'_j-\delta_j}{2^{p}\delta_j}\ .
\end{eqnarray*}
Taking $p$ large enough in this last upper bound contradicts (\ref{majoration_norme_op}). Thus, $\left(\mathcal{D}_3\setminus \mathcal{D}_1\right) \cap \overline{\mathcal{E}}'_n=\emptyset$, which allows us to conclude.
\end{proof}

\begin{proof}[Proof of Lemma \ref{lemma01}]
The first bound is a straightforward consequence of Lemma \ref{perturb1} since $\mathcal{E}_n=\cup_{j\in\mathcal{K}_n}\mathcal{E}_{j,n}$. The second bound proceeds from the same approach as Lemma \ref{perturb1}.

Let us turn to the third bound. By Weyl's theorem, (e.g. Theorem 4.3.1 in \cite{horn}), we have $|\widehat{\lambda}_1-\lambda_1|\leq \|\widehat{\Gamma}_n-\Gamma\|_{\infty}$ so that 
\begin{eqnarray*}
 \mathbb{P}\left[\widehat{\lambda}_1\geq \frac{3\lambda_1-\lambda_2}{2}\right]&\leq& \mathbb{P}\left[|\widehat{\lambda}_1-\lambda_1|\geq \frac{\lambda_1-\lambda_2}{2}\right]\leq \mathbb{P}\left[\|\widehat{\Gamma}_n-\Gamma\|_{\infty}\geq \frac{\lambda_1-\lambda_2}{2}\right]\\
&\leq &\mathbb{P}\left[\|\widehat{\Gamma}_n-\Gamma\|_{HS}\geq \frac{\lambda_1-\lambda_2}{2}\right]\leq \frac{4}{(\lambda_1-\lambda_2)^2}\mathbb{E}\left[\|\widehat{\Gamma}_n-\Gamma\|_{HS}^2\right]\ .
\end{eqnarray*}
We have 
\[\mathbb{E}\left[\|\widehat{\Gamma}_n-\Gamma\|_{HS}^2\right]=\sum_{k,l=1}^{\infty}\mathbb{E}\left[\langle (\widehat{\Gamma}_n-\Gamma)V_k,V_l\rangle^2\right]\leq \frac{C}{n}\left(\sum_{k=1}^{\infty}\lambda_k\right)^2\ ,\]
by Assumption {\bf B.1}. By Assumption {\bf B.2},  $2\lambda_2\leq \lambda_1$. Applying  Lemma \ref{DomEigen}, we get
\begin{eqnarray*}
\mathbb{P}\left[\widehat{\lambda}_1\geq \frac{3\lambda_1-\lambda_2}{2}\right]&\leq& \frac{C}{n}\left(\frac{\sum_{k=1}^{\infty}\lambda_k}{\lambda_1}\right)^2\leq \frac{C(\gamma)}{n}\ .
\end{eqnarray*}

\end{proof}

\section{Proofs of technical details}\label{section_technique}

\begin{proof}[Proof of Lemma \ref{lemma_3}]
 
We have $\|{\bf Y}- \widehat{\Pin}_k {\bf Y}\|_n^2=\|{\bf Y}\|_n^2 - \|\widehat{\Pin}_k {\bf Y}\|_n^2$.
By the Central limit Theorem, the classical Berry-Esseen inequality, and a classical deviation inequality of $\chi^2$ random variables (e.g. Lemma 1 in \cite{Laurent00}), we get
\begin{eqnarray}
\mathbb{P}\left[\left|\frac{\|{\bf Y}\|_n^2}{n\sigma^2} -1\right|\geq 2\sqrt{\frac{\log(1/x)}{n}}+ 2\frac{\log(1/x)}{n} \right]\leq
2x+
C\frac{\mathbb{E}(|\epsilon|^3)}{\sigma^3\sqrt{n}}\ ,\label{majoration_y}
\end{eqnarray}
for any $x>0$.
Let us compute the expectation of $\|\widehat{\Pin}_k  {\bf Y}\|_n^2$. 
\begin{eqnarray*}
\mathbb{E}\left[\|\widehat{\Pin}_k  {\bf Y}\|_n^2\right]
& = & \mathbb{E}\left[\mathbb{E}\left\{\|\widehat{\Pin}_{k}{\bf
Y}\|_n^2|{\bf X}\right\}\right] = \mathbb{E}\left[\mathbb{E}\left[tr[{\bf Y}^*\widehat{\Pin}_{k}{\bf
Y}]|{\bf X}\right]\right]\nonumber\\
&= & \sigma^2 \mathbb{E}\left[tr\left[\widehat{\Pin}_{k}\right]\right]\leq \sigma^2 k
\end{eqnarray*}
Applying Markov inequality to $\|\widehat{\Pin}_k  {\bf Y}\|_n^2$ and gathering this deviation inequality with (\ref{majoration_y}), we conclude that  
\begin{eqnarray*}
 \mathbb{P}\left[\left|\frac{\|{\bf Y}-\widehat{\Pin}_k{\bf Y}\|_n^2}{n\sigma^2} -
1\right|\geq \frac{ k\log^2 (n)}{n}+ 8\sqrt{\frac{\log\log n}{n}}\right]\leq \frac{3}{\log^2(n)}+\frac{C}{\sqrt{n}}\ , 
\end{eqnarray*}
uniformly over all $k\leq \bar{k}_n$.

\end{proof}

\begin{proof}[Details of the proof of Theorem \ref{thrm_power_KL_fixed_multiple}]
Here, we provide  some details on the comparison between the lower bound \eqref{eq:lower_phi_k_alternative} and the quantile \eqref{eq:quantile}.
By Assumption ${\bf B'.3}$, we derive that $\phi_k({\bf Y},{\bf X}) - k\bar{\F}^{-1}_{k,n-k}(\alpha/|\mathcal{K}_n|)$ is positive with probability larger than $1-3\beta/4-C(\gamma)\log^{-1}(n)$ if 
\begin{eqnarray*}
\lefteqn{C'_1n\|\Gamma_k^{1/2}\theta\|^2 - C'_2\sigma^2\left(\sqrt{k\log(1/\beta)}+ \log(1/\beta)\right)-\frac{2n}{\log(n)}\|(\Gamma^{1/2}-\Gamma_k^{1/2})\theta\|^2}&&\\ &\geq& C_3\sigma^2 \left[\sqrt{k\log\left(\frac{|\mathcal{K}_n|}{\alpha}\right)}+\log\left(\frac{|\mathcal{K}_n|}{\alpha}\right)+ \frac{k^2}{n}+ k\sqrt{\frac{\log(n)}{n}}\right]\\ & & +\frac{C_4}{\beta}\|\Gamma^{1/2}\theta\| ^2 \left[k\vee \log\left(\frac{|\mathcal{K}_n|}{\alpha}\right)\right]\ .
\end{eqnarray*}
Since $\log(|\mathcal{K}_n|/\alpha)\leq 2\sqrt{n}$, $k\leq n^{1/4}$ and $\beta\geq C(\gamma)/\log(n)$, we derive that for $n$ larger than a numerical quantity, $\phi_k({\bf Y},{\bf X}) - k\bar{\F}^{-1}_{k,n-k}(\alpha/|\mathcal{K}_n|)$ is positive with probability larger than $1-3\beta/4-C(\gamma)\log^{-1}(n)$ if 
\begin{eqnarray*}
 \|\Gamma^{1/2}\theta\| ^2\geq C_1\|(\Gamma^{1/2}-\Gamma_k^{1/2})\theta\| ^2 + \sigma^2\frac{C_2}{n}\left[ \sqrt{k\log\left(\frac{|\mathcal{K}_n|}{\alpha\beta}\right)}+\log\left(\frac{|\mathcal{K}_n|}{\beta\alpha}\right)\right]\ .
\end{eqnarray*}

\end{proof}

\begin{proof}[Proof of Lemma \ref{lemme_principal_puissance}]
We have shown in the proof of Lemma \ref{lemma_2} that 
\begin{equation*}
 \mathbb{E}\left[\left\Vert \sqrt{n}\left(  \widehat{A}%
_{k}-A_{k}\right)  \Delta_{n,1}\right\Vert^{2}\mathbf{1}_{\overline{\mathcal{A}}_n}\right]\leq C(\gamma)\left[\frac{\bar{k}_n^3\log^2(n)}{n} + \frac{\bar{k}_n}{\sqrt{n}}\right] \ ,
\end{equation*}
Gathering this bound with Markov inequality and Assumptions ${\bf B'.3}$ allows us to derive the second lower bound of Lemma \ref{lemme_principal_puissance}. 
Focusing on the first bound, we shall prove the following stronger result. For any $x>0$, $k\leq \bar{k}_n$ and any $n\geq 1$, 
\begin{eqnarray}\nonumber
\lefteqn{ \mathbb{P}\left[\|\sqrt{n}(\widehat{A}_k-A_k)\widehat{\Gamma}_n\theta\| \geq x \right]\leq 
 \mathbb{P}[\mathcal{A}_n]+ \frac{C'}{\log(n)}}&& \nonumber\\
 &\hspace{1cm}&+ C(\gamma)\frac{n\log(n)}{x^2}\|\Gamma^{1/2}\theta\| ^2 \left[\frac{k^{3} \log^2 (k\vee e)}{n}\vee \frac{k}{\sqrt{n}}\vee \frac{\bar{k}^{5/2}_n\log(\bar{k}_n\vee e)}{n}\right]\ .\label{eq_1_principal}
\end{eqnarray}
If we take $x=\|\Gamma^{1/2}\theta\|\sqrt{n}/(k^{1/4}\log(n))$ in this inequality and if we combine it with Lemma \ref{lemme_convergence_valeurs_propres_empirique} and  Assumption ${\bf B'.3}$, we recover the conclusion of 
Lemma \ref{lemme_principal_puissance}.\\

\medskip

Define the event $\mathcal{U}_n:=\{\|\widehat{\Gamma}_n^{1/2}\theta\|^2>\log(n)\|\Gamma^{1/2}\theta\| ^2\}$. Since 
\[\mathbb{E}\left[\|\widehat{\Gamma}_n^{1/2}\theta\|^2\right]=\frac{1}{n}\mathbb{E}\left[\sum_{i=1}^n \langle  {\bf X}_{i},\theta
\rangle  ^{2}\right]=\|\Gamma^{1/2}\theta\| ^2\ ,\]
we derive
\begin{equation}\label{eq_majoration_premier_terme_puissance'}
 \mathbb{P}\left[\mathcal{U}_n\right] \leq \frac{1}{\log(n)}\ . 
\end{equation}
We  bound  $\mathbb{P}[  \Vert (
\widehat{A}_{k}-A_{k})  \widehat{\Gamma}_n\theta\Vert  \geq x]$ as follows 
\begin{eqnarray}
 \mathbb{P}\left[  \left\Vert \left(  \widehat{A}_{k}%
-A_{k}\right) \right. \right. \!\! \!\! \!\! \!\!  && \!\! \!\! \!\! \!\!
\left. \left. \widehat{\Gamma}_n\theta\right\Vert  \geq x\right]\leq \mathbb{P}\left[  \left\{\left\Vert \left(  \widehat{A}_{k}%
-A_{k}\right)  \widehat{\Gamma}_n\theta\right\Vert  \geq x\right\}\cap \overline{\mathcal{U}_n}\cap\overline{\mathcal{A}}_n\right] + \mathbb{P}\left[\mathcal{U}_n\cup \mathcal{A}_n\right]\nonumber\\
&\leq &\frac{1}{x^2}\mathbb{E}\left[\left\Vert \left(  \widehat{A}_{k}%
-A_{k}\right)  \widehat{\Gamma}_{n}\theta\right\Vert  ^{2}\mathbf{1}_{\overline{\mathcal{U}_n}\cap\overline{\mathcal{A}}_n}\right]+ \mathbb{P}\left[\mathcal{U}_n\cup\mathcal{A}_n\right]\nonumber\\ 
 &\leq&\frac{1}{x^{2}}\mathbb{E}\left[\left\Vert \left(  \widehat{A}_{k}%
-A_{k}\right)  \widehat{\Gamma}_{n}^{1/2}\right\Vert _{HS}^{2}\left\Vert
\widehat{\Gamma}_{n}^{1/2}\theta\right\Vert ^{2} \mathbf{1}_{\overline{\mathcal{U}_n}\cap\overline{\mathcal{A}}_n}\right]+ \mathbb{P}\left[\mathcal{U}_n\cup\mathcal{A}_n\right]\nonumber\\
&\leq & \frac{\log(n)}{x^2}\|\Gamma^{1/2}\theta\| ^2\mathbb{E}\left[\left\Vert \left(  \widehat{A}_{k}%
-A_{k}\right)  \widehat{\Gamma}_{n}^{1/2}\right\Vert _{HS}^{2}\mathbf{1}_{\overline{\mathcal{A}}_n}\right]+ \mathbb{P}\left[\mathcal{U}_n\cup\mathcal{A}_n\right]\ .\nonumber\label{eq_majoration_premier_terme_puissance}
 \end{eqnarray}
 As a consequence, we have to investigate 
\begin{align*}
\left\Vert \left(  \widehat{A}_{k}-A_{k}\right)  \widehat{\Gamma}_{n}%
^{1/2}\right\Vert _{HS}^{2}  &  =\mathrm{tr}\left(  \widehat{A}_{k}%
-A_{k}\right)  \widehat{\Gamma}_{n}\left(  \widehat{A}_{k}-A_{k}\right) \\
&  =\mathrm{tr}\widehat{A}_{k}\widehat{\Gamma}_{n}\widehat{A}_{k}%
-\mathrm{tr}\widehat{A}_{k}\widehat{\Gamma}_{n}A_{k}-\mathrm{tr}A_{k}%
\widehat{\Gamma}_{n}\widehat{A}_{k}+\mathrm{tr}A_{k}\widehat{\Gamma}_{n}A_{k}\ .%
\end{align*}
Arguing as in the proof of Lemma \ref{lemma_2}, we take the expectation
\begin{eqnarray*}
\lefteqn{\mathbb{E}\left[\left\Vert \left(  \widehat{A}_{k}-A_{k}\right)  \widehat{\Gamma
}_{n}^{1/2}\right\Vert _{HS}^{2}\mathbf{1}_{\overline{\mathcal{A}}_n}\right] } &&\\ &  =&\mathbb{E}\left[\mathrm{tr}\widehat{\Pi}%
_{k}\mathbf{1}_{\overline{\mathcal{A}}_n}\right]+\mathbb{E}\left[\mathrm{tr}\Pi_k\mathbf{1}_{\overline{\mathcal{A}}_n}\right]+\mathbb{E}\left[\mathrm{tr}\left(A_k(\widehat{\Gamma}_n-\Gamma)A_k\right)\mathbf{1}_{\overline{\mathcal{A}}_n}\right] \\
& & - \; 2\mathbb{E}\left[\mathrm{tr}\widehat{\Gamma}_{n,k}%
^{1/2}\Gamma_{k}^{-1/2}\mathbf{1}_{\overline{\mathcal{A}}_n}\right]\\
&  \leq & 2\mathbb{E}\left[ \left\{ k-\mathrm{tr}\widehat{\Gamma}_{n,k}^{1/2}\Gamma
_{k}^{-1/2}\right\}\mathbf{1}_{\overline{\mathcal{A}}_n}\right]+ \sqrt{\mathbb{E}\left[\mathrm{tr}^2\left(A_k(\widehat{\Gamma}_n-\Gamma)A_k\right)\right]}\sqrt{\mathbb{P}[\mathcal{A}_n]}\\ &\leq &2\mathbb{E}\left[\mathrm{tr}\left\{  \Gamma_{k}^{-1/2}\left(  \Gamma_{k}%
^{1/2}-\widehat{\Gamma}_{n,k}^{1/2}\right)\right\}\mathbf{1}_{\overline{\mathcal{A}}_n}  \right]+ \sqrt{\mathbb{E}\left[\mathrm{tr}^2\left(A_k(\widehat{\Gamma}_n-\Gamma)A_k\right)\right]}\sqrt{\mathbb{P}[\mathcal{A}_n]}\ .
\end{eqnarray*}
These expectations have already been upper bounded in (\ref{S2}) and (\ref{eq_fin_esp}).
Thus,  we derive
\[
\mathbb{E}\left[\left\Vert \left(  \widehat{A}_{k}-A_{k}\right)  \widehat{\Gamma
}_{n}^{1/2}\right\Vert _{HS}^{2}\mathbf{1}_{\overline{\mathcal{A}}_n}\right]\leq C(\gamma)\left[\frac{k^{3} \log^2 (k\vee e)}{n}\vee \frac{k}{\sqrt{n}}\vee \frac{\bar{k}^{5/2}_n\log(\bar{k}_n\vee e)}{n}\right]%
\ . \] 
Gathering this last bound with (\ref{eq_majoration_premier_terme_puissance'}) and (\ref{eq_majoration_premier_terme_puissance}) allows us derive the desired inequality 
\eqref{eq_1_principal}.
\end{proof}

\begin{proof}[Proof of Lemma \ref{lemma_A3}]
Observe that $\|{\bf Y}- \widehat{\Pin}_k{\bf Y}\|_n^2\leq \|{\bf Y}\|_n^2= \|\boldsymbol{\epsilon}\|_n^2+ 2\sum_{i=1}^n\boldsymbol{\epsilon}_i\langle {\bf X}_i,\theta\rangle + \sum_{i=1}^n\langle {\bf X}_i,\theta\rangle^2$. By Assumption ${\bf B.1}$ and Tchebychev inequality, $\|\boldsymbol{\epsilon}\|_n^2\leq \sigma^2 (n+C\sqrt{n\log(n)})$ with probability larger than $1-1/\log(n)$. Tchebychev inequality also tells us that $\sum_{i=1}^n\epsilon_i\langle {\bf X}_i,\theta\rangle\leq \sqrt{n\log(n)}(\sigma^2+\|\Gamma^{1/2}\theta\|^2)$ with probability larger than $1-1/\log(n)$. Furthermore, we apply Markov inequality to derive that $\sum_{i=1}^n\langle {\bf X}_i,\theta\rangle^2\leq 4n\|\Gamma^{1/2}\theta\|^2/\beta$ with probability larger than $1-\beta/4$.
 Since $k\leq n/2$, we conclude that 
\begin{eqnarray*}
 \frac{\|{\bf Y}-\widehat{\Pin}_k{\bf Y}\|_n^2}{n-k} \leq \sigma^2\left[1+ C\left(\frac{k}{n}+\sqrt{\frac{\log(n)}{n}}\right)\right]+ C'\|\Gamma^{1/2}\theta\|^2/\beta\ ,
\end{eqnarray*}
with probability larger than $1-2/\log(n)-\beta/4$.
\end{proof}

\begin{proof}[Proof of Lemma \ref{lemma_bidouillage_proba}]
Define $t= 8\sqrt{\frac{\log\log n}{n}}+\frac{k\log^2 (n)}{n}+\frac{1}{k\log^2(n)}$.
First, we use the following bound that will be proved at the end of the proof:
\begin{equation}\label{equation_controle_quantile}
k\bar{\F}_{k,n-k}^{-1}\left(\alpha/|\mathcal{K}_n|\right)\geq \bar{\chi}_k^{-1}\left(\frac{\alpha}{|\mathcal{K}_n|}+\frac{1}{n}\right)/\left(1+4\sqrt{\frac{\log(n)}{n}}\right)\ .
\end{equation}
As a consequence, we have 
\begin{eqnarray*}
 \bar{\chi}_k\left[k\left(1-t\right)\bar{\F}_{k,n-k}^{-1}\left(\frac{\alpha}{|\mathcal{K}_n|}\right)\right]\leq \bar{\chi}_k\left[\frac{\left(1-t\right)\bar{\chi}_k^{-1}\left(\frac{\alpha}{|\mathcal{K}_n|}+\frac{1}{n}\right)}{1+4\sqrt{\frac{\log(n)}{n}}}\right]\ . 
\end{eqnarray*}
Since for any $0<u<1$ and any integer $k\geq 1$, we have $\bar{\chi}_{k}^{-1}(u)\leq k +2\sqrt{\log(1/u)k}+2\log(1/u)$ (e.g. Lemma 1 in~\cite{Laurent00}), it follows from Assumption ${\bf B'.3}$ that 
\begin{eqnarray}\nonumber
 \lefteqn{\frac{\left(1-t\right)\bar{\chi}_k^{-1}\left(\frac{\alpha}{|\mathcal{K}_n|}+\frac{1}{n}\right)}{1+4\sqrt{\frac{\log(n)}{n}}}}&&\nonumber \\ &\geq&  \bar{\chi}_k^{-1}\left(\frac{\alpha}{|\mathcal{K}_n|}+\frac{1}{n}\right) - C(\alpha)\left[k\vee  \log(n)\right]\left[\sqrt{\frac{\log(n)}{n}}\vee \frac{k\log^2 (n)}{n}\vee \frac{1}{k\log^2(n)}\right]\
\nonumber \\ &\geq &  \bar{\chi}_k^{-1}\left(\frac{\alpha}{|\mathcal{K}_n|}+\frac{1}{n}\right) - \frac{C(\alpha)}{\log(n)}
\ .\nonumber
\end{eqnarray}
 Let us note $f_{\chi_k}(x)$ the density at $x$ of a $\chi^2$ random variable with $k$ degrees of freedom. Consider some positive numbers $x$ and $u$ such that $x\geq u$.
\begin{eqnarray*}
 \frac{\bar{\chi}_k(x-u)}{\bar{\chi}_k(x)}&=& 1 +
\frac{\mathbb{P}\left(x-u\leq \chi^2(k)\leq
x\right)}{\bar{\chi}_k(x)}
\leq 1+ u\frac{\sup_{t\in [x-u;x]}f_{\chi_k}(t)}{\bar{\chi}_k(x)} \leq  1+ue^{u/2}\frac{f_{\chi_k}(x)}{\bar{\chi}_k(x)}\ ,
\end{eqnarray*}
since $f_{\chi_k}(x)= x^{k/2-1}e^{-x/2}/[2^{k/2}\Gamma(k/2)]$. By integration
by part, one observes that $f_{\chi_k}(x)/\bar{\chi}_k(x)\leq 1/2$. As a consequence, we have 
$\bar{\chi}_k(x-u)\leq \bar{\chi}_k(x)[1+u/2e^{u/2}]$ for any $u\leq x$.  This upper bound also holds when $u> x$.
\begin{eqnarray*}
\bar{\chi}_k\left[\bar{\chi}_k^{-1}\left(\frac{\alpha}{|\mathcal{K}_n|}+\frac{1}{n}\right)- \frac{C(\alpha)}{\log(n)}\right]&\leq& \bar{\chi}_k\left[\bar{\chi}_k^{-1}\left(\frac{\alpha}{|\mathcal{K}_n|}+\frac{1}{n}\right)\right] \left[1+\frac{C_2(\alpha)}{\log(n)}\right]\\ &\leq&   \frac{\alpha}{|\mathcal{K}_n|} \left(1+\frac{C_3(\alpha)}{\log(n)}\right)\ ,
\end{eqnarray*}
which allows us to derive the desired result.

To finish the proof, we need to prove (\ref{equation_controle_quantile}). Let $X_k$ and $X_{n-k}$ respectively denote two independent random variables that follow a $\chi^2$ distribution with $k$ and $n-k$ degrees of freedom. Moreover we define $F_{k,n-k}$ as $X_k(n-k)/(X_{n-k}k)$. Since $\bar{\chi}_{k}^{-1}(u)\leq k +2\sqrt{\log(1/u)k}+2\log(1/u)$, we have 
\begin{eqnarray*}
\frac{\alpha}{|\mathcal{K}_n|}&= &  \mathbb{P}\left[X_k\geq \bar{\chi}_k^{-1}\left(\frac{\alpha}{|\mathcal{K}_n|}+\frac{1}{n}\right)\right]-\frac{1}{n}\\ &\leq& \mathbb{P}\left[kF_{k,n-k}\geq \frac{\bar{\chi}_k^{-1}\left(\frac{\alpha}{|\mathcal{K}_n|}\right)}{1+4\sqrt{\frac{\log(n)}{n}}}\right]+ \mathbb{P}\left[\frac{X_{n-k}}{n-k}\geq 1+4\sqrt{\frac{\log(n)}{n}}\right]- \frac{1}{n}\\ &\leq& \mathbb{P}\left[kF_{k,n-k}\geq \frac{\bar{\chi}_k^{-1}\left(\frac{\alpha}{|\mathcal{K}_n|}\right)}{1+4\sqrt{\frac{\log(n)}{n}}}\right]\ ,
\end{eqnarray*}
since $4\sqrt{\log(n)/n}\geq 2\sqrt{\log(n)/(n-k)}+2\log(n)/(n-k)$ for $k\leq n/2$ and $n$ large enough. We
conclude  that $k\bar{\F}_{k,n-k}^{-1}\left(\alpha/|\mathcal{K}_n|\right)\geq \bar{\chi}_k^{-1}\left(\frac{\alpha}{|\mathcal{K}_n|}+\frac{1}{n}\right)/\left(1+4\sqrt{\frac{\log(n)}{n}}\right)$ for $k\leq n/2$ and $n$ large enough.
\end{proof}

\begin{proof}[Proof of Lemma \ref{lemma_bidouillage_proba2}]
 Arguing as above, we get 
\[k\bar{\F}_{k,n-k}^{-1}\left(\alpha \right)\geq \bar{\chi}_k^{-1}\left(\alpha+\frac{1}{n}\right)/\left(1+4\sqrt{\frac{\log(n)}{n}}\right)\ .\]
Applying the inequality $\bar{\chi}_{k}^{-1}(u)\leq k +2\sqrt{\log(1/u)k}+2\log(1/u)$ for any $0<u<1$ and Condition ${\bf B.3}$, we get 
\begin{eqnarray*}
\bar{\chi}_k\left[k\left(1-t\right)\bar{\F}_{k,n-k}^{-1}\left(\alpha\right)\right]&\leq & \bar{\chi}_k \left[\bar{\chi}_k^{-1}\left(\alpha+\frac{1}{n}\right)\frac{1- t}{1+4\sqrt{\log(n)/n}}\right]\\
&\leq &\bar{\chi}_k \left[\bar{\chi}_k^{-1}\left(\alpha+\frac{1}{n}\right)- \frac{C(\alpha)}{\log(n)}\right]\\
&\leq &\alpha\left(1+\frac{C(\alpha)}{\log(n)}\right)\ , 
\end{eqnarray*}
where we have used $\bar{\chi}_k(x-u)\leq \bar{\chi}_k(x)[1+u/2e^{u/2}]$ in the last inequality.

\end{proof}

\begin{proof}[Proof of Lemma \ref{DomEigen}]

Since $\left(  j\lambda_{j}\right)  _{j\in\mathbb{N}}$ is  a decreasing
sequence  $j\lambda_{j}\geq k\lambda_{k}$ for $k>j$. Hence, 
we get \[\sum_{j=1}^{k-1}\frac{\lambda_{j}}{ \lambda_{j}-\lambda_{k}}
\leq k\sum_{j=1}^{k-1}(k-j)^{-1}=k\sum_{j=1}^{k-1}j^{-1}\ .\] Similarly
$\sum_{j=k+1}^{2k}\lambda_{j}/\left(  \lambda_{k}-\lambda_{j}\right)  \leq
k\sum_{j=k+1}^{2k}(j-k)^{-1}=k\sum_{j=1}^{k}j^{-1}$. Now we focus on 
$\sum_{j\geq2k+1}\lambda_{j}/\left(  \lambda_{k}-\lambda_{j}\right) $.
The assumption on the eigenvalues implies that for $j\geq k$,  \[k\log^{1+\gamma}(k\vee 2)\left(  \lambda_{k}-\lambda
_{j}\right)  \geq\left(  j\log^{1+\gamma}j-k\log^{1+\gamma
}(k\vee 2)\right)  \lambda_{j}\ .\] Thus, we get \[\frac{\lambda_{j}}{  \lambda_{k}-\lambda
_{j}}  \leq\left[  \left(  j\log^{1+\gamma}j/k\log^{1+\gamma}(k\vee 2)\right)
-1\right]  ^{-1}\] and 
\begin{align*}
\sum_{j\geq2k+1}\frac{\lambda_{j}}{  \lambda_{k}-\lambda_{j}}   &
\leq\int_{2k}^{+\infty}\left[  \frac{x\log^{1+\gamma}x}{k\log^{1+\gamma
}(k\vee 2)}  -1\right]  ^{-1}dx\ .
\end{align*}
For $x\geq2k$, we have 
\[
\frac{x\log^{1+\gamma}x}{k\log^{1+\gamma}(k\vee 2)}\geq\frac{2k\log^{1+\gamma}%
2k}{k\log^{1+\gamma}(k\vee 2)}\geq2\ ,
\]
so that 
\[
\frac{x\log^{1+\gamma}x}{k\log^{1+\gamma}(k\vee 2)}-1\geq\frac{1}{2}\frac
{x\log^{1+\gamma}x}{k\log^{1+\gamma}(k\vee 2)}\ .
\]
It follows that 
\begin{eqnarray*}
\int_{2k}^{+\infty}\left[  \frac{  x\log^{1+\gamma}x}{k\log^{1+\gamma}(k\vee 2)}
-1\right]  ^{-1}dx & \leq & 2k\log^{1+\gamma}(k\vee 2)\int_{2k}^{+\infty}\frac{dx}{
x\log^{1+\gamma}x}  \\
& \leq & \frac{2k\log^{1+\gamma}(k\vee 2)}{\gamma\log^{\gamma}%
2k}\leq\frac{2k\log (k\vee 2)}{\gamma}\ . %
\end{eqnarray*}
All in all, we conclude that 
\begin{eqnarray*}
\sum_{j=1,j\neq k}^{\infty}\frac{\lambda_{j}}{  |\lambda_{k}-\lambda_{j}|}   &
\leq & \frac{2k\log (k\vee 2)}{\gamma} + 2\sum_{j=1}^k\frac{1}{j} \ .
\end{eqnarray*}

\end{proof}


\end{document}